\documentclass[a4paper, 12pt]{amsart}

\usepackage{amssymb}
\usepackage{amsthm}
\usepackage{amsmath}
\usepackage[mathscr]{eucal}

\usepackage{comment}
\usepackage{graphicx}
\usepackage{stmaryrd}

\usepackage[T1]{fontenc}

\usepackage{longtable}
\usepackage{tabu}
\usepackage{tabularx}
\usepackage{ltablex}

\theoremstyle{plain}
\newtheorem{thm}{Theorem}[section]
\newtheorem{prop}[thm]{Proposition}
\newtheorem{cor}[thm]{Corollary}
\newtheorem{lem}[thm]{Lemma}

\theoremstyle{definition}
\newtheorem{df}{Definition}[section]

\theoremstyle{remark}
\newtheorem{rmk}{Remark}[section]

\AtBeginDocument{%
   \def\MR#1{}
}

\newcommand{\zz}{\mathbb{Z}}

\newcommand{\rr}{\mathbb{R}}

\newcommand{\To}{\Longrightarrow}

\DeclareMathOperator{\card}{Card}

\DeclareMathOperator{\cl}{CL}

\DeclareMathOperator{\met}{Met}

\DeclareMathOperator{\ult}{UMet}

\DeclareMathOperator{\chara}{\chi}

\newcommand{\invs}[1]{\triangledown(#1)}

\newcommand{\hahnsp}[1]{\mathbb{H}(#1)}
\newcommand{\hahnkor}[1]{\mathbb{F}(#1)}
\newcommand{\hahntai}[1]{\mathbb{P}(#1)}

\newcommand{\setdis}{\varrho}

\newcommand{\Arc}[1]{\mathbf{A}(#1)}

\newcommand{\mychu}[1]{\sigma(#1)}
\newcommand{\opop}{\mathrm{op}}

\DeclareMathOperator{\metdis}{\mathcal{D}}

\DeclareMathOperator{\umetdis}{\mathcal{UD}}

\DeclareMathOperator{\metrank}{\mathscr{MG}}

\newcommand{\doublecirc}{{\ooalign{$\bigcirc$\crcr\hss$\bullet$\hss}}}
\newcommand{\mzero}{\doublecirc}

\newcommand{\GGG}[1]{\mathrm{GRP}(#1)}
\newcommand{\FFF}[1]{\mathrm{ORD}(#1)}

\newcommand{\lolo}[1]{\Xi(#1)}

\newcommand{\mainmap}{\Phi}
\newcommand{\mainmapult}{\Upsilon}

\newcommand{\submap}{\Psi}

\newcommand{\subsubmap}{\Theta}
\newcommand{\paramap}{\Sigma}

\newcommand{\comp}[1]{#1^{\#}}

\newcommand{\abs}{\mathrm{abs}}

\newcommand{\myfilter}[1]{\mathcal{#1}}

\newcommand{\mysub}{\subseteq}
\newcommand{\mysup}{\supseteq}

\newcommand{\cova}{\lambda}

\newcommand{\eemap}{E}
\newcommand{\eemapsec}{W}
\newcommand{\aimap}{I}

\newcommand{\opmap}{O}

\newcommand{\concon}{\tau}

\newcommand{\stst}[2]{#1[#2]}

\newcommand{\stars}[1]{#1^{\bigstar}}
\newcommand{\oposi}[1]{\opop(#1)}

\newcommand{\myomega}{\omega_{0}}

\newcommand{\myBmap}{\mathfrak{B}}

\newcommand{\myVnbd}{\mathscr{V}}

\DeclareMathOperator{\tdim}{\dim_{T}}
\newcommand{\orddis}{M}

\newcommand{\metabs}{D[\abs]}
\newcommand{\metcova}{D[\cova_{G}]}
\makeatletter
\DeclareRobustCommand{\ggeq}{\mathrel{\mathpalette\gglleq@\gg}}
\DeclareRobustCommand{\lleq}{\mathrel{\mathpalette\gglleq@\ll}}

\newcommand{\gglleq@}[2]{%
  \begingroup
  \sbox\z@{$\m@th#1#2$}\sbox\tw@{$#1\leq$}%
  \vcenter to \dimexpr\ht\tw@+\dp\tw@{%
    \offinterlineskip
    \hbox{$\m@th#1#2$}%
    \vss
    \vbox to \dimen@{
      \vss
      \hbox to \wd\z@{$\m@th\mspace{-0.5mu}#1{-}\hss{-}\mspace{-0.5mu}$}%
      \kern-1.5\fontdimen8 \gglleq@font{#1} 3
      \kern-\fontdimen22 \gglleq@font{#1} 2
    }
  }%
  \endgroup
}
\newcommand{\gglleq@font}[1]{%
  \ifx#1\displaystyle\textfont\else
  \ifx#1\textstyle\textfont\else
  \ifx#1\scriptstyle\scriptfont\else
  \scriptscriptfont\fi\fi\fi
}

\makeatother

\newcommand{\arel}{\asymp}
\newcommand{\arcle}{\lleq}
\newcommand{\arclele}{\ll}

\renewcommand{\labelenumi}{\textup{(\arabic{enumi})}}

\makeatletter
\@addtoreset{equation}{section}

\makeatother

\begin{document}

\title[Extension]
{
Simultaneous
extensions  of  metrics 
and 
ultrametrics 
of high power
 }

\author[Yoshito Ishiki]
{Yoshito Ishiki}
\address[Yoshito Ishiki]
{\endgraf
Photonics Control Technology Team
\endgraf
RIKEN Center for Advanced Photonics
\endgraf
2-1 Hirasawa, Wako, Saitama 351-0198, Japan}
\email{yoshito.ishiki@riken.jp}

\subjclass[2020]{Primary 54E99, Secondary  54C15, 54C20, 12J25, 26E30}
\keywords{
$\omega_{\mu}$-metrics, 
Ultrametrics, 
Hahn fields, 
Retractifiable spaces, 
Extension of metrics, 
Final compactness}

\begin{abstract}

In this paper, 
generalized metrics mean  
metrics taking values in general linearly ordered Abelian groups. 
Using the Hahn fields, 
we first prove that for every generalized metric space, 
if the set of 
the  Archimedean equivalence classes of the range group of the metric has an  infinite
 decreasing sequence, 
 then 
every  non-empty  closed subset of the space is 
a uniform retract of the ambient space. 
Next
 we construct simultaneous extensions 
 of  generalized  metrics and ultrametrics. 
From the existence of extensors of generalized  metrics, 
we characterize the 
final compactness of generalized metrizable
 spaces
using the completeness of generalized metrics. 
\end{abstract}

\maketitle

\setcounter{tocdepth}{1}
\tableofcontents

\section{Introduction}\label{sec:intro}
 Sikorski \cite{sikorski1950remarks} 
introduced the notion of 
$\omega_{\mu}$-metric spaces as spaces 
equipped with metrics taking  values  in 
a 
linearly ordered Abelian group, 
which  is an example of 
 topological spaces of high power ($\omega_{\mu}$-additive spaces). 
There are many developments of research 
on $\omega_{\mu}$-metric ($\omega_{\mu}$-metrizable) spaces concerning 
different or  common properties between 
ordinary metrics and generalized metrics. 
For instance, 
Nyikos--Reichel \cite{MR440515} and 
Wang \cite{MR166749} 
provided 
some  $\omega_{\mu}$-metrization theorems. 
Stevenson--Thorn \cite{stevenson1969results}
 proved that 
a topological space is 
$\omega_{\mu}$-metrizable for some $\mu$ if and 
only if the space have a linearly ordered uniformity. 
Juh\'{a}sz \cite{MR195052}, 
 Kucia--Kulpa
\cite{MR326672}, 
Hayes \cite{hayes1973uniformities}, 
and
 Souppouris 
\cite{MR370524} 
independently
proved that all 
$\omega_{\mu}$-metrizable spaces 
are paracompact. 
Di Concilio--Guadagni 
\cite{MR3542043} 
investigated 
$\omega_{\mu}$-metrizable spaces on which 
all continuous maps are uniformly continuous. 
Di Concilio--Guadagni  \cite{MR3721339} and  
Artico--Marconi--Pelant 
\cite{MR1419403}
discussed 
 hyperspaces of 
$\omega_{\mu}$-metric spaces. 
Hung \cite{MR314008} discussed the amalgamation property of 
generalized metric, 
which property  is related to the construction of 
of the  Urysohn universal space. 
Comicheo 
\cite{MR3946544}
proved 
a generalized open mapping theorem for 
linear  spaces equipped with norms taking values in general 
 ordered sets.  

In the present paper, 
using the Hahn fields, 
we first prove that for every generalized metric space, 
if the set 
of the Archimedean equivalence classes of the range group of the metric has an infinite decreasing sequence, 
then 
every  non-empty  closed subset of the space is 
a uniform retract of the ambient space (see Theorem \ref{thm:retract}). 
This is a generalization of  van Douwen's theorem \cite[Theorem 7 in Chapter 3]{vD1975}  and Brodskiy--Dydak--Higes--Mitra's theorem 
\cite[Theorem 2.9]{brodskiy2007dimension}. 
Next
 we construct simultaneous extensions
 of  ultrametrics and generalized  metrics
  (see Theorems \ref{thm:extensor} and \ref{thm:extensorult}). 
 This is an analogue  of 
Nguyen Van Khue and Nguyen To Nhu's theorem  \cite{NN1981} on simultaneous extensions of metrics, and a generalization of the author's extension theorem  \cite{Ishiki2021ultra} of ultrametrics. 
From the  existence of  extensions of generalized metrics, 
we characterize the 
final compactness of generalized metrizable
 spaces
using the completeness of generalized metrics 
 (see Theorem \ref{thm:kappa-compact}). 
This is an analogue of Niemytzki--Tychonoff's theorem 
\cite{NT1928}.


Before stating precisely our main results, 
we introduce some  notions and notations.

\subsection{Basic definitions}\label{subsec:basicnotations}
Let $L$ be a set. 
A binary relation $R$ on $L$ is said to be 
a
\emph{linear order} if 
it is reflexive, 
transitive, 
antisymmetric
and 
it 
satisfies 
$a\mathbin{R}b$ or 
$b\mathbin{R}a$ 
for all 
$a, b\in L$
(see,  for example, \cite[Chapter 1]{W1970}). 
In this case, the relation $R$ is often  symbolically
represented as $\le_{L}$, 
and we write 
$a<_{L}b$ 
if $a\le_{L}b$ and $a\neq b$. 
We denote by $x\ge_{L}y$ if $y\le_{L} x$. 
The pair $(L, \le_{L})$ 
is call a \emph{linearly ordered set}. 
By abuse if notation, 
we simply denote by 
$L$ the linearly ordered set $(L, \le_{L})$. 
In this paper, 
we often denote by    ``$\le$'' and ``$<$'' the orders
 ``$\le_{L}$'' and ``$<_{L}$'', respectively,  
 when no confusion can arise. 

In this paper, 
for an Abelian group $G$, 
we always denote by 
$+$ 
the group  operation on $G$, 
and 
we denote by $0_{G}$ its zero element. 
A pair $(G, \le_{G})$ of an Abelian group $G$ and a 
linear  order $\le_{G}$ on $G$
is said to be a 
\emph{linearly ordered Abelian group} if 
for all $a, b, c\in G$,  the inequality $a\le_{G} b$ implies that $a+c\le_{G} b+c$. 
By abuse of notation, 
we  simply  denote by
$G$  the linearly ordered Abelian group $(G, \le_{G})$. 
For example, the real numbers $\rr$ and  the integers $\zz$
are linearly ordered Abelian groups.
For simplicity of notation, 
we often write  $0$ instead  of $0_{G}$. 
For a linearly ordered Abelian group $G$, and 
for  $a\in G$, 
we denote by $G_{\ge a}$ (resp.~$G_{>a}$) the set of all $x\in G$ satisfying $a\le x$ (resp.~$a<x$).

Let $G$ be a linearly ordered Abelian group. 
For $x, y\in G_{>0}$, 
we denote by 
$x\ll y$ if for all $n\in \zz_{\ge 1}$
 we have 
$n\cdot x<y$. 
We define
a relation $\arel$ on $G_{>0}$ by 
  $x\arel y$ if and only if 
there exist integers $n, m\in \zz_{\ge  1}$
such that $y\le n\cdot x$ and $x\le m\cdot y$. 
Then $\arel$ is an equivalence relation
on $G_{>0}$, and we denote by 
$[x]_{\arel}$ the equivalence class of $x$ by
$\arel$. 
If $x\arel y$, 
we say that $x$ and $y$ are 
\emph{Archimedean 
equivalent to each other}. 
The relation $\arel$ is  called the 
\emph{Archimedean 
equivalence} on $G$. 
We denote by 
$\Arc{G}$ the quotient set of $G_{>0}$ by 
$\arel$. 
If $\card(\Arc{G})=1$, the group $G$ is said to be 
\emph{Archimedean}; otherwise, 
it is said to be  \emph{non-Archimedean}, 
where ``$\card$'' stands for the cardinality. 
For example, the real numbers $\rr$ and 
the integers $\zz$ are Archimedean. 
For $\alpha, \beta\in \Arc{G}$ with 
$\alpha=[x]_{\arel}$ and $\beta=[y]_{\arel}$, 
we write 
$\alpha\arcle \beta$ if 
$x\ll y$ or  $x\arel y$
(For the well-definedness of 
the order $\arcle$, 
see Lemma \ref{lem:welldef}). 
Then $(\Arc{G}, \arcle)$  becomes 
a linearly ordered set. 
Since $\Arc{G}$ plays an important role in this paper, 
we use 
the special symbol $\arcle$ as the order on 
$\Arc{G}$ 
rather than $\le$. 
By abuse of notation, 
we use the same symbol 
the symbol $\arclele$ on $G_{>0}$ as 
the strict order  $\arclele$ on 
$\Arc{G}$
meaning $\alpha\arcle \beta$ and 
$\alpha\neq \beta$. 
This is equivalent to 
$x\arclele y$, where 
$\alpha=[x]_{\arel}$ and 
$\beta=[y]_{\arel}$. 


\begin{df}\label{df:gmetrics}
Let $X$ be a set, and $G$ be a
linearly 
ordered Abelian group. 
We say that 
a map $d\colon X\times X\to G$ is 
a \emph{$G$-metric} if the following conditions are satisfied:
\begin{enumerate}
\renewcommand{\labelenumi}{(M\arabic{enumi})}
\item\label{item:m0} 
for all $x, y\in X$, the equality 
$d(x, y)=0$ implies $x=y$; 
\item\label{item:m1}
for all $x\in X$, we have 
$d(x, x)=0$;

\item\label{item:m2}
for all $x, y\in X$, we have $0\le d(x, y)$;
\item\label{item:m3}
for all $x, y\in X$, 
we have $d(x, y)=d(y, x)$;
\item\label{item:m4}
for all $x, y, z\in X$, we have 
$d(x, y)\le d(x, z)+d(z, y)$.
\end{enumerate}
The condition (M\ref{item:m4}) is called 
the \emph{triangle inequality}. 
For a $G$-metric $d$ on $X$, 
the topology on $X$ induced from $d$ is 
defined as the topology generated from 
open balls of $d$. 
\end{df}

For a topological  space $X$, 
we denote by $\met(X; G)$ the set of all 
$G$-metrics that generate the same topology 
of 
$X$. 
A topological space $X$ is said to be 
\emph{$G$-metrizable} if 
$\met(X; G)\neq \emptyset$. 
A topological  space is metrizable in the ordinary sense
if and only if it is  $\rr$-metrizable. 

\begin{df}\label{df:sultrametrics}
We say that a linearly  ordered set $S$ is 
\emph{bottomed} if it  has the least element, 
and we denote by $\mzero_{S}$ the least element. 
We often
simply
 denote by  $\mzero$ the least element 
$\mzero_{S}$ of $S$
when no confusion can arise. 
Let $S$ be a bottomed linearly ordered 
set. 
Let $X$ be a set. 
We say that 
a map $d: X\times X\to S$ is 
an \emph{$S$-ultrametric} if the following conditions are satisfied:
\begin{enumerate}
\renewcommand{\labelenumi}{(U\arabic{enumi})}
\item\label{item:u0}
for all $x, y\in X$, the equality 
$d(x, y)=\mzero$ implies $x=y$; 
\item\label{item:u1}
for all $x\in X$, we have 
$d(x, x)=\mzero$;

\item\label{item:u2}
for all $x, y\in X$, we have $\mzero\le d(x, y)$;
\item\label{item:u3}
for all $x, y\in X$, 
we have $d(x, y)=d(y, x)$;
\item\label{item:u4}
for all $x, y, z\in X$, we have 
$d(x, y)\le d(x, z)\lor d(z, y)$, 
where $\lor$ stands for the maximal operator on $S$, 
i.e., $x\lor y=\max\{x, y\}$.
\end{enumerate}
The condition (U\ref{item:u4}) is called 
the \emph{strong triangle inequality}. 
The topology on $X$ induced from $d$ is defined in a 
similar way to $G$-metrics. 
\end{df}

For a topological space $X$, 
we denote by 
$\ult(X; S)$
the set of all $S$-ultrametrics that generate
the same topology of $X$.

We say that a topological space $X$ is 
\emph{$S$-ultrametrizable} if we have 
$\ult(X; S)\neq \emptyset$. 
A topological space is ultrametrizable in the 
ordinary sense if and only if 
it is $\rr_{\ge 0}$-ultrametrizable.













For an ordinal $\mu$, 
the symbol $\omega_{\mu}$ 
stands for 
the $\mu$-th cardinal, namely,  
$\omega_{\mu}$  is the least ordinal 
whose cardinal is $\aleph_{\mu}$. 
Note   that $\myomega$ is  the least infinite cardinal, namely,  $\myomega=\{0, 1, 2, \dots\}$. 
\begin{df}\label{df:character}
For a bottomed linearly ordered set $S$, 
we define the 
\emph{character} 
$\chara(S)$ of 
$S$ as the minimal cardinal 
of all 
$\kappa>0$
such that there exists a map 
$f\colon\kappa+1\to T$ satisfying that 
\begin{enumerate}
\item if $\alpha, \beta<\kappa$ satisfy 
$\alpha<\beta$, 
then $f(\beta)<f(\alpha)$;
\item 
$f(\kappa)=\mzero_{S}$; 
\item 
for every $t\in S$, there exists $\alpha<\kappa$ with $f(\alpha)\le t$. 
\end{enumerate}
Note that $\chara(S)=1$,  or 
 $\chara(S)$ is a regular cardinal. 
 Remark  that 
the character $\chara(S)$ is equal to 
the cofinality of $S\setminus \{\mzero_{S}\}$. 
For example, we have $\chara(\rr_{\ge 0})=\myomega$. 
\end{df}

A topological space is said to be 
\emph{$\omega_{\mu}$-metrizable} 
(in the sense of Sikorski)
if 
there exits a linearly ordered group $G$ such that
$\chara(G_{\ge 0})=\omega_{\mu}$ and 
$\met(X; G)\neq \emptyset$. 
The notion of 
$\omega_{\mu}$-metrizable space was 
 introduced by 
 Sikorski \cite{sikorski1950remarks} as
 a generalization of ordinary metric spaces.


As remarked in 
 \cite{MR440515} and \cite{artico1981some}, 
 if a topological space $X$ is 
 $\omega_{\mu}$-metrizable,
then
$X$ is $S$-ultrametrizable 
for some  linearly ordered set  
$S$ with $\chara(S)=\omega_{\mu}$. 
This statement can be proven by 
considering the Archimedean equivalence class of 
$G$ with $\met(X; G)\neq \emptyset$
(see Lemma \ref{lem:uniformdis} and Proposition \ref{prop:characterization}). 
Focusing on this fact, 
and
based on 
the
Archimedean equivalence on 
linearly ordered Abelian groups, 
we introduce 
the 
notion of 
metrizable gauges in Definition  
\ref{df:metgauge}. 
Before metrizable gauges, 
we  define a one-point extension of 
a linearly ordered set. 

\begin{df}\label{df:oneptextorder}
Let $L$ be a linearly ordered set. 
We  denote by $\lolo{L}$ 
the one-point extension ordered set of $L$
by adding the element $\mzero_{\lolo{L}}\not\in L$ 
to $L$, 
and by defining 
$\mzero_{\lolo{L}}\le x$ for all $x\in L$. 
The set $\lolo{L}$ contains $L$ as a ordered subset. 
Note that $\lolo{L}=\{\mzero_{\lolo{L}}\}\sqcup L$
 and $\mzero_{\lolo{L}}$ is the minimal element 
 of $\lolo{L}$. 
\end{df}

\begin{df}\label{df:metgauge}
Let $X$ be a topological space. 
We say that a cardinal $\kappa$  is  
a \emph{metrizable gauge of $X$}
 if 
there exists a linearly 
 ordered abelian group $G$ such that 
$\chara(\lolo{\Arc{G}})=\kappa$ and 
$\met(X; G)\neq \emptyset$. 
In this case, the cardinal $\kappa$ should be $1$ or 
a regular cardinal. 
We denote by $\metrank(X)$ the set (class) of 
all metrizable gauges of $X$. 
We say that 
a topological  space $X$ \emph{possesses an  infinite metrizable gauge} if
there exists $\kappa\in \metrank(X)$ 
such that $\myomega \le \kappa$.  
\end{df}
\begin{df}
For a cardinal $\kappa$, 
we denote by $\GGG{\kappa}$ the 
class of all linearly ordered Abelian groups 
$G$ such that $\chara(\lolo{\Arc{G}})=\kappa$.
Note that for an uncountable  regular  cardinal $\kappa$, we have  $G\in \GGG{\kappa}$ if and only if 
$\chara(G_{\ge 0})=\kappa$ (see Proposition  \ref{prop:chacha}). 
We also denote by 
$\FFF{\kappa}$
 the class of all bottomed  linearly ordered sets
$S$ such that $\chara(S)=\kappa$.
For example, 
we observe that $\rr\in \GGG{1}$ and 
$\rr_{\ge 0}\in \FFF{\myomega}$. 
Note that a linearly ordered Abelian group 
$G$ is Archimedean if and only if 
$G\in \GGG{1}$. 
\end{df}

We should note  the difference 
between 
the $\omega_{\mu}$-metrizability 
in  the sense 
of Sikorski
and 
the metrizable gauges of the present paper. 
For an uncountable regular cardinal 
$\kappa$, 
there exists $\mu>0$ such that 
$\kappa=\omega_{\mu}$,
and  
a topological  space $X$ is $\omega_{\mu}$-metrizable if and only if $\kappa\in \metrank(X)$. 
The difference arises only  in the countable case. 
Since all ordinary metric and ultrametric spaces are $\myomega$-metrizable, 
we can not distinguish  ultrametrizable  spaces and 
 metrizable  and non-ultrametrizable space using 
the $\omega_{0}$-metrizability. 
In main results of this paper, we shall treat 
common properties on ultrametric spaces and 
$\omega_{\mu}$-metrizable spaces with $\mu>0$.
Using metrizable gauges, 
we unify those cases as the case where a topological space possesses an infinite metrizable gauge, 
and  
we can separate 
 ultrametrizable  spaces and 
  metrizable  and non-ultrametrizable 
  spaces
 (see Proposition \ref{prop:classify}). 
Due to this observation, 
in this paper, 
we mainly use the metrizable gauges 
rather than the $\omega_{\mu}$-metrizability. 

One of the key points of the present paper is 
to utilize 
concepts  on  ordered fields such as 
the Hahn fields,  Hahn's embedding theorem,   
and the Archimedean equivalence,  in
the theory of metrics taking values in general 
linearly ordered Abelian groups. 


\subsection{Main results}\label{subsec:mainresults}
Let $X$ be a topological space. 
A subset $A$ of $X$ is said to be 
a \emph{retract} if there exists a 
continuous map $r\colon X\to A$ such that 
$r(a)=a$ for all $a\in A$. In this case, 
the continuous map $r$ is said to be 
a \emph{retraction}. 
A topological space is said to be 
\emph{retractifiable} if every its
non-empty closed subset is a 
retract of the ambient space. 
This concept was first introduced by 
van Douwen \cite{vD1975}. 
There are some 
 results on retracts of zero-dimensional spaces.

The proofs of 
Theorem 9.1 in 
Dugundji's paper \cite{MR113217} and 
Theorem 4
in 
Kodama's paper  \cite{MR89411}
contain the statement that 
every 
non-empty closed subset of 
an
ultrametrizable space is a retract of 
the ambient space
(see also \cite{dancis1993each}).

Engelking \cite{MR239571} proved that 
if a closed subset $F$ of a metrizable space $X$ 
satisfies that $\tdim(X\setminus F)=0$, then 
there exists a retraction 
$r\colon X\to F$ which is a closed map, 
where $\tdim$ stands for the covering dimension. 

In 1975, 
van Douwen \cite{vD1975} established 
 the 
following results: 
\begin{enumerate}
\item 
All ultrametrizable spaces and 
$\omega_{\mu}$-metrizable spaces for 
$\mu>0$ are retractifiable. 
\item 
Every non-empty closed $G_{\delta}$-subset 
of a strongly zero-dimensional collectionwise normal space
 is a retract of the ambient space. 
 \item For every ordinal $\alpha$, 
 the space $\alpha$ with
 the ordered topology 
 is (hereditarily) retractifiable. 
 \item 
 The Sorgenfrey line is (hereditarily)
 retractifiable. 

\item 
Every locally compact totally disconnected 
orderable space is retractifiable. 
\end{enumerate}

K\k{a}kol--Kubzdela--\'{S}liwa
\cite{MR3035503}
showed that every compact metrizable subspace 
$Y$ of an ultraregular space $X$ is a retract of 
$X$.

For a metric space $(X, d)$, we denote by 
$\exp(X)$ the space of all non-empty compact subsets of $X$ equipped with the Hausdorff distance. 
By Michael's zero-dimensional selection theorem \cite[Theorem 2]{MR1529282}, 
for a complete ultrametric  space $(X, d)$, 
Tymchatyn--Zarichnyi
\cite{tymchatyn2005note} constructed a 
map $R\colon X\times \exp(X)\to X$
such that 
for all $(x, A)\in X\times \exp(X)$, 
we have $R(x, A)\in A$,  and 
such that if 
 $x\in A$, then $R(x, A)=x$. 
Stasyuk--Tymchatyn
\cite{stasyuk2009continuous}
proved the existence of uniformly continuous 
$R\colon X\times \exp(X)\to X$ satisfying 
the conditions mentioned above.

For every (ordinary) ultrametric space $(X, d)$, 
 for every  closed subset $A$ of $X$, 
 for all $\delta\in (1, \infty)$, 
Brodskiy--Dydak--Higes--Mitra \cite{brodskiy2007dimension}  constructed a $\delta$-Lipschitz retraction from $X$ to $A$
with respect to $d$. 

The author \cite{Ishiki2021ultra} used 
the Brodskiy--Dydak--Higes--Mitra 
 theorem to prove the extension theorem 
of ultrametrics. 
Since  
 the multiplication of the real numbers was utilized in the proof in \cite{brodskiy2007dimension}, 
 the author \cite{Ishiki2021ultra}   said that 
an analogue of  the retractifiable  theorem  for 
ultrametrics valued in general bottomed 
linearly  ordered sets
seemed to be not true; however, 
that is false. 
Indeed,  
in  \cite[Theorem 7 in Chapter 3]{vD1975}, 
it was proven that all $\omega_{\mu}$-metrizable spaces
are retractifiable for all $\mu>0$. 
Note that 
the proof of \cite[Theorem 7 in Chapter 3]{vD1975} 
seems not to imply the existence of uniformly continuous retractions. 
In the present paper, 
 we show
 the
 existence of uniformly continuous retractions. 
 Using the construction of the Hahn fields (see Section \ref{sec:pre}), 
every  linearly ordered set can be regarded as 
a subset of positive numbers of 
a
linearly ordered field (see Proposition \ref{prop:ee}), 
and we can apply the argument 
of Brodskiy--Dydak--Higes--Mitra \cite{brodskiy2007dimension} to 
generalized (ultra)metrics. 
The following is our first result:
\begin{thm}\label{thm:retract}
We assume that 
$X$ is  a topological space possessing 
an infinite metrizable gauge. 
Let 
$\kappa$ be a regular cardinal with 
$\kappa\in \metrank(X)$. 
Then, 
 for every $G\in \GGG{\kappa}$, 
 for every $d\in \met(X; G)$,  
 and for every closed non-empty subset
 $A$ of $X$, 
there exists  a uniformly continuous  retraction 
$r\colon X\to A$ with respect to $d$. 
In particular, the space  $X$ is retractifiable. 
\end{thm}
\begin{rmk}
In contrast to 
Brodskiy--Dydak--Higes--Mitra's theorem 
\cite[Theorem 2.9]{brodskiy2007dimension} on the existence of 
$\delta$-Lipschitz retractions for all $\delta\in (1, \infty)$, 
it is worth noting when there exits a $1$-Lipschitz retraction $r\colon X\to A$. 
Artico--Moresco \cite{artico1981some} 
characterize a $1$-Lipschitz retract 
of 
generalized ultrametric spaces using the proximality
(see Theorem \ref{thm:Artico} in this paper, or 
\cite[Theorem 4.6 and Proposition 2.6]{artico1981some}). 
By their theorem, we characterize 
 closed subsets of generalized ultrametric spaces
 (see Corollary \ref{cor:characlosed}) using the proximality and  $1$-Lipschitz maps.
\end{rmk}



We next explain a result on extensors of ultrametrics and 
generalized metrics. 
Hausdorff 
\cite{Ha1930} proved 
that for every $\rr$-metrizable space $X$,  
for every closed subset $A$ of $X$, 
and for every  $d\in \met(A; \rr)$, 
there exists $D\in \met(X; \rr)$
satisfying that  
$D|_{A^2}=d$. 
The author \cite{Ishiki2021ultra}
proved 
an ultrametric version of the Hausdorff metric extension theorem, i.e., 
for every 
 $S\mysub \rr_{\ge 0}$ with $0\in S$  possessing 
a decreasing sequence convergent to $0$, and 
 for every
  $\rr_{\ge 0}$-ultrametrizable space $X$, 
and for every 
   closed subset $A$ of $X$, and for 
every $d\in \ult(A;  S)$, 
there exists $D\in \ult(X; S)$ satisfying that 
$D|_{A^{2}}=d$. 

Let $\metdis_{X}$ stand for the 
supremum metric on $\met(X; \rr)$, 
which can take the value $\infty$. 
For every  metrizable space $X$, 
and for every closed subset $A$ of $X$, 
Nguyen Van Khue and Nguyen To Nhu
\cite{NN1981}
constructed 
 maps 
 $\Phi_1, \Phi_2\colon \met(A; \rr)\to \met(X; \rr)$ such that 
\begin{enumerate}
\item\label{item:NN1}  
the maps $\Phi_1$ and $\Phi_2$ are extensors; namely, for every $d\in \met(A)$, we have 
$\Phi_1(d)|_{A^{2}}=d$ and $\Phi_2(d)|_{A^2}=d$;  

\item\label{item:NN2} 
the map 
$\Phi_1$ is $20$-Lipschitz with respect to the metrics $\mathcal{D}_A$ and $\mathcal{D}_X$; 
\item\label{item:NN3}  
the map
$\Phi_2$ is continuous with respect to the 
topologies of pointwise convergence  on $\met(A; \rr)$ and 
$\met(X; \rr)$;
\item\label{item:NN4} 
 the map $\Phi_2$  preserves orders; namely,  
if   $d, e\in \met(A)$ satisfy  $d(a, b)\le e(a, b)$ for all $a, b\in A$, then    $\Phi_2(d)(x, y)\le \Phi_2(e)(x, y)$ for all $x, y\in X$; 
\item\label{item:NN5} 
 if $X$ is completely metrizable, 
then $\Phi_1$ and $\Phi_2$ map any complete metric in 
$\met(A; \rr)$ into  a complete metric in 
$\met(X; \rr)$.  
\end{enumerate}

Although 
their constructions of $\Phi_{1}$ and $\Phi_{2}$ 
need
 the Dugundji extension theorem
  and 
it is known that an analogue of Dugundji's extension theorem for $\omega_{\mu}$-metric spaces is 
false in general (see \cite{MR1387953}), 
we can prove an analogue of 
Nguyen Van Khue and Nguyen To Nhu's 
theorem for generalized metrics. 
We
shall  construct monotone  extensors
 of  ultrametrics or generalized  metrics 
 (see Theorem \ref{thm:extensor}). 
 
For a compact ultrametrizable $X$, 
Tymchatyn--Zarichnyi
\cite{tymchatyn2005note} 
constructed maps 
from the set of all continuous ultrametrics defined on 
closed subsets of $X$ into the set of all 
continuous ultrametrics defined on $X$, 
which is continuous with respect to 
the Vietoris topology. 
Stasyuk--Tymchatyn
\cite{stasyuk2009continuous}
constructed such a map for 
a complete ultrametrizable space. 
The author does not know whether 
we can constructed a Tymchatyn--Zarichnyi type 
map from $\bigcup_{A\in \exp(X)}\ult(X; \rr_{\ge 0})$
into $\ult(X; \rr_{\ge 0})$ or not.

Extension theorems  of ultrametrics 
 can be considered as special cases of
 extending  a weight on  the edge set of a given graph 
 into an ultrametric on the vertex set 
of the graph. 
Dovgoshey--Martio--Vuorinen \cite{MR3090172}
proved 
theorems extending weights into (pseudo)ultrametrics. 
Dovgoshe\u{\i}--Petrov \cite{MR3135687}
also 
provided  theorems 
on metrization of weighted graphs. 

Let $X$ be a topological space.  
Let $G$ be a linearly ordered Abelian group
and $S$ be a linearly ordered set such that 
$\met(X; G)\neq \emptyset$ and 
$\ult(X; S)\neq \emptyset$. 
Let
$d\in \met(X; G)$ or $d\in \ult(X; S)$. 
We say that the  space 
 $(X, d)$ is \emph{complete} if 
every  Cauchy filter on $(X, d)$ has a limit point
(for the definition of Cauchy filters, see Section \ref{sec:pre}).

Let $d, e\in \met(X; G)$ 
(resp.~$d,  e\in \ult(X; G)$). 
We define $d\lor e\in \met(X; d)$ 
(resp.~$d\lor e\in \ult(X; S)$)
by $(d\lor e)(x, y)=d(x, y)\lor e(x, y)$, where
$\lor$ in the right hand side is the maximal operator on 
$G$ (resp.~$S$). 

We introduce the topology on $\met(X; G)$ as follows: 
For all $\epsilon\in G_{>0}$, 
we denote by  $\myVnbd(d; \epsilon)$ the 
set of all $e\in \met(X; G)$ such that 
for all $x, y\in X$ we have $e(x, y)<d(x, y)\lor \epsilon$ and $d(x, y)<e(x, y)\lor \epsilon$. 
We consider that 
$\met(X; G)$ is equipped with 
 the topology induced from 
$\{\, \myVnbd(d; \epsilon)\mid d\in \met(X; G), \epsilon\in G_{>0}\, \}$.

We introduce 
the notion of characteristic or g-characteristic 
subsets, which is a central concept in the present paper. 
\begin{df}\label{df:charazero}
Let $S$ be a 
bottomed linearly ordered set. 
A subset $T$ of $S$ 
is said to be 
\emph{characteristic} if 
$\mzero_{S}\in T$
and 
for all $s\in S\setminus \{\mzero_{S}\}$, 
there exists $t\in T\setminus \{\mzero_{S}\}$
 such that $t\le s$. 
Let $G$ be a linearly ordered Abelian group. 
A subset $Q$ of $G$ is said to be 
\emph{g-characteristic} if 
$Q$ is a characteristic subset of 
the bottomed linearly ordered  set $G_{\ge 0}$. 
The word ``g-characteristic'' means 
``group-characteristic''. 
\end{df}

The following is our second result: 

\begin{thm}\label{thm:extensor}
We assume that $X$ 
is a topological space possessing 
an infinite metrizable gauge. 
Let $\kappa$ be a regular cardinal
 with 
$\kappa\in \metrank(X)$. 
Let $G\in \GGG{\kappa}$. 
Let $S$ be a g-characteristic subset of 
$G$. 
Let  $A$ be a closed subset of $X$. 
Then
 there exists a map 
$\mainmap\colon\met(A, G)\to \met(X, G)$ such that 

\begin{enumerate}
\renewcommand{\labelenumi}{(A\arabic{enumi})}
\item 
the map
$\mainmap$ is continuous; \label{item:contia1}
\item 
for all $d\in \met(A; G)$, 
we have 
$\mainmap(d)|_{A^{2}}=d$;\label{item:a2}
\item\label{item:a3}
if 
$d_{1}, d_{2}\in \met(X; G)$ satisfy 
$d_{1}(a, b)\le d_{2}(a, b)$ for all $a, b\in A$, 
then for all $x, y\in X$, we have $\mainmap(d_{1})(x, y)\le \mainmap(d_{2})(x, y)$;
\item\label{item:a4}
if $d\in \ult(A; G_{\ge 0})$, 
then $\mainmap(d)\in \ult(X; G_{\ge 0})$;
\item\label{item:a45}
if $d\in \ult(A; S)$, 
then $\mainmap(d)\in \ult(X; S)$;
\item 
for all $d, e\in \met(A; G)$, we have 
$\mainmap(d\lor e)=\mainmap(d)\lor \mainmap(e)$. \label{item:a5}
\end{enumerate}
Moreover, if there exists 
a complete $G$-metric in $\met(X; G)$, 
then 
we can choose $\mainmap$ as a map satisfying 
all the conditions mentioned above and  
\begin{enumerate}
\setcounter{enumi}{6}
\renewcommand{\labelenumi}{(A\arabic{enumi})}
\item 
if $d\in \met(A; G)$ is  complete, 
then so is $\mainmap(d)$.\label{item:a6}
\end{enumerate}
\end{thm}

To state our  next result, 
for a topological space $X$, 
and for $S\mysub \rr_{\ge 0}$ with $0\in S$, 
we define a function 
$\mathcal{UD}_X^S\colon\ult(X, S)^2\to [0, \infty]$
 by assigning $\mathcal{UD}_X^S(d, e)$ to  the infimum of 
 $\epsilon\in S\sqcup \{\infty\}$ such that 
 for all $x, y\in X$ we have 
 \[
 d(x, y)\le e(x, y)\lor \epsilon, 
 \]
 and 
 \[
  e(x, y)\le d(x, y)\lor \epsilon. 
 \]
 Note that 
the function 
$\umetdis_{X}^{S}$
 is an ultrametric on $\ult(X, S)$ valued in $\cl(S)\sqcup \{\infty\}$, 
 where $\cl(S)$ stands for the closure of $S$ in 
 $\rr_{\ge 0}$. 
 The ultrametric $\umetdis_{X}^{S}$
 was utilized  in \cite{Ishiki2021ultra}.

In the case of ordinary  ultrametric spaces, 
the  extensor in Theorem \ref{thm:extensor} becomes an isometric embedding.

 \begin{thm}\label{thm:extensorult}
Let $S$ be a 
characteristic subset of $\rr_{\ge 0}$.  
Let $X$ be an $\rr_{\ge 0}$-ultrametrizable space. 
Then there exists a map 
$\mainmapult\colon \ult(A, S)\to \ult(X, S)$ such that 
\begin{enumerate}
\renewcommand{\labelenumi}{(B\arabic{enumi})}

\item\label{item:b1}
the map $\mainmapult$ is an isometric embedding; 
namely, 
for all $d_{1}, d_{2}\in \ult(A; S)$, 
we have 
\[
\umetdis_{X}^{S}(\mainmapult(d_{1}), \mainmapult(d_{2}))=\umetdis_{A}^{S}(d_{1}, d_{2});
\]
\item\label{item:b2}
for all $d\in \ult(A; S)$, 
we have $\mainmapult(d)|_{A^{2}}=d$;
\item\label{item:b3}
if $d_{1}, d_{2}\in \ult(A; G)$ satisfies 
$d_{1}(a, b)\le d_{2}(a, b)$ for all $a, b\in A$, 
then for all $x, y\in X$, we have 
 $\mainmapult(d_{1})(x, y)\le 
 \mainmapult(d_{2})(x, y)$;
\item\label{item:b4}
for all $d, e\in \ult(A; G)$, we have 
$\mainmapult(d\lor e)=
\mainmapult(d)\lor \mainmapult(e)$.

\end{enumerate}

Moreover, if 
$X$ is completely metrizable, 
then we can choose 
$\mainmapult$ as a map satisfying 
all the conditions mentioned above and 
\begin{enumerate}
\setcounter{enumi}{4}
\renewcommand{\labelenumi}{(B\arabic{enumi})}
\item\label{item:b5}
if $d\in \ult(A; G)$ is  complete, 
then so is $\mainmapult(d)$. 
\end{enumerate}

\end{thm}

Theorems \ref{thm:extensor} and \ref{thm:extensorult} are
 not only  
 new  extension theorems for ultrametrics and generalized metric spaces but also  the   improvement 
of the author's extension theorem of ordinal ultrametrics in \cite{Ishiki2021ultra}. 

\begin{rmk}
In Nguyen Van Khue and Nguyen To Nhu's theorem, 
the $20$-Lipschitz extensor  $\Phi_{1}$ and 
the 
monotone pointwise continuous extensor $\Phi_{2}$
are constructed. 
In Theorem \ref{thm:extensorult} (or Theorem \ref{thm:extensor}), 
we unify these properties into a single extensor, 
namely, 
we construct a
metric  extensor which is 
isometric and monotone. 
\end{rmk}

\begin{rmk}
Let $G$ be a linearly ordered Abelian group. 
We assume that $G$ is non-Archimedean. 
Then, 
we 
can 
take 
$x, y\in G_{>0}$ with  
$x\ll y$. 
Put 
\begin{align}
&E=\{\, z\in G\mid 
\text{for all $n\in \zz_{\ge 1}$, 
we have $nz<y$}\, \}, \\
&F=\{\, z\in G\mid \text{there exists $n\in \zz_{\ge 1}$ such that $y\le nz$}\, \}.
\end{align}
Then, we have 
$F=G\setminus E$, 
and $E\neq \emptyset$
and $F\neq \emptyset$ (in fact, 
$x\in E$ and $y\in F$). 
Moreover, 
we have $a<b$ for all $a\in E$ and 
$b\in F$, 
namely, the set $E$
(or the pair $(E, F)$) is a 
cut in the sense of Dedekind
(see Subsection \ref{subsec:ord}). 
Remark that 
the set $E$ has no  supremum  and 
$F$ has no infimum 
(see \cite[Lemma 2.2]{MR314008}). 
This observation implies that 
if $G$ is non-Archimedean,  
then $G$ is not Dedekind complete, 
and 
$G$ does not have 
the  Dedekind completion which is 
a linearly ordered Abelian group. 
Due to this phenomenon, 
when 
we defined the topology of $\met(X, G)$, 
we used  $\myVnbd(d; \epsilon)$, 
and we did not 
define   a metric on $\met(X; G)$  
such as $\umetdis_{X}^{S}$ 
on $\ult(X; S)$, 
 which is defined using the infimum. 
\end{rmk}


We next explain an application of Theorem \ref{thm:extensor}. 
Niemytzki--Tychonoff \cite{NT1928} proved that 
an $\rr$-metrizable space $X$ is compact if and only if 
all $d\in \met(X; \rr)$ are complete. 
There are many analogues of this characterization for 
various geometric structures. 

Nomizu--Ozeki \cite{NO1961}
proved that a second countable connected 
differentiable manifold is compact if and only if all Riemannian metrics on the manifold are complete. 

The author \cite{Ishiki2021ultra} 
showed that 
an $\rr_{\ge 0}$-ultrametrizable space 
$X$ is compact 
if and only if 
all $d\in \ult(X; \rr_{\ge 0})$ are complete. 
The author \cite{ishiki2021dense} also showed that 
a metrizable space $X$ is finite-dimensional 
and compact 
(resp.~zero-dimensional and compact) if and only if 
the set of all doubling (resp.~uniformly disconnected) metrics is dense $F_{\sigma}$ in $\met(X; \rr)$. 

Dovgoshey--Shcherbak \cite{MR4335845} proved that 
an $\rr_{\ge 0}$-ultrametrizable space $X$ is separable if and only if 
for every  $d\in \ult(X; \rr_{\ge 0})$, the set 
$\{\, d(x, y)\mid x, y\in X\, \}$ is countable, 
and 
they also 
proved that  
an $\rr_{\ge 0}$-ultrametrizable space $X$ is compact if and only if all $d\in \ult(X; \rr_{\ge 0})$ are totally bounded.  

Hausdorff \cite{Ha1930} used the  extension theorem of metrics explained before   to give another proof of 
Niemytzki--Tychonoff's characterization theorem. 
The author's theorem  \cite{Ishiki2021ultra} explained above is proven using Hausdorff's argument.

Let $\kappa$ be a cardinal. 
A topological space is said to be  \emph{finally $\kappa$-compact} 
if 
every  open cover of the space  
has a subcover with cardinal $<\kappa$. 
In this paper, using 
the existence of extensors of generalized metrics
(Theorem \ref{thm:extensor})
 and 
Hausdorff's argument, 
we characterized the 
final compactness of generalized metrizable
 spaces
by the completeness of generalized metrics. 
This is an analogue of Niemytzki--Tychonoff's theorem for generalized metics.

\begin{thm}\label{thm:kappa-compact}

We assume that $X$ 
is  a topological space possessing 
an infinite metrizable gauge. 
Let $\kappa$ be a regular cardinal 
with $\kappa\in \metrank(X)$. 
Then the following 
statements are equivalent to each other. 
\begin{enumerate}
\item\label{item:com1}
The space $X$ is finally $\kappa$-compact. 
\item\label{item:com2}
There exists $G\in \GGG{\kappa}$ such that 
for all $d\in \met(X; G)$, the space 
$(X, d)$ is complete. 
\item\label{item:com3}
For all $G\in\GGG{\kappa}$ and for all 
$d\in \met(X; G)$, the space 
$(X, d)$ is complete. 
\item\label{item:com4}
There exists $S\in \FFF{\kappa}$ such that 
for all $d\in \ult(X; S)$, the space 
$(X, d)$ is complete. 

\item\label{item:com5}
For all $S\in \FFF{\kappa}$ and for all $d\in \ult(X; d)$, 
the space $(X, d)$ is complete. 
\end{enumerate}

\end{thm}

\begin{rmk}
Since 
the final $\omega_{0}$-compactness 
is 
identical with 
 the ordinary  compactness, 
Theorem \ref{thm:kappa-compact} contains 
the author's result \cite[Corollary 1.3]{Ishiki2021ultra}. 
\end{rmk}


The organization of this paper is as follows: 
In Section \ref{sec:pre}, 
we prepare some concepts such as 
the Hahn groups, the Hahn fields, and 
Cauchy filters.  
We also discuss some basic  statements on  
ordered sets, metrizable gauges, 
 and generalized metrics and ultrametrics.
In Section \ref{sec:retraction}, 
we prove Theorem \ref{thm:retract}. 
As a consequence of 
the existence of retractions,  
and  Artico--Moresco's characterization 
(Theorem \ref {thm:Artico}) of proximal subsets of 
generalized 
ultrametric spaces, 
we prove Corollary \ref{cor:characlosed}. 
In Section \ref{sec:extensors}, 
we  first discuss modifications of generalized metrics. 
We construct  a metric vanishing on a given closed subset
(see Proposition \ref{prop:vanishingmetric}), 
and show basic properties of 
 an extension of generalized metrics 
 (see Lemma \ref{lem:metmet} and \ref{lem:completeext}). 
We then prove Theorem \ref{thm:extensor} and 
Theorem \ref{thm:extensorult}. 
In Section \ref{sec:tableofsymbols}, 
we exhibit the table of symbols appearing  in this
paper.


\section{Preliminaries}\label{sec:pre}

In this section, we prepare some basic
statements. 

\subsection{Order structures}\label{subsec:ord}
We first discuss
 order structures such as
the Dedekind completeness, 
the Hahn groups, 
and 
the Hahn fields.

\subsubsection{Basic statements on orders}
In the present  paper, 
we use the set-theoretic representation of 
ordinals. 
For example, 
if $\alpha$, $\kappa$ are
 ordinals, 
 then 
the relation $\alpha<\kappa$ means 
$\alpha\in \kappa$, and $\kappa+1=\kappa\cup\{\kappa\}$. 
For more information, 
we refer the readers to 
\cite{MR1940513}.

For a linearly ordered set $(L, \le_{L})$, 
we define the \emph{dual order  $\le_{\oposi{L}}$ of 
$\le_{L}$} by 
$a\le_{\oposi{L}}b$ if and only if 
$b\le_{L} a$. 
We denote by  
$\oposi{L}$ the ordered set $(L, \le_{\oposi{L}})$. 

In Definition \ref{df:opkappa}, 
we introduce a special linearly ordered set, 
which plays an important role in this paper. 
\begin{df}\label{df:opkappa}
For a cardinal $\kappa$, 
we define the ordered set $\invs{\kappa}$ by 
$\invs{\kappa}=\oposi{\kappa+1}$. 
Note that $\invs{\kappa}$ is bottomed and 
$\mzero_{\invs{\kappa}}=\kappa$, and note that 
the maximum of $\invs{\kappa}$ is $0$. 
\end{df}

We define  isotone and 
 antitone maps, 
and isomorphisms between ordered sets. 
\begin{df}\label{df:isotone}
Let $L, R$ be  linearly ordered sets. 
A map $f\colon L\to R$ is  said to be 
\emph{isotone} (resp.~\emph{antitone})
if for all 
$x, y\in L$, the inequality 
$x\le y$ implies 
$f(x)\le f(y)$ (resp.~$f(y)\le f(x)$). 
An injective isotone map is said 
to be an 
\emph{isotone embedding}. 
If $f$ is isotone and bijective, then
$f^{-1}$ is also isotone. In this case, 
the map $f$ is called an 
\emph{isomorphism
between ordered sets}, and 
$L$ and $R$ are  said to be 
\emph{isomorphic as ordered sets}. 
\end{df}

We define the notion of  characteristic 
maps (see also Definition \ref{df:charazero}). 
\begin{df}\label{df:chara}
Let $S$ be a 
bottomed linearly ordered set. 
Let $G$ be a linearly ordered Abelian group. 
Let $E$ be a set. 
A map $f\colon E\to S$ (resp.~ $f\colon E\to G$) is said to be 
\emph{characteristic} (resp.~
\emph{g-characteristic})
 if  its image is a characteristic subset of $S$ 
(resp.~a g-characteristic subset of $G$). 
\end{df}

By the definitions of 
$\GGG{\kappa}$ and 
$\FFF{\kappa}$,
and characters, 
we obtain: 
\begin{lem}\label{lem:charakappa}
Let $\kappa$ be a regular cardinal. 
Let $G\in \GGG{\kappa}$ and 
$S\in \FFF{\kappa}$. 
Then there exist a characteristic isotone 
embedding 
$l: \invs{\kappa}\to S$ and 
a g-characteristic isotone embedding 
$l': \invs{\kappa}\to G$. 
\end{lem}

Let $L$ be a linearly ordered set. 
A non-empty subset $A$ of $L$ is 
said to be 
\emph{bounded above} 
(resp.~\emph{bounded below}) if 
there exists $a\in L$ such that 
$x\le a$ (resp.~$a\le x$) for all 
$x\in A$. 
We say that 
 $L$ is \emph{Dedekind complete}
if for every non-empty subset $A$ of $L$ which is 
bounded above, 
the supremum $\sup A$ of $A$ exists in $L$. 
Note that $L$ is Dedekind complete if and only if 
every non-empty set bounded below has 
the  infimum. 
A subset $A$ of $L$ is said to be a \emph{cut} if 
the following conditions hold true:
\begin{enumerate}
\item $A\neq \emptyset$; 
\item $A$ is bounded above; 
\item 
if $x\in A$ and $y\in L$ satisfy  $y<x$, then $y\in A$; 
\item if $\sup A$ exists, then $\sup A\in A$. 
\end{enumerate}
We define the 
\emph{Dedekind completion}
$\comp{L}$ of $L$
 by the set of all cuts of $L$. 
We also define a map 
$\iota\colon L\to \comp{L}$ 
by 
$\iota(a)=\{\, x\in S\mid  x\le a\, \}$. 
Then $\comp{L}$ is linearly ordered by
the  inclusion 
$\mysub$, and the map $\iota$ is an 
isotone embedding. 
By abuse of notations, 
we represent the order on $\comp{L}$ as the same symbol as the order $\le$ on $L$. 
Using 
the canonical map $\iota\colon L\to \comp{L}$, 
we consider that  $L\mysub \comp{L}$. 

The statement (1) in the following  lemma 
is deduced  from 
the definition of $\comp{L}$. 
The statement (2) is
presented in 
  \cite[Proposition 1.1.4]{MR3156547}. 
\begin{lem}\label{lem:dedekindcomp}
Let $L$ be a linearly ordered set. Then 
the Dedekind completion 
$\comp{L}$ of $L$ satisfies the following statements: 
\begin{enumerate}
\item If $L$ is bottomed, then so is $\comp{L}$. 

\item If $s, t\in \comp{L}$ satisfy  $s<t$, then there exist
$x, y\in L$ with $s\le x<t$ and $s<y\le t$.\label{item:compst}

\end{enumerate}
\end{lem}

Let $\kappa$ be a cardinal. 
Since $\kappa+1$ is well-ordered and 
closed under the supremum operator, 
we obtain: 
\begin{lem}\label{lem:dedekindkappa}
Let $\kappa$ be a cardinal. 
Then
 the linearly  ordered set $\invs{\kappa}$
 is Dedekind complete. 
\end{lem}

\begin{df}\label{df:morph0}
Let $L$ be a linearly ordered set. 
A subset $E$  of $L$ is said to be 
\emph{cofinal} (resp.~\emph{coinitial})
 if 
for all $l\in L$, there exists $e\in E$ with 
$l\le e$ (reps.~$e\le l$). 
Let $E$ be a  set. 
A map $f\colon E\to L$ is said to be 
\emph{cofinal} (resp.~\emph{coinitial}) if 
its image is cofinal (resp.~coinitial). 
\end{df}
\begin{rmk}
A cardinal $\kappa$ is regular if and only if 
$\kappa$ is infinite  and 
$\kappa$ is the least cardinal of all cardinal 
$\nu$ 
satisfying that there exists a cofinal  isotone embedding 
$\nu\to \kappa$. Note that $\myomega$ is regular, 
and note that  if $\kappa$ is regular, 
then $\myomega\le \kappa$. 
\end{rmk}

By Lemma  \ref{lem:dedekindcomp}, 
we obtain the following two corollaries: 

\begin{cor}\label{cor:cofinal}
Let $L$ be a linearly ordered set. 
Then 
the set $L$ is
a cofinal and coinitial subset  in its 
Dedekind completion $\comp{L}$. 
\end{cor}

\begin{cor}\label{cor:ch}
Let $S$ be a bottomed  linearly ordered set. 
Then
the set $S$ is a characteristic subset 
 in 
its Dedekind completion 
$\comp{S}$. 
\end{cor}

\subsubsection{Linearly ordered Abelian groups}
As mentioned in Section \ref{sec:intro}, 
we verify the well-definedness of the order 
$\arcle$ on 
$\Arc{X}$. 

\begin{lem}\label{lem:welldef}
Let $G$ be a linearly ordered Abelian group. 
Let $x, y, u, v\in G_{>0}$.
If $x\arel y$ and $u\arel v$ and $x\ll u$, 
then we have $y\ll v$. 
Namely, the order $\arcle$ on $\Arc{G}$
is well-defined. 
\end{lem}
\begin{proof}
By $x\arel y$ and $u\arel v$, 
there exist integers $N, M\in \zz_{\ge 1}$ such that 
$y\le N\cdot x$ and $u\le M\cdot v$. 
Then we obtain 
$n\cdot y<nN\cdot x$ for all $n\in \zz_{\ge 1}$. 
From $x\ll u$, 
it follows that for all $n\in \zz_{\ge 1}$, 
we have 
$n\cdot y<u$. 
Hence $y\ll u$.
For the sake of contradiction, 
we suppose 
that 
there exists
$k\in \zz_{\ge 1}$ 
such that 
$v\le k\cdot y$. 
Then, we obtain 
$u\le Mk\cdot y$. 
This  contradicts $y\ll u$. 
Thus, we conclude that 
$y\ll v$. 
\end{proof}

\begin{df}\label{df:lambda}

For a linearly ordered Abelian group $G$, 
we define  the absolute value function 
$\abs\colon G\to G_{\ge 0}$ by 
\[
\abs(x)=
\begin{cases}
x & \text{if $x\ge 0$;}\\
-x & \text{if $x<0$.}
\end{cases}
\]
Some authors denote by $|x|$ the value  $\abs(x)$.
To emphasize that it is a function on $G$, 
we use the symbol $\abs(x)$ in this paper. 
We also define
a map  $\cova_{G}\colon G\to \lolo{\Arc{G}}$ by 
\[
\cova_{G}(x)=
\begin{cases}
[\abs(x)]_{\arel} & \text{if $x\neq 0$;}\\
\mzero_{\lolo{\Arc{G}}} & \text{if $x=0$.}
\end{cases}
\]
\end{df}

By definitions of $\abs$ and $\cova_{G}$,  we have: 
\begin{prop}\label{prop:abslambda}
Let $G$ be a linearly ordered Abelian group. 
Then  the following are satisfied:
\begin{enumerate}
\item\label{item:000cova}
For all $x\in G$, 
 we have 
$\cova_{G}(x)=\mzero_{\lolo{\Arc{G}}}$ if and only if 
$x=0_{G}$.
\item\label{item:abs}
For all $x, y\in G$, we have 
$\abs(x+y)\le \abs(x)+\abs(y)$.
\item\label{item:monotone}
If 
$x, y\in G$ and $n\in \zz_{\ge 1}$
satisfy 
$\abs(x)\le n\cdot  \abs(y) $, then 
we have 
$\cova_{G}(x)\arcle \cova_{G}(y)$.

\item\label{item:strong}
For all $x, y\in G$, 
 we have 
$\cova_{G}(x+y)\arcle 
\cova_{G}(x)\lor  \cova_{G}(x)$, where
$\lor$ is the maximum operator on $\lolo{\Arc{G}}$.
\end{enumerate}
\end{prop}
\begin{proof}
The statement (\ref{item:000cova}) follows from 
the definition of $\cova_{G}$. 

 The statement (\ref{item:abs}) follows
from the fact that 
 for all $a\in G$, we have 
 $a\le \abs(a)$
 and 
 $-a\le \abs(a)$. 
 
 We next verify (\ref{item:monotone}). 
 We divide the proof into two cases. 
 Recall that $\arel$ is  the relation on 
 $G_{>0}$. 
 
 Case 1 ($x=0$ or $y=0$):
 If $x=0$, then, by $\cova_{G}(0)=\mzero$, 
 we obtain the conclusion. 
 If $y=0$, then, by the assumption, 
 we have $x=y=0$. Thus, we obtain 
 $\cova_{G}(x)=\cova_{G}(y)$. 
 This finishes the proof in Case 1. 
 
 Case 2 ($x\neq 0$ and $y\neq 0$): 
 If $\abs(x)\arel\abs(y)$,then  we obtain the conclusion. 
 If $\abs(x)\not\arel \abs(y)$, 
 then, by 
 the assumption that $\abs(x)\le n\cdot \abs(y)$, 
 we have $m\cdot \abs(x)<\abs(y)$ for all 
 $m\in \zz_{\ge 1}$, 
 namely, $
 \abs(x)\ll \abs(y)$. 
 Lemma \ref{lem:welldef} implies that 
 $\cova_{G}(x)\arcle \cova_{G}(y)$. 
This finishes the proof of  (\ref{item:monotone}). 

We next show the statement (\ref{item:strong}).  
We may assume that 
$\abs(y)\le \abs(x)$. 
By the statement (\ref{item:abs}), 
we have 
 $\abs(x+y)\le 2\cdot \abs(x)$. 
 Then, by 
 the statement 
 (\ref{item:monotone}), 
 we have 
 $\cova_{G}(x+y)\arcle \cova_{G}(x)$. 
Since $\cova_{G}(x)\lor \cova_{G}(y)=
\cova_{G}(x)$, 
we obtain 
$\cova_{G}(x+y)\arcle\cova_{G}(x)\lor \cova_{G}(y)$. 
This completes the proof of  the statement  (\ref{item:strong}). 
\end{proof}

\begin{prop}\label{prop:uniuni}
Let $\kappa$ be a regular  cardinal. 
Let $G\in \GGG{\kappa}$. 
Then the following statements hold true. 
\begin{enumerate}
\item 
For all $\epsilon\in G_{>0}$, 
there exits $\rho\in \Arc{G}$ such that  if 
 $\cova_{G}(x)\arclele \rho$, 
then we have 
$\abs(x)<\epsilon$.

\item 
For all $\eta\in \Arc{G}$, 
there exists $\theta\in G_{>0}$
such that if $\abs(x)<\theta$, 
then $\cova_{G}(x)\arclele \eta$.
\end{enumerate}

\end{prop}
\begin{proof}
We first verify the statement  (1). 
Take $\epsilon \in G_{>0}$, 
and put $\rho=[\epsilon]_{\arel}$. 
If $\cova_{G}(x)\arclele \rho$,  
then we have 
$\abs(x)<\epsilon$ 
(if $\epsilon\le \abs(x)$, we have 
$\rho\arcle \cova_{G}(x)$ by (\ref{item:monotone}) 
in Proposition \ref{prop:abslambda}). 

We next show the statement (2). 
Since $\kappa$ is infinite, 
we can take $\eta, \eta'\in \Arc{G}$ with 
$\eta'<\eta$. 
Take $\theta\in G_{>0}$ with 
$\eta'=[\theta]_{\arel}$. 
If $\abs(x)<\theta$, 
then $\cova_{G}(x)\arcle \eta'$, 
and hence $\cova_{G}(x)\arclele \eta$.
This complete the proof.  
\end{proof}

\subsubsection{The Hahn groups and  fields}
For   a linearly  ordered set 
$L$, 
we say that $L$ is 
\emph{dually  well-ordered} 
if every  its non-empty subset has the maximum. 
This concept is the dual of the well-ordered property. 

We now define the  Hahn groups  and 
the Hahn fields. 
Let $L$ be a linearly ordered set. 
We define 
$\hahnsp{L}$ by the 
set of all maps 
$f\colon L\to \rr$ such that 
the set $\{\, x\in L\mid f(x)\neq 0_{\rr}\, \}$
is a dually well-ordered subset of  
$L$. 
We define $f+g$ as the coordinate-wise 
addition, namely, 
$(f+g)(x)=f(x)+g(x)$. 
We define $f\le_{\hahnsp{S}}g$ by 
$f(\theta)\le_{\rr}g(\theta)$, where 
$\theta=\max\{\, x\in L\mid f(x)\neq g(x)\,\}$. 

For a linearly ordered Abelian group 
$G$, we can define the product on 
the Hahn space $\hahnsp{G}$ induced from $G$ by 
\[
fg(x)=\sum_{a+b=x}f(a)g(b). 
\]
In this case,  the group $\hahnsp{G}$ becomes 
a field (see \cite{Hahn1907}, or see also \cite[Exercises 3.5.5 and 3.5.6]{MR2183496}). 
To emphasize the field structure, 
we denote by  $\hahnkor{G}$ the Hahn group 
$\hahnsp{G}$ equipped  with the product structure. 
Some authors used $\oposi{L}$ instead of $L$ and 
 defined the Hahn fields by the set of all 
functions whose supports are well-ordered sets in $\oposi{L}$ instead of dually well-ordered sets in $L$. 
The definition of the Hahn fields in this paper 
 can be found in  
\cite{MR2379002} and 
\cite{MR3946544}. 

\begin{lem}\label{lem:hahnlinear}
Let $L$ be a 
linearly ordered set. 
Then the Hahn group 
$\hahnsp{L}$
is a linearly ordered Abelian group. 
\end{lem}
\begin{proof}
Take $f, g, h\in \hahnsp{L}$, and assume that 
$f\le_{\hahnsp{L}}g$. 
Put 
$\theta=\max\{\, x\in L\mid f(x)\neq g(x)\,\}$. 
Then $f(\theta)\le_{\rr}g(\theta)$. 
By the definition of the addition on $\hahnsp{L}$, 
we have 
\[
\theta=\max\{\, x\in L\mid (f+h)(x)\neq (g+h)(x)\,\}.
\] 
Thus, $(f+h)(\theta)\le_{\rr}(g+h)(\theta)$, 
and hence 
 $f+h\le_{\hahnsp{L}} g+h$. 
Therefore we conclude  that $\hahnsp{L}$ is
a linearly ordered 
 Abelian group. 
\end{proof}

\begin{df}\label{df:morph}
Let $G$ be a linearly ordered Abelian group. 
A subset $E$ of $G$ is said to be 
\emph{full} if 
for all $x\in G$, there exists $a\in E$ such that 
$\abs(x)\arel \abs(a)$. 
Let $N$ be a set. 
A map 
$f\colon N\to G$ is said to be  \emph{full} if 
its image is a full subset of $G$. 
\end{df}

We next  discuss the relation between 
a linearly ordered set and the Hahn group 
induced from the linearly ordered set. 
\begin{df}\label{df:eemap}
Let $(L, \le)$ be a linearly  ordered set. 
For $s\in L$, 
we define $e_{L, s}\in \hahnsp{L}$ by 
$e_{L, s}(x)=0$ if $x\neq s$; otherwise, 
$e_{L, s}(s)=0$.  
We define a 
map 
$\eemap_{L}\colon L\to \hahnsp{L}$
by 
$\eemap_{L}(s)=e_{L, s}$. 
\end{df}
\begin{lem}\label{lem:echoes}
Let $L$ be a linearly ordered set. 
Then the map $\eemap_{L}\colon L\to \hahnsp{L}$ 
satisfies the following 
statements: 
\begin{enumerate}
\item\label{item:eee1}
For all $s\in L$, 
we have $0<\eemap_{L}(s)$. 
\item\label{item:eee2}
The map $\eemap_{L}$ is an
isotone embedding. 
\item\label{item:eee3}
 If $s, t\in L$ satisfy $s<t$, 
then we have $\eemap_{L}(s)\arclele \eemap_{L}(t)$.

\item\label{item:eeecova}
For all $f\in \hahnsp{L}$, we have 
\[
\cova_{\hahnsp{L}}(f)=
[\eemap_{L}(\max\{\, x\in L\mid f(x)\neq 0\, \})]_{\arel}. 
\]

\item\label{item:eee4}
The map $\eemap_{L}$ is full. 

\item\label{item:eee5}
The map 
$[\eemap_{L}]_{\arel}\colon L\to \Arc{\hahnsp{L}}$
defined by 
$[\eemap_{L}]_{\arel}(s)
=[\eemap_{L}(s)]_{\arel}$
 is 
 an
isomorphism  between  
$(L, \le_{L})$ and 
$(\Arc{\hahnsp{L}}, \arcle)$. 
\end{enumerate}
\end{lem}
\begin{proof}
By the definitions of $e_{L, s}$ and  $\eemap_{L}$, 
we have $0<\eemap_{L}(s)$ for all $s\in L$, which is 
the condition (\ref{item:eee1}). 

We next show that 
the map
$\eemap_{L}$ is an isotone embedding. 
Take $s, t\in L$ satisfying that  $s<t$. 
In this case, 
we have 
\[
t=\max\{\, x\in L\mid e_{L, s}(x)\neq e_{L, t}(x)\, \}, 
\]
and $e_{L, s}(t)=0$ and $e_{L, t}(t)=1$.
Hence,  for all $m\in \zz_{\ge 1}$, we
obtain 
$m\cdot e_{L, s}(t)<e_{L, t}(t)$. 
Thus, 
we have  $\eemap_{L}(s)\arclele \eemap_{L}(t)$. 
In particular, we also have 
$\eemap_{L}(s)<\eemap_{L}(t)$.
Then we conclude that
the statements (\ref{item:eee2}) and 
(\ref{item:eee3}) hold true. 

We next show the statement (\ref{item:eeecova}). 
Take $f\in \hahnsp{L}$, 
and put 
\[
\theta=\max\{\, x\in L\mid f(x)\neq 0\, \}\in L.
\] 
By the definition of the order on $\hahnsp{L}$, 
and by the fact that $\rr$ is Archimedean, 
we obtain  $\eemap_{L}(\theta)\arel f$. 
Thus, the statement 
(\ref{item:eeecova}) is true. 

We next verify that the map $\eemap_{L}$ is full. 
The statement (\ref{item:eeecova}) implies 
that the set 
$\{\, \eemap_{L}(s)\mid s\in L\, \}$
is  full in  $\hahnsp{L}$. 
Then the map $\eemap_{L}$ is full, and hence 
the statement (\ref{item:eee4}) is valid. 

We now prove the statement (\ref{item:eee5}). 
The statement (\ref{item:eee3}) implies that 
the map $[\eemap_{L}]_{\arel}$ is
an  isotone embedding. 
From (\ref{item:eee4}), 
the surjectivity of the map 
$[\eemap_{L}]_{\arel}$ follows. 
Therefore 
$[\eemap_{L}]_{\arel}\colon L\to \Arc{\hahnsp{L}}$ 
is an isomorphism between the  ordered sets
$(L, \le_{L})$ and 
$(\Arc{\hahnsp{L}}, \arcle)$. 
\end{proof}

\begin{rmk}
Let $G$ be a linearly ordered Abelian group. 
Instead of $\cova_{\hahnkor{G}}$, 
in the theory of valued fields, 
the 
Krull valuation 
$v_{\hahnkor{G}}\colon \hahnkor{G}\to G\sqcup\{\infty\}$ is 
ordinarily  used.  
The map $\cova_{\hahnkor{G}}$ is   a dual of $v_{\hahnkor{G}}$. 
The map $v_{\hahnkor{G}}$ is sometimes called the  
\emph{natural valuation on $\hahnkor{G}$} (see, for example, \cite{MR1760173}). 

\end{rmk}

\begin{df}\label{df:stars}
For a bottomed linearly ordered set 
$S$, we define 
$\stars{S}=S\setminus\{\mzero_{S}\}$. 
\end{df}

\begin{df}\label{df:opomap}
Let $L$ be a linearly ordered set. 
We define a map 
$\opmap_{L}\colon L\to \oposi{L}$
 by $\opmap_{L}(x)=x$.
 Note 
 $\opmap_{L}$ is antitone, namely, 
if  $x, y\in L$ satisfy 
$x\le_{L} y$, 
then $\opmap_{L}(y)\le_{\oposi{L}}\opmap_{L}(x)$. 
\end{df}

\begin{lem}\label{lem:antianti}
Let $S$ be a bottomed linearly ordered set
 with $\myomega\le \chara(S)$. 
Then the map 
$\eemapsec_{S}\colon \stars{S}\to \hahnsp{{\oposi{\stars{S}}}}$ defined by 
\[
\eemapsec_{S}(s)=-E_{\oposi{\stars{S}}}\circ\opmap_{\stars{S}}(s)
\]
 is a  coinitial  
isotone embedding. 
\end{lem}
\begin{proof}
By the statement (\ref{item:eee2}) in Lemma \ref{lem:echoes}, 
the map
$E_{\oposi{\stars{S}}}\colon \oposi{\stars{S}}\to \hahnsp{{\oposi{\stars{S}}}}$ 
is an  isotone embedding. 
We define
a map 
$m\colon \hahnsp{{\oposi{\stars{S}}}}\to \hahnsp{{\oposi{\stars{S}}}}$ 
by $m(x)=-x$. 
Then, 
the maps $\opmap_{\stars{S}}$ and 
$m$ are
antitone and  injective. 
 In this case, we obtain 
 $\eemapsec_{S}=m\circ E_{\oposi{\stars{S}}}\circ \opmap_{\stars{S}}$. 
Thus, the map $\eemapsec_{S}$ is
an  isotone embedding. 

We next show that $\eemapsec_{S}$ is coinitial. 
Take arbitrary $f\in \hahnsp{{\oposi{\stars{S}}}}$. 
Since $\eemap_{\oposi{\stars{S}}}$ is full
(see the statement (\ref{item:eee4}) in Lemma \ref{lem:echoes}), there exists $t\in L$
satisfying  that 
\begin{align}\label{al:21full}
[\abs(f)]_{\arel}=[\eemap_{\oposi{\stars{S}}}(t)]_{\arel}. 
\end{align}
By $\myomega\le \chara(S)$, 
there exists $s\in \stars{S}$ with $s<t$. 
According to 
 the statement (\ref{item:eee5}) in  Lemma \ref{lem:echoes}, and 
 $t<_{\oposi{\stars{S}}}s$, 
we have 
\begin{align}\label{al:2222ll}
[\eemap_{\oposi{\stars{S}}}(t)]_{\arel}
\arclele 
[\eemap_{\oposi{\stars{S}}}(s)]_{\arel}.
\end{align}
From  (\ref{al:21full}) and (\ref{al:2222ll}), 
it follows that 
\[
[\abs(f)]_{\arel}\arclele [\eemap_{\oposi{\stars{S}}}(s)]_{\arel}. 
\]
This inequality implies that 
$\abs(f)<E_{\oposi{\stars{S}}}(s)$. 
By $-f\le \abs(f)$, we obtain 
$-f<E_{\oposi{\stars{S}}}(s)$. 
Then $-E_{\oposi{\stars{S}}}(s)<f$, 
and hence  $\eemapsec_{S}(s)<f$. 
This means that
the map  $\eemapsec_{S}$ is coinitial. 
\end{proof}

\begin{df}\label{df:hahnp}
For a bottomed linearly ordered set $S$, 
we define  
$\hahntai{S}=\hahnkor{\hahnsp{\oposi{\stars{S}}}}$. 
\end{df}

\begin{prop}\label{prop:ee}
Let $S$ be a bottomed linearly ordered set
with $\myomega\le \chara(S)$. 
We define 
a map 
$I\colon S\to \hahntai{S}$
defined by 
\begin{align*}
I(s)=
\begin{cases}
E_{\hahnsp{\oposi{\stars{S}}}}\circ \eemapsec_{S}(s) & \text{if $\mzero_{S}<s$}\\
0_{\hahntai{S}}& \text{ if $s=\mzero_{S}$}
\end{cases}
\end{align*}
Then the map $\aimap$ is
a g-characteristic
isotone embedding. 
\end{prop}
\begin{proof}
Recall that $E_{\hahnsp{\oposi{\stars{S}}}}$ is 
a map 
$E_{\hahnsp{\oposi{\stars{S}}}}\colon 
\oposi{\stars{S}}\to \hahntai{S}$ defined in the same way to 
Definition \ref{df:eemap}. 
By Lemma \ref{lem:antianti}, 
the map $\eemapsec_{S}\colon \stars{S}\to \hahnsp{\oposi{\stars{S}}}$ is a coinitial  isotone embedding, and hence 
$\aimap\colon S\to \hahntai{S}$ is an 
isotone embedding.

We now show that $\aimap$ is 
g-characteristic. 
Take arbitrary  $f\in \hahntai{S}$ with 
$f>0_{\hahntai{S}}$. 
Put 
\[
\theta=\max\{\, x\in \hahnsp{\oposi{\stars{S}}}\mid f(x)\neq 0\, \}.
\] 
Then we have  $f(\theta)>0$. 
By the statement (\ref{item:eeecova}) in 
Lemma \ref{lem:echoes}, 
we verify  that  
$f\arel \eemap_{\hahnsp{\oposi{\stars{S}}}}(\theta)$. 
Since 
$\eemapsec_{S}$ is coinitial (see Lemma \ref{lem:antianti}),
there exists 
$s\in \stars{S}$
 such that 
 $\eemapsec_{S}(s)<\theta$. 
 By the statement (\ref{item:eee3}) in 
 Lemma \ref{lem:echoes}, 
 we have 
\begin{align}\label{al:love}
\eemap_{\hahnsp{\oposi{\stars{S}}}}(\eemapsec_{S}(s))
\ll
\eemap_{\hahnsp{\oposi{\stars{S}}}}(\theta).
\end{align}
By 
(\ref{al:love}) and
by 
$f\arel \eemap_{\hahnsp{\oposi{\stars{S}}}}(\theta)$, 
  we obtain the inequality 
\begin{align}\label{al:lovelove}
\eemap_{\hahnsp{\oposi{\stars{S}}}}(\eemapsec_{S}(s))<f.
\end{align}
According to the statement (\ref{item:eee1}) in Lemma \ref{lem:echoes}, we have 
\begin{align}\label{al:lovelovelove}
0_{\hahntai{S}}<\eemap_{\hahnsp{\oposi{\stars{S}}}}(\eemapsec_{S}(s)). 
\end{align}
Since  $I(s)=\eemap_{\hahnsp{\oposi{\stars{S}}}}(\eemapsec_{S}(s))$, 
by (\ref{al:lovelove}) and (\ref{al:lovelovelove}), 
we conclude that  
the map
$I$ is g-characteristic. 
\end{proof}
\begin{rmk}
Based on 
Lemma \ref{lem:hahnlinear}
and 
Proposition \ref{prop:ee}, 
we can consider
that
 every bottomed linearly ordered set is
a g-characteristic subset of  an ordered field. 
  \end{rmk}
  
\subsubsection{Characters}
We next discuss  the characters of linearly ordered Abelian groups. 
By the definition of characters, 
we obtain: 
\begin{lem}\label{lem:ch22}
Let $S$ be a bottomed linearly ordered set. 
Then 
the character $\chara(S)$ of $S$ is 
equal to the minimal cardinal of 
all cardinals $\kappa$ such that there exits
a characteristic isotone embedding 
$f\colon\invs{\kappa}\to S$. 
\end{lem}

By the definition of 
characteristic subsets, we obtain: 
\begin{lem}\label{lem:glass}
Let $S$ be a bottomed linearly ordered set. 
If $T$ is a characteristic subset of $S$, 
then we have $\chara(S)=\chara(T)$. 
\end{lem}

\begin{prop}\label{prop:chacha}
Let $G$ be a linearly ordered Abelian group
with $\myomega\le \chara(G_{\ge 0})$. 
Then we have 
$\chara(G_{\ge 0})=\chara(\lolo{\Arc{G}})$. 
\end{prop}
\begin{proof}
Let $\kappa=\chara(G_{\ge 0})$. 
Note that $\kappa$ is regular. 
According to  Lemma \ref{lem:ch22}, 
 we can  
 take a 
characteristic map 
$f\colon \invs{\kappa}\to G_{\ge 0}$. 
We define a map
$F\colon \invs{\kappa}\to \lolo{\Arc{G}}$  by 
$F(\alpha)=[f(\alpha)]_{\arel}$ if $\alpha\neq \mzero_{\invs{\kappa}}$; 
otherwise $F(\mzero_{\invs{\kappa}})=\mzero_{\lolo{\Arc{G}}}$. 
Then the map $F$ is characteristic; however, 
it can happen that $F$ is not isotone. 
Since $F$ is characteristic, 
we observe that for each $\alpha<\kappa$, we have
$\card(\{\, \beta<\kappa\mid F(\alpha)=F(\beta)\, \})<\kappa$. 
By this fact,  and by the regularity of  
$\kappa$, 
we can take a characteristic map 
$l: \invs{\kappa}\to \invs{\kappa}$ such that 
$F\circ l\colon \invs{\kappa}\to G_{\ge 0}$
 is a characteristic isotone embedding. 
Thus, 
by Lemma \ref{lem:ch22}, 
we obtain 
$\chara(G_{\ge 0})\le \chara(\lolo{\Arc{G}})$. 

Let $\mu=\chara(\lolo{\Arc{G}})$. 
According to  Lemma \ref{lem:ch22}, we 
can
 take a characteristic isotone embedding  
$g\colon \invs{\mu}\to \lolo{\Arc{G}}$. 
For each 
$\alpha< \mu$, 
we
take 
$h(\alpha)\in G_{>0}$ such that 
$g(\alpha)=[h(\alpha)]_{\arel}$. 
We define $h(\mu)$ by 
$h(\mu)=0_{G}$ 
(note that $\mzero_{\invs{\mu}}=\mu$).
Then 
$h\colon \invs{\mu}\to G_{\ge 0}$ 
is a 
characteristic isotone embedding, and hence 
$\chara(\lolo{\Arc{G}})\le \chara(G_{\ge 0})$. 
Therefore we conclude that 
$\chara(G_{\ge 0})=\chara(\lolo{\Arc{G}})$. 
\end{proof}

Lemma \ref{lem:glass} and
Propositions \ref{prop:ee} and 
 \ref{prop:chacha} imply the following:
\begin{cor}\label{cor:charaP}
Let $S$ be a bottomed linearly ordered set
with 
$\myomega\le \chara(S)$. 
Then we have $\chara(\hahntai{S})=\chara(S)$. 
Equivalently, we have 
$\hahntai{S}\in \GGG{\chara(S)}$. 
\end{cor}

\begin{cor}\label{cor:kappachara}
Let $\kappa$ be a regular cardinal. 
Then we obtain  
$\chara(\invs{\kappa})=\kappa$ 
and 
$\chara(\hahntai{\invs{\kappa}})=\kappa$. 
Equivalently, 
we have $\invs{\kappa}\in \FFF{\kappa}$
and $\hahntai{\invs{\kappa}}\in \GGG{\kappa}$. 
\end{cor}
\begin{proof}
The regularity of $\kappa$ indicates 
$\chara(\invs{\kappa})=\kappa$. 
Corollary \ref{cor:charaP} implies 
$\chara(\hahntai{\invs{\kappa}})=\chara(\invs{\kappa})$, 
and hence 
$\chara(\hahntai{\invs{\kappa}})=\kappa$. 
\end{proof}

\subsubsection{Embedding theorems}
The following is known as 
Hahn's embedding theorem. 
The proof can be seen in \cite{MR67882} and \cite{MR52045}. 
\begin{thm}\label{thm:hahn}
Let $G$ be a linearly  ordered Abelian group. 
Then 
there exists a full isotone embedding 
$F\colon G\to \hahnsp{\Arc{G}}$
which is a group homomorphism. 
\end{thm}

As a corollary of  Hahn's embedding theorem, 
we obtain
the following statement  known as  H\"{o}lder's embedding theorem: 
\begin{thm}\label{thm:holder}
Let $G$ be a linearly ordered Abelian group. 
If we have $\card(\Arc{G})=1$ (namely, $G$ is Archimedean), 
then there exits an isotone embedding
$F\colon G\to \rr$
 which is a group homomorphism. 
\end{thm}
\begin{rmk}
According to  Theorem \ref{thm:holder}, 
if $G$ is Archimedean, 
then 
$G$ can be considered as the subgroup of $\rr$. 
Thus, 
the group $G$ is dense in $\rr$, 
 or isomorphic to $\zz$ 
 (see, for example, \cite[Lemma 2.2]{abels2012topological}). 
\end{rmk}

\subsection{Metric structures}\label{subsec:met}
We next discuss  metric structures. 
\subsubsection{Basic statements on metrics}

\begin{df}\label{df:uequi}
Let $X$ and $Y$ be topological spaces. 
Let $G$ and $H$ be linearly ordered 
Abelian groups,  
and
let 
 $S$ and $T$ be bottomed linearly ordered sets. 
Let $d\in \met(X; G)$ or $d\in \ult(X; S)$, 
and $e\in \met(Y; H)$ or $e\in \ult(Y; T)$. 
Put $E=G_{>0}$ if $d\in \met(X; G)$; 
otherwise $E=\stars{S}$, and put 
$F=H_{>0}$ if 
$e\in \met(Y; H)$; 
otherwise $F=\stars{T}$. 
A map $f\colon (X, d) \to (Y, e)$
is  said to be \emph{uniformly continuous}
if
for all $\epsilon\in F$, 
there exists $\delta\in E$ such that 
if  $x, y\in X$ satisfy  $d(x, y)<\delta$, then 
$e(f(x), f(y))<\epsilon$. 
If there exists a bijection $f\colon (X, d)\to (Y, e)$ such that $f$ and $f^{-1}$ are uniformly continuous, 
we say that $(X, d)$ and $(Y, e)$ are 
\emph{uniformly equivalent to each other}. 
In the case of $X=Y$, 
we say that $d$ and $e$ are 
\emph{uniformly equivalent to each other} if 
the identity map $\mathrm{id}_{X}\colon (X, d)\to (X, e)$ and its inverse 
are uniformly continuous. 
The definition  of 
the uniform equivalence in this paper 
 is 
nothing but a 
paraphrase 
 of the
uniform equivalence between uniform spaces
induced from generalized (ultra)metrics. 
For uniform spaces, we refer the readers 
to \cite{W1970} or \cite{MR0370454}. 
\end{df}

\begin{lem}\label{lem:uniformdis}
We assume that 
$X$ is a topological space possessing an infinite metrizable gauge. 
Let 
$\kappa$ be a regular cardinal with 
$\kappa\in \metrank(X)$. 
Let $G\in \GGG{\kappa}$. 
Let $d\in \met(X; G)$. 
Then 
$\cova_{G}\circ d
\in \ult(X; \lolo{\Arc{G}})$, 
and
$d$ and $\cova_{G}\circ d$ are
uniformly equivalent to each other. 
\end{lem}
\begin{proof}
From the statements (\ref{item:000cova}), 
and 
(\ref{item:monotone})
in  Proposition \ref{prop:abslambda}, 
it follows  that $e$ satisfies the 
condition (U\ref{item:u0})--(U\ref{item:u3}) in 
Definition 
\ref{df:sultrametrics}. 

By the statements (\ref{item:abs})
and   (\ref{item:strong})
  in  Proposition \ref{prop:abslambda}, 
we verify
 that $e$ satisfies the strong triangle inequality
(the condition (U\ref{item:u4}) in Definition 
\ref{df:sultrametrics}). 
Hence $e$ is a 
$\lolo{\Arc{G}}$-ultrametric on $X$. 
The remaining 
 part 
 of the lemma follows from 
 Proposition \ref{prop:uniuni}. 
\end{proof}

We next prove an analogue of 
\cite[Lemma 2.2]{Ishiki2021ultra}
concerning 
functions preserving ultrametrics.
\begin{lem}\label{lem:amenable}
Let $S$ and $T$  be 
bottomed linearly ordered sets. 
We assume that 
 $\psi\colon S\to T$ satisfies that 
\begin{enumerate}
\item\label{item:iso1} 
the map $\psi$ is isotone;
\item\label{item:zero2}
if $\psi(x)=\mzero_{T}$, then 
$x=\mzero_{S}$;
\item\label{item:conti3}
$\psi$ is continuous at $\mzero_{S}$, namely, 
for all $\epsilon\in \stars{T}$, 
there exists $\delta\in \stars{S}$ such that 
if $x\in S$ satisfies $x<\delta$, then we have 
$\psi(x)<\epsilon$.
\end{enumerate} 
Let $d\in \ult(X; S)$. 
Then 
$\psi\circ d\in \ult(X; T)$, 
 and 
$d$ and $\psi\circ d$
generate the same topology on $X$. 
Moreover, 
they 
 are 
uniformly equivalent to 
each other. 
\end{lem}
\begin{proof}
Since $\psi$ is isotone, 
the  map
$\psi\circ d$ satisfies the strong triangle inequality
(the condition (U\ref{item:u4}) in Definition \ref{df:sultrametrics}). 
By the assumptions 
(\ref{item:iso1}) and  
(\ref{item:zero2}), 
the map $\psi\circ d$ 
satisfies the conditions (U\ref{item:u0})--(U\ref{item:u3}) 
in Definition \ref{df:sultrametrics}, 
and hence we conclude that 
$\psi\circ d$ is a
$T$-ultrametrics. 
It suffices to show that 
$d$ and $\psi\circ d$ are 
uniformly equivalent to each other. 
We divide the proof into two cases. 

Case 1 ($\stars{T}$ has the least element): 
By the assumptions 
(\ref{item:zero2}) and 
(\ref{item:conti3}), 
we verify 
 that $\stars{S}$
has the least element. 
Put $\alpha=\min \stars{S}$ and $\beta=\min \stars{T}$. 
In this case, 
if $d(x, y)<\alpha$  if and only if $\psi\circ d(x, y)<\beta$. 
Then $d$ and $\psi\circ d$ are uniformly equivalent to 
each other. 

Case 2 ($\stars{T}$ does not have the least element):
Take arbitrary $\epsilon\in \stars{T}$. 
By the assumption (\ref{item:conti3}), there exists $\delta$ with 
$\psi(\delta)<\epsilon$. 
Thus, if $x, y\in X$ satisfy $d(x, y)<\delta$, then 
$(\psi\circ d)(x, y)<\epsilon$. 

We next prove the converse. 
Take arbitrary $\eta\in \stars{S}$. 
Since $\stars{T}$ does not have the least element, 
we can take $\theta\in T$ with $\theta<\psi(\eta)$. 
By the assumption that 
$\psi$ is isotone, 
if $x, y\in X$ satisfy  $(\psi\circ d)(x, y)<\theta$, 
then $d(x, y)<\eta$. 
Therefore, the ultrametrics $d$ and $\psi\circ d$
are uniformly equivalent to each other. 
\end{proof}
\begin{rmk}
As the author proved 
Lemma 2.2 in 
\cite{Ishiki2021ultra} 
using \cite[Theorem 9]{MR3206769}
and 
\cite[Theorem 2.9]{MR4052988}, 
we can show the converse of 
Lemma \ref{lem:amenable}; 
namely, 
we can show 
that 
if $\psi\colon S\to T$ satisfies that 
for all $d\in\ult(d; S)$ we have 
$\psi\circ d\in \ult(X; T)$, 
then $\psi$ satisfies 
the conditions (\ref{item:iso1}), 
(\ref{item:zero2}) and 
(\ref{item:conti3})
in Lemma \ref{lem:amenable}. 
Since \cite[Lemma 2.2]{Ishiki2021ultra}, 
\cite[Theorem 9]{MR3206769}
and 
\cite[Theorem 2.9]{MR4052988} are proven using 
only the order structure of $\rr$, 
we can obtain an analogue of these statements for 
general ordered sets. 
\end{rmk}

\begin{lem}\label{lem:lll}
Let $S$ be a bottomed  linearly ordered set. 
Let $\kappa$ be a cardinal. 
Let $l\colon \invs{\kappa}\to S$ be a 
characteristic isotone embedding. 
Then the map $l$ satisfies 
the 
assumptions 
(\ref{item:iso1}), (\ref{item:zero2}), and 
(\ref{item:conti3}) appearing in 
Lemma \ref{lem:amenable}. 
\end{lem}
\begin{proof}
By the assumption, 
the map 
$l$
satisfies the conditions 
(\ref{item:iso1}) and (\ref{item:zero2}). 
Since $l$ is characteristic, 
it satisfies the condition (\ref{item:conti3}). 
\end{proof}

\begin{df}\label{df:zeta}
Let $\kappa$ be a regular cardinal. 
Let $S\in \FFF{\kappa}$. 
Let $l\colon\invs{\kappa}\to S$ be a 
characteristic isotone embedding. 
We define  a map $\zeta_{S, l}\colon S\to \invs{\kappa}$ by 
\begin{align*}
\zeta_{S, l}(s)=
\begin{cases}
0 & \text{if $l(0)<s$}\\
\alpha+1 & \text{if $l(\alpha+1)<s\le l(\alpha)$}\\
\mzero_{\invs{\kappa}} & \text{ if $s=\mzero_{S}$}. 
\end{cases}
\end{align*}
Note that $0$ is 
the maximal element of
$\invs{\kappa}$. 
Remark
 that 
$\zeta_{S, l}\colon S\to \invs{\kappa}$ satisfies the 
assumptions 
 (\ref{item:iso1}), 
 (\ref{item:zero2}), 
 and 
 (\ref{item:conti3}) in Lemma \ref{lem:amenable}. 
\end{df}

By the definition of g-characteristic subsets, 
we obtain: 
\begin{lem}\label{lem:inclusions}
Let $X$ be a topological space 
with $\metrank(X)\neq \emptyset$. 
Let  $G$ be a linearly ordered Abelian group
with 
$\met(X; G)\neq \emptyset$. 
If $S$ is a g-characteristic subset of $G$, 
we have $\ult(X; S)\mysub \met(X; G)$. 
\end{lem}
\begin{proof}
The lemma follows from the definition of 
g-characteristic subsets. 
Note that, 
to prove that all members in $\ult(X; S)$ satisfy the 
triangle inequality, 
we use the fact that 
$x\lor y\le x+y$ for all $x, y \in G_{\ge 0}$. 
\end{proof}

\begin{prop}\label{prop:characterization}
Let $\kappa$ be a regular cardinal. 
Let $X$ be a topological space. 
Then the following conditions are
equivalent to each other. 
\begin{enumerate}
\item\label{item:000} 
$\kappa\in \metrank(X)$;
\item\label{item:existsg} 
there exists $G\in \GGG{\kappa}$ such that
$\met(X; G)\neq \emptyset$;
\item\label{item:allg}
 for all  $G\in \GGG{\kappa}$, we have 
$\met(X; G)\neq \emptyset$;
\item\label{item:existss}
 there exists $S\in \FFF{\kappa}$ such that
$\ult(X; S)\neq \emptyset$;
\item\label{item:alls}
for all  $S\in \FFF{\kappa}$, we have 
$\ult(X; S)\neq \emptyset$;
\item\label{item:+1}
$\ult(X; \invs{\kappa})\neq \emptyset$. \end{enumerate}
\end{prop}
\begin{proof}
The equivalence $(\ref{item:000})\iff (\ref{item:existsg})$
is the definition of $\kappa\in \metrank(X)$. 

The implications $(\ref{item:allg})\To (\ref{item:existsg})$
and $(\ref{item:alls})\To (\ref{item:existss})$ follows 
from the fact that  $\GGG{\kappa}\neq \emptyset$ and 
$\FFF{\kappa}\neq \emptyset$ for all regular cardinals 
$\kappa$ (see Corollary \ref{cor:kappachara}).

According to  Proposition \ref{prop:chacha}, 
for every $G\in \GGG{\kappa}$, we have 
$\lolo{\Arc{G}}\in \FFF{\kappa}$. 
Then, by Lemma \ref{lem:uniformdis}, 
we obtain 
the implication 
$(\ref{item:existsg})\To (\ref{item:existss})$. 

We assume the condition (\ref{item:+1}). 
Take arbitrary $G\in \GGG{\kappa}$ and 
$S\in \FFF{\kappa}$. 
Take 
$d\in \ult(X; \invs{\kappa})$. 
Take 
a g-characteristic isotone embedding 
$f\colon \invs{\kappa}\to G$ and 
a characteristic isotone   embedding 
$g\colon \invs{\kappa}\to S$. 
Then, by Lemmas \ref{lem:amenable} and 
\ref{lem:lll}, 
we have $f\circ d\in \ult(X; G_{\ge 0})$ and 
$g\circ d\in \ult(X; S)$. 
Thus, 
we obtain 
the implication 
$(\ref{item:+1})\To (\ref{item:alls})$. 
According to 
Lemma \ref{lem:inclusions}, 
we observe 
 that 
$f\circ d\in \met(X; G)$, 
and hence 
we obtain 
the implication
$(\ref{item:+1})\To (\ref{item:allg})$. 
Note that the implication
$(\ref{item:+1})\To (\ref{item:allg})$ 
can be reduced  from 
Corollary \ref{cor:kappachara}. 

To finish the proof, it suffices to show the implication 
$(\ref{item:existss}) \To (\ref{item:+1})$. 
We assume the condition (\ref{item:existss}), and 
take $d\in \ult(X; S)$. 
Take a characteristic
isotone embedding 
$l\colon \invs{\kappa}\to S$, 
and  take a map 
 $\zeta_{S, l}\colon S\to \invs{\kappa}$ 
in Definition \ref{df:zeta}. 
We also define a map 
$e\colon X^{2}\to \invs{\kappa}$ by
$e=\zeta_{S, l}\circ d$. 
Since $\zeta_{S, l}$ satisfies 
the assumptions in
Lemma \ref{lem:amenable}, 
we obtain 
$e\in \ult(X; \invs{S})$. 
This finishes the proof. 
\end{proof}

We next observe that
the Hahn group and the Hahn fields are 
universal in the theory of generalized  metrics in 
some sense. 
Let $X$ be a set. 
Let $G$ be a linearly ordered  group and 
$S$ a bottomed linearly ordered set. 
Let $d$ be a $G$-metric or $S$-ultrametric on 
$X$. 
We denote by 
$U(x, \epsilon; d)$ (resp.~$B(x, \epsilon; d)$)
the open (resp.~closed) ball centered at 
$x$ with radius $\epsilon$.

\begin{lem}\label{lem:univg}

Let $G$ be a linearly ordered Abelian group. 
Let $X$ be a topological space with 
$\met(X; G)\neq \emptyset$ . 
Let $d\in \met(X; G)$. 
By Hahn's embedding theorem 
(Theorem \ref{thm:hahn}), 
we can consider that $G$ is a full subgroup of
$\hahnsp{\Arc{G}}$. 
Then, the topology of 
$(X, d)$ is generated by 
$\{U(x, \epsilon; d)\mid x\in X,  \epsilon\in\hahnsp{\Arc{G}}_{>0} \}$. 
In particular, we have $\met(X; G)\mysub \met(X; \hahnsp{\Arc{G}})$. 
\end{lem}
\begin{proof}
The lemma follows from the fact that 
 $G$ is full in 
$\hahnsp{\Arc{G}}$. 
\end{proof}

Let $S$ be a bottomed linearly  ordered set. 
We now define an extension of $S$. 
If $\stars{S}$ has the least element, we
define  a linearly ordered set 
$\mychu{S}$ as follows: 
Take a countable set 
$E=\{e_{i}\}_{i\in \zz_{\ge 0}}$
with $E\cap S=\emptyset$. 
Put 
$\mychu{S}=S\sqcup E$. We
define an order on 
$\mychu{S}$,  which is 
an extension of the order of $S$, 
 by 
 $\mzero_{S}<e_{i}<\min\stars{S}$ and $e_{i+1}<e_{i}$ for all $i\in \zz_{\ge 0}$. 
If $\stars{S}$ does not have the least element, 
then we put 
$\mychu{S}=S$. 
Note that $\mychu{S}$ satisfies 
$\myomega\le \chara(\mychu{S})$. 

\begin{lem}\label{lem:univs}
Let $S$ be a bottomed linearly ordered set. 
Let $X$ be a topological space with 
$\ult(X; S)\neq \emptyset$. 
Then, we have 
$\ult(X; S)\mysub 
\ult(X; \hahntai{\mychu{S}}_{\ge 0})$. 
Moreover, 
$\ult(X; S)\mysub 
\met(X; \hahntai{\mychu{S}})$. 
\end{lem}
\begin{proof}
 The lemma follows from Lemmas \ref{lem:amenable}
 and \ref{lem:inclusions}. 
\end{proof}

\begin{rmk}
Based on  Lemma \ref{lem:univs}, 
we can observe that 
some
statements on $G$-metrics 
are still valid for $S$-ultrametrics. 
For example, 
statements in the future of this paper such as 
Lemmas 
\ref{lem:senkei}, 
\ref{lem:kappacompleteandcomplete}, 
and 
\ref{lem:cptandcomplete}
 hold true for not only 
 $G$-metrics but also
 $S$-ultrametric spaces. 
\end{rmk}

\subsubsection{Metrizable  gauges}
We next give a complete description of 
$\metrank(X)$ for a topological space $X$
(Proposition \ref{prop:classify}). 
\begin{prop}\label{prop:arc=1}
Let $X$ be a topological space. 
If $1\in \metrank(X)$, 
then $X$ is $\rr$-metrizable (metrizable in the ordinary sense). 
\end{prop}
\begin{proof}
Take $G\in \GGG{1}$ such that 
$\met(X; G)\neq \emptyset$. 
According to 
 Theorem \ref{thm:holder}, the set $G$ can be considered as 
a subgroup of $\rr$. 
Thus, 
the space $X$ is $\rr$-metrizable. 
\end{proof}

\begin{prop}\label{prop:arc=countable}
A topological space 
$X$ is $\rr_{\ge 0}$-ultrametrizable (ultrametrizable in the ordinary sense) if and only if 
$\myomega\in \metrank(X)$. 
\end{prop}
\begin{proof}
The proposition follows from 
Proposition \ref{prop:characterization} and the fact that
$\rr_{\ge 0}\in \FFF{\myomega}$
(see also Lemma \ref{lem:univg}). 
\end{proof}
\begin{rmk}
We provide examples of  
theorems, 
similarly to Proposition
 \ref{prop:arc=countable}, 
 stating 
 that 
values of metrics determine the topology of the underlying space. 
As mentioned in Section \ref{sec:intro}, 
Dovgoshey--Shcherbak \cite{MR4335845} proved that 
an $\rr_{\ge 0}$-ultrametrizable space $X$ is separable if and only if 
for every
 $d\in \ult(X; \rr_{\ge 0})$, the set 
$\{\, d(x, y)\mid x, y\in X\, \}$ is countable. 
Broughan  \cite[Theorem 2]{MR314012} showed that 
for a topological space $X$, 
 the following are equivalent to each other:
\begin{enumerate}
\item\label{item:Broughan1}
 $X$ is $\rr_{\ge 0}$-ultrametrizable; 
\item\label{item:Broughan2} 
there exists $d\in \met(X; \rr)$ such that 
$\{\, d(x, y)\mid x, y\in X\, \}\mysub \{0\}\cup
\{\, 1/n\mid
n\in \zz_{\ge 1}\, \}$; 
\item\label{item:Broughan3}
there exist
 a characteristic  subset $S$ of $\rr_{\ge 0}$
isomorphic to $\invs{\myomega}$ as ordered sets, and 
 $d\in \met(X; \rr)$ such that 
$\{\, d(x, y)\mid x, y\in X\, \}\mysub S$.  
\end{enumerate}
Note that the metrics $d$ appearing 
in 
(\ref{item:Broughan2}) or 
(\ref{item:Broughan3})
is not assumed to be an  ultrametric 
(see also \cite[Proposition 2.14]{Ishiki2021ultra}). 

\end{rmk}

The next lemma states that all 
$\kappa$-metrizable space (in the sense of Sikorski)
 are $\kappa$-additive
(see \cite[(viii),  p129, ]{sikorski1950remarks}). 
\begin{lem}\label{lem:additive}
We assume that 
 $X$ is a topological space 
possessing an infinite metrizable gauge. 
Let 
$\kappa$ be a regular cardinal with 
$\kappa\in \metrank(X)$. 
Then, for every family $\{U_{i}\}_{i\in I}$ of open subsets 
with $\card(I)<\kappa$, 
the set $\bigcap_{i\in I}U_{i}$ is open in $X$. 
\end{lem}
\begin{proof}
Let $G\in \GGG{\kappa}$, and 
$d\in \met(X; G)$. 
Take a 
g-characteristic
isotone embedding
$l\colon \invs{\kappa}\to G$. 
The lemma follows from the fact that 
for every 
$a\in X$, 
and 
for every 
 $\theta<\kappa$, 
we have 
$U(a, l(\theta); d)\mysub 
\bigcap_{\alpha<\theta}U(a, l(\alpha); d)$. 
\end{proof}
\begin{prop}\label{prop:twocard}
Let $X$ be a topological space. 
We assume that there exist two  cardinals  $\kappa, \mu\in \metrank(X)$ with 
$\kappa<\mu$, 
and  assume that either of the following is satisfied: 
\begin{enumerate}
\item 
$\kappa=1$  and 
$\myomega<\mu$; 
\item 
$\myomega\le \kappa$. 
\end{enumerate}
Then 
the space $X$ is discrete. 
\end{prop}
\begin{proof}
Put $\theta=\max\{\myomega, \kappa\}$. 
By the assumptions, 
 we have $\theta<\mu$. 
By $\theta\in \metrank(X)$, 
in any case, 
for every point $x\in X$, 
there exists an open set 
$\{U_{\alpha}\}_{\alpha<\theta}$ such that 
$\{x\}=\bigcap_{\alpha<\theta}U_{\alpha}$. 
By Lemma \ref{lem:additive},  and 
 $\theta<\mu$, 
the set $\{x\}$ is open. 
Therefore the space $X$ is discrete. 
\end{proof}

\begin{prop}\label{prop:discrete}
If $X$ is a discrete space, 
then $\metrank(X)$ consists of 
and only of  
$1$ and 
all regular  cardinals. 
\end{prop}
\begin{proof}
Let $\kappa$ be a cardinal
such that $\kappa=1$ or it is regular. 
Take $G\in \GGG{\kappa}$, and 
take $\delta\in G_{>0}$. 
We define a map $d: X\times X\to G_{\ge 0}$ by 
$d(x, y)=\delta$ if $x\neq y$; otherwise, 
$d(x, x)=0$. 
Then $d\in \met(X; G)$, and hence 
$\kappa\in \metrank(X)$. 
\end{proof}

In summary, 
combining 
Propositions \ref{prop:arc=1}, 
\ref{prop:arc=countable}, 
\ref{prop:twocard}, 
and 
\ref{prop:discrete}, 
we obtain the following proposition: 
\begin{prop}\label{prop:classify}
Let $X$ be a topological space. 
The one and only one of the following 
holds true:
\begin{enumerate}
\renewcommand{\labelenumi}{(C\arabic{enumi})}
\item 
$\metrank(X)=\emptyset$.\label{item:met0}
\item 
$\metrank(X)=\{1\}$. \label{item:c1}
\item 
$\metrank(X)=\{1, \myomega\}$. \label{item:c2}
\item 
$\metrank(X)=\{\kappa\}$ for some 
 uncountable  regular cardinal $\kappa$. \label{item:c3}
\item 
$\metrank(X)$ consists of and only of   $1$ and  all regular cardinals. \label{item:c4}
\end{enumerate}
\end{prop}
\begin{rmk}
Let $X$ is a topological space. 
The  space $X$ satisfies the condition (C\ref{item:met0}) if and only if 
$X$ is not $G$-metrizable for any  $G\in \GGG{\kappa}$, nor 
 $S$-ultrametrizable for any $S\in \FFF{\kappa}$, 
 for any cardinal $\kappa$. 
The  space $X$ 
satisfies the condition (C\ref{item:c1}) if and only if 
$X$ is
$\rr$-metrizable and 
 non-discrete and 
non-$\rr_{\ge 0}$-ultrametrizable. 
The  space $X$ satisfies the condition (C\ref{item:c2}) if and only if 
$X$ is $\rr_{\ge 0}$-ultrametrizable and non-discrete. 
The  space $X$ satisfies the condition (C\ref{item:c3}) if and only if $X$ is non-discrete and 
$\kappa$-metrizable in the sense of Sikorski. 
The  space $X$ satisfies the condition (C\ref{item:c4}) if and only if 
$X$ is discrete. 
Note that 
$X$ possessing an infinite metrizable gauge if and only if $X$ satisfies  any one of  the conditions  
(C\ref{item:c2}), (C\ref{item:c3}), 
or  (C\ref{item:c4}). 
\end{rmk}

\subsubsection{Ultrametrics}

The following lemma can be proven by a 
similar method  
to the case of ordinary ultrametrics 
(see, for example, \cite[The statement 4 in the page 3]{MR2598517}). 
\begin{lem}\label{lem:isosceles}
Let 
$X$ 
be a set and 
$S$ be a bottomed 
linearly ordered set.
Let 
$d: X^2\to S$ 
be an $S$-ultrametric on $X$. 
Then 
for all 
$x, y, z\in X$, 
the inequality 
$d(x, z)<d(y, z)$ 
implies 
$d(y, z)=d(x, y)$. 
\end{lem}
The following is well-known in the case of ordinary ultrametrics. 
\begin{lem}\label{lem:centers}
Let 
$X$ 
be a set and 
$S$ be a bottomed linearly ordered set.
Let 
$d\colon X^{2}\to S$ be an $S$-ultrametric. 
Let $x\in X$ and $s\in \stars{S}$. 
Then,  for all $y\in B(x, s; d)$, we have 
$B(x, s; d)=B(y, s; d)$. 
In particular, 
each 
$B(x, s; d)$ is  clopen, and 
we have 
$d(a, b)\le s$ for all $a, b\in B(x, s; d)$. 
\end{lem}
\begin{proof}
Take $z\in B(x, s; d)$, 
then $d(y, z)\le d(y, x)\lor d(x, z)\le s$. 
Thus, 
$B(x, s; d)\mysub B(y, s; d)$. 
Similarly, 
we obtain
$B(y, s; d)\mysub B(x, s; d)$. 
The latter part follows from the former one. 
\end{proof}
\begin{rmk}
The similar statement to Lemma 
\ref{lem:centers} for open balls holds true. 
We omit the proof. 
\end{rmk}

\subsubsection{Completeness}

Let $X$ be a non-empty set. 
A set $\myfilter{F}$ 
consisting of subsets of $X$ is 
said to be a \emph{filter on $X$} if 
the following conditions are  satisfied:
\begin{enumerate}
\item $\emptyset\not \in \myfilter{F}$ (in this paper, all filters are assumed to be proper);
\item $X\in \myfilter{F}$; 
\item if $A, B\in \myfilter{F}$, then $A\cap B\in \myfilter{F}$; 
\item if $E\mysub X$ and $A\in \myfilter{F}$ satisfy 
$A\mysub E$, then $E\in \myfilter{F}$. 
\end{enumerate}
Let $X$ be a topological space. 
We say that a filter $\myfilter{F}$ on $X$ \emph{converges to 
$p\in X$} if $\myfilter{F}$ contains all neighborhoods of $p$. 
In this case, the point $p$ is called a 
\emph{limit point of $\myfilter{F}$}. 
We say that a 
 filter $\myfilter{F}$ on $X$  
 \emph{has a cluster point} if 
 $\bigcap_{A\in \myfilter{F}}\cl_{X}(A)\neq \emptyset$, where 
 $\cl_{X}$ is the closure operator of $X$. 
 
Let $X$ be a set. Let $G$ be a linearly ordered Abelian group and 
$S$ be a linearly ordered set. Let $d$ be a 
$G$-metric (resp.~an $S$-ultrametric).  
A filter $\myfilter{F}$ on $(X, d)$ is said to be 
\emph{Cauchy} if for all $\epsilon\in G_{>0}$ (resp.~$\epsilon\in \stars{S}$), there exists $A\in \myfilter{F}$ such that $d(x, y)< \epsilon$ for all $x, y\in A$. 
As defined in Section \ref{sec:intro}, 
we say that the space $(X, d)$ is \emph{complete} if 
every Cauchy filter on $(X, d)$ has a limit point. 

Let $X$ and  $Y$ be sets, 
and $f\colon X\to Y$ be a map. 
Then, for a filter $\myfilter{F}$ on $X$, 
we define a filter 
$f_{\sharp}\myfilter{F}$ on $Y$ by the 
filter generated by 
the set 
$\{\, f(A)\mid A\in \myfilter{F}\, \}$. 
The filter $f_{\sharp}\myfilter{F}$ is 
called the 
\emph{pushout filter of $\myfilter{F}$ by $f$}.

 For more discussions 
 on
  filters, 
 we refer the readers to 
 \cite{W1970}.

The following lemme states that 
being a Cauchy filter is invariant under the 
uniform equivalence, 
which can be deduced from  
  the definitions of 
Cauchy filters and 
the uniform equivalence. 
\begin{lem}\label{lem:cpltuni}
We assume that 
 $X$ is  a topological space possessing an infinite metrizable gauge. 
Let $\kappa$ be a regular cardinal  with $\kappa\in \metrank(X)$. 
Let $d\in \met(X; G)$ or $d\in \ult(X; S)$
for some $G\in \GGG{\kappa}$, 
or $S\in \FFF{\kappa}$. 
Let $e\in \met(X; H)$ or $e\in \ult(X; T)$
for some $H\in \GGG{\kappa}$, 
or $T\in \FFF{\kappa}$. 
If $d$ and $e$ are uniformly equivalent, 
then every Cauchy filter $\myfilter{F}$ on 
$(X, d)$ is also Cauchy  on $(X, e)$. 
Moreover, if $(X, d)$ is complete, 
then so is $(X, e)$. 
\end{lem}

\begin{lem}\label{lem:completecomplete}
We assume that $X$ is a topological space 
possessing  an infinite metrizable gauge. 
Let $\kappa$ be a cardinal with 
$\kappa\in \metrank(X)$. 
Let $G\in \GGG{\kappa}$. 
If there exists a complete 
$G$-metric $d\in \met(X; G)$, 
then there exists a
complete  $\invs{\kappa}$-ultrametric 
$h\in \ult(X; \invs{\kappa})$. 
\end{lem}
\begin{proof}
Put $S=\lolo{\Arc{G}}$. 
Put $e=\cova_{G}\circ d$. 
According to 
Lemmas \ref{lem:uniformdis}
and \ref{lem:cpltuni}, 
 we observe that $e\in \ult(X; \lolo{\Arc{G}})$ and 
$e$ is complete. 
Take a characteristic isotone embedding 
$l\colon \invs{\kappa}\to S$, and take a map 
$\zeta_{S, l}\colon S\to \invs{\kappa}$
in 
Definition 
\ref{df:zeta}. 
Put $h=\zeta_{S,  l}\circ e$. 
By Lemma \ref{lem:amenable}, 
we have 
$h\in \ult(X; \invs{\kappa})$, 
and 
we verify that 
 $e$ and $h$ are 
 uniformly equivalent to each other. 
By Lemma \ref{lem:cpltuni}, 
we conclude that $h$ is a complete
$\invs{\kappa}$-ultrametric. 
\end{proof}

We next show that  a Cauchy filter on a generalized metric space has a linearly ordered (nested) subfamily, 
 which plays
 a key role to discuss the completeness on 
 generalized metric spaces. 
\begin{lem}\label{lem:senkei}
We assume that 
$X$ is a topological space 
possessing an infinite metrizable gauge. 
Let $\kappa$ be a regular cardinal
with 
$\kappa\in \metrank(X)$. 
Let $G\in \GGG{\kappa}$ 
 and 
$d\in \met(X; G)$. 
Let $\myfilter{F}$ be a Cauchy filter on $(X, d)$. 
Let 
$l\colon \invs{\kappa}\to G$ be 
a g-characteristic isotone embedding. 
Then, 
there exists a  map 
$\myBmap\colon \kappa\to \myfilter{F}$ such that 
\begin{enumerate}
\item\label{item:senkei1}
each $\myBmap(\alpha)$ is a clopen subset of $X$, namely, the set $\myBmap(\alpha)$ is closed and open in $X$; 
\item\label{item:senkei2}
for all $\alpha<\kappa$, 
and for all $x, y\in \myBmap(\alpha)$, 
we have 
$d(x, y)\le l(\alpha)$;
\item\label{item:senkei3}
if $\alpha\le \beta$, then 
$\myBmap(\alpha)\mysup \myBmap(\beta)$, namely, 
the map $\myBmap$ is an isotone map from 
$(\kappa, \le )$ to 
the ordered set
$(\{\, \myBmap(\alpha)\mid \alpha<\kappa\, \}, \mysup)$. 
\end{enumerate}
\end{lem}
\begin{proof}
Put $S=\lolo{\Arc{G}}$, and 
$e=\cova_{G}\circ d\in \ult(X; S)$. 
Since $(X, d)$ and $(X, e)$ are 
uniformly equivalent to each other
(see Lemma \ref{lem:uniformdis}), 
by Lemma \ref{lem:cpltuni}, 
the filter $\myfilter{F}$ is also Cauchy on 
$(X, e)$. 
By this observation, 
for each $s\in \stars{S}$, there exists 
$A\in \myfilter{F}$ such that 
 $e(x, y)\le s$ for all $x, y\in A$. 
Then, by the definition of 
filters, we observe  that for all $s\in\stars{S}$, 
there exists $c(s)\in X$ 
such that 
$B(c(s), s; e)\in \myfilter{F}$. 
We define a map 
$\myBmap\colon \kappa\to \myfilter{F}$ 
by 
$\myBmap(\alpha)=B(c(l(\alpha)), l(\alpha); e)$. 
Then, 
by 
Lemma  \ref{lem:centers},  the map 
$\myBmap$ satisfies the conditions 
(\ref{item:senkei1}) and 
(\ref{item:senkei2}). 

We next prove the condition 
(\ref{item:senkei3}). 
Take $\alpha, \beta<\kappa$ with $\alpha\le \beta$. 
Since $\myfilter{F}$ is a filter, we have 
$\myBmap(\alpha)\cap \myBmap(\beta)\neq \emptyset$. 
Take $p\in \myBmap(\alpha)\cap \myBmap(\beta)$. 
Then, by Lemma \ref{lem:centers}, we have 
$\myBmap(\alpha)=B(p, l(\alpha); e)$ and 
$\myBmap(\beta)=B(p, l(\beta); e)$. 
By $l(\beta)\le l(\alpha)$, 
we obtain  $\myBmap(\alpha)\mysup \myBmap(\beta)$. 
This completes the proof. 
\end{proof}

Let $\kappa$ be a regular  cardinal. 
Let $X$ be a topological space with 
$\kappa\in \metrank(X)$. 
Let $G\in \GGG{\kappa}$ and 
$d\in \met(X; G)$. 
A map from $\kappa$ to $X$ is called a 
\emph{$\kappa$-sequence}. 
A $\kappa$-sequence $\{x_{\alpha}\}_{\alpha<\kappa}$
 is said to be \emph{Cauchy} if 
for every $\epsilon\in G_{>0}$, 
there exists $\gamma<\kappa$ such that 
for all $\alpha, \beta<\kappa$ with $\gamma<\alpha, \beta$, we have $d(x_{\alpha}, x_{\beta})<\epsilon$. 
We say that 
a $\kappa$-sequence $\{x_{\alpha}\}_{\alpha}$ converges 
to $p\in X$ if for all $\epsilon\in G_{>0}$, there 
there exists $\gamma<\kappa$ such that 
for all $\alpha<\kappa$ with $\gamma<\alpha$, 
we have $d(x_{\alpha}, p)<\epsilon$. 
In this case, the point $p$ is call a \emph{limit of 
$\{x_{\alpha}\}_{\alpha<\kappa}$}. 
A space $(X, d)$ is said to be 
\emph{$\kappa$-complete} if every Cauchy 
$\kappa$-sequence has a limit. 
The same notions on $S$-ultrametric spaces 
are defined by a similar method.

In \cite[Theorem 1.4]{stevenson1969results}, 
it is  stated that an $\omega_{\mu}$-metric space 
$(X, d)$ is $\omega_{\mu}$-complete if and only if 
the uniform space $X$ induced from $d$ is complete 
as a uniform space (i.e., all Cauchy filters are convergent), 
and the proof is omitted since that is  ``a 
straightforward generalization of the proof of standard 
topological theorem''; however, 
the author can not find a proof   by such  a method. 
In the present paper, we give   a proof using 
Lemma \ref{lem:senkei}: 

\begin{lem}\label{lem:kappacompleteandcomplete}
We assume that 
$X$ is a topological space 
possessing an infinite metrizable gauge. 
Let $\kappa$ be a regular  cardinal with 
$\kappa\in \metrank(X)$. 
Let $G\in \GGG{\kappa}$ and 
$d\in \met(X; G)$. 
Then 
the  space $(X, d)$ is complete if and only if 
it is $\kappa$-complete. 
\end{lem}
\begin{proof}
First assume that $(X, d)$ is complete. 
Take a $\kappa$-sequence 
$\{x_{\alpha}\}_{\alpha<\kappa}$. 
By considering the 
filter generated by 
$\{\, \{x_{\alpha}\mid \beta\le \alpha\}\mid \beta<\kappa\, \}$, 
and applying the completeness to this filter, 
we observe that $\{x_{\alpha}\}_{\alpha<\kappa}$ is convergent. Then,  the space 
$(X, d)$ is $\kappa$-complete. 

We next assume that $(X, d)$ is $\kappa$-complete. 
Take an arbitrary  Cauchy filter $\mathcal{F}$ on $X$. 
We take a g-characteristic 
isotone embeding 
$l\colon \invs{\kappa}\to G$
and 
 take a map 
$\myBmap\colon \kappa\to \myfilter{F}$ satisfying the conditions in 
Lemma \ref{lem:senkei}
associated with the map $l$. 
For each $\alpha<\kappa$, we take 
$x_{\alpha}\in \myBmap(\alpha)$. 
Then, 
by 
(\ref{item:senkei2}) 
and 
(\ref{item:senkei3}) in Lemma \ref{lem:senkei}, 
the $\kappa$-sequence  
$\{x_{\alpha}\}_{\alpha<\kappa}$ 
is 
Cauchy 
on 
$(X, d)$. 
Since $(X, d)$ is $\kappa$-complete, there exists a limit 
$p$ of 
$\{x_{\alpha}\}_{\alpha<\kappa}$. 
According to the condition  (\ref{item:senkei2}) in Lemma 
\ref{lem:senkei}, 
we have 
$\bigcap_{\alpha<\kappa}\myBmap(\alpha)=\{p\}$. 
Then, 
the filter 
$\mathcal{F}$ converges to $p$, 
and hence 
$(X, d)$ is complete. This finishes the proof. 
\end{proof}

\begin{lem}\label{lem:cptfilter}
Let $\kappa$ be a cardinal. 
Let $X$ be a finally $\kappa$-compact  topological space. 
If a family $\{A_{\alpha}\}_{\alpha<\kappa}$ of 
closed subsets of $X$ satisfies that for all $\theta<\kappa$ we have $\bigcap_{\beta<\theta}A_{\beta}\neq \emptyset$, 
then $\bigcap_{\alpha<\kappa}A_{\alpha}\neq \emptyset$. 
\end{lem}
\begin{proof}
Supposing  $\bigcap_{\alpha<\kappa}A_{\alpha}=\emptyset$, and 
applying the final $\kappa$-compactness to 
$\{X\setminus A_{\alpha}\}_{\alpha<\kappa}$, 
we obtain a contradiction. 
\end{proof}

\begin{lem}\label{lem:cptandcomplete}
We assume that 
$X$ is a topological space possessing 
an infinite metrizable gauge. 
Let $\kappa$ be a regular  cardinal
 with 
$\kappa\in \metrank(X)$. 
Let $G\in \GGG{\kappa}$. 
If $X$ is finally $\kappa$-compact, 
then for all $d\in \met(X; G)$, 
the space  $(X, d)$ is complete. 
\end{lem}
\begin{proof}
Let $\myfilter{F}$ be a Cauchy filter on $(X, d)$. 
We take a g-characteristic 
isotone embedding  
$l: \invs{\kappa}\to G$, 
 and 
take a map 
$\myBmap \colon\kappa\to \myfilter{F}$ 
satisfying the conditions in 
Lemma \ref{lem:senkei}
associated with the map $l$. 
By the condition (\ref{item:senkei3}) 
in Lemma \ref{lem:senkei}, 
for every $\theta<\kappa$, 
we have 
$\bigcap_{\beta<\theta}\myBmap(\beta)\neq \emptyset$. 
Since $X$ is finally $\kappa$-compact, 
by Lemma \ref{lem:cptfilter}, 
we have 
$\bigcap_{\alpha<\kappa}\myBmap(\alpha)\neq \emptyset$.
By the conditions
 (\ref{item:senkei2}) in Lemma \ref{lem:senkei}, 
there exists $p\in X$ with 
$\bigcap_{\alpha<\kappa}\myBmap(\alpha)=\{p\}$. 
Then, $\myfilter{F}$ converges to $p$. Thus, 
the space $(X, d)$ is complete. 
\end{proof}

\begin{rmk}
Let $\kappa$ be a cardinal. 
We say that a filter $\myfilter{F}$ is a 
\emph{$\kappa$-filter}
if 
the intersection of less than $\kappa$ many 
member of $\myfilter{F}$ belongs to $\myfilter{F}$. 
Using the notion of $\kappa$-filters, 
the proof of Lemma \ref{lem:cptandcomplete} can be 
translated as follows: 
Lemma \ref{lem:senkei} states that 
every Cauchy filter $\myfilter{F}$ on $(X, d)$
contains a Cauchy $\kappa$-filter $\myfilter{G}$. 
Since $X$ is finally $\kappa$-compact, 
by \cite[Proposition 2.2]{MR682706}, 
the filter $\myfilter{G}$ has a cluster point. 
Since $\myfilter{G}$ is Cauchy,  it is convergent, 
and hence so does $\myfilter{F}$. 
This means that $(X, d)$ is complete. 
\end{rmk}

\subsection{Examples}
We shall  provide
 some examples of 
generalized metric spaces  and ultrametrics spaces. 
\begin{df}
Let $G$ be a linearly ordered Abelian group. 
We define maps 
$\metabs, \metcova\colon G\times G\to G$ by 
$\metabs(x, y)=\abs(x-y)$ and 
$\metcova(x, y)=\cova_{G}(x-y)$. 
\end{df}
\begin{prop}\label{prop:GG}
Let $G$ be a linearly ordered Abelian group. 
Then, 
we have 
$\metabs\in \met(G;  G)$. 
Moreover, 
if 
$G\in \GGG{\kappa}$ for some
infinite cardinal $\kappa$, 
we have 
$\metcova\in \ult(G; \lolo{\Arc{G}})$. 
\end{prop}
\begin{proof}
Similarly to the proof that the order topology on 
$\rr$ are generated by $|*|$, 
we can prove that  the order topology is generated by 
the absolute value 
$\abs\colon G\to G_{\ge 0}$. 
By Proposition 
\ref{prop:uniuni}, 
we observe  that 
$\metabs$ 
and 
$\metcova$
generate the same topology on $G$. 
This finishes the proof. 
\end{proof}

\begin{df}\label{df:kappasp}
Let $S$ be a bottomed linearly ordered set. 
We define an ultrametric 
$\orddis_{S}$ by 
\[
\orddis_{S}(x, y)
=
\begin{cases}
\mzero_{S} & \text{if $x=y$;}\\
x\lor y & \text{if $x\neq y$, }
\end{cases}
\]
where $\lor$ means the maximum operator on 
$S$. 
Then $\orddis_{S}$ is an $S$-ultrametric. 
\end{df}

\begin{rmk}
The metric $\orddis_{S}$ is a generalization of 
the Laflamme--Pouzet--Sauer's
construction of ultrametrics
 \cite[Proposition 2]{MR2435142}, 
which also can be found in \cite{Ishiki2021ultra} and
\cite{MR2854677}. 
\end{rmk}

\begin{lem}\label{lem:kappaspsp}
Let $S$ be a bottomed linearly ordered set. 
Let $\kappa$ be a regular cardinal. 
Then  the following statements hold true:
\begin{enumerate}
\item\label{item:ssss1}
The set  
$\stars{S}$ 
is a discrete subset of 
$(S, \orddis_{S})$. 
\item\label{item:ssss2}
If $\stars{S}$ does not have 
 the least element, 
then 
the point $\mzero_{S}\in S$ is the unique accumulation point of 
$(S, \orddis_{S})$. 
\end{enumerate}
\end{lem}
\begin{proof}
We first show the statement  (\ref{item:ssss1}). 
Take  arbitrary $x\in \stars{S}$. 
By the definition of $\orddis_{S}$, 
we have 
$U(x, x; \orddis_{S})=\{x\}$. 
Thus, the set $\stars{S}$ is discrete. 

We next verify the statement  (\ref{item:ssss2}). 
Take arbitrary $x\in \stars{S}$. 
Since $\stars{S}$ dos not have the least element, 
we can take $y\in \stars{S}$ with $y<x$. 
Since we have 
$y\in U(\mzero_{S}, x; \orddis_{S})$, 
the point $\mzero_{S}$ is an accumulation point of 
the space
$(S, \orddis_{S})$. 
\end{proof}

\begin{rmk}
Let $S$ be a bottomed linearly  ordered set. 
The order topology on $S$ is
not induced from $\orddis_{S}$ in general. 
\end{rmk}

Let $\kappa$ be a cardinal. 
Note that  we have 
$\invs{\kappa}=\kappa+1$ as sets,  and hence 
we obtain $\kappa\mysub \invs{\kappa}$,  and 
$\kappa=\mzero_{\invs{\kappa}}\in \invs{\kappa}$. 
Applying Lemma 
\ref{lem:kappaspsp}
to the order set 
$\invs{\kappa}$, 
we obtain:

\begin{cor}\label{cor:kappaspsp}
Let $\kappa$ be a regular cardinal. 
Then the following statements 
hold true:
\begin{enumerate}
\item\label{item:kpkp1}
The set  
$\kappa=\{\, \alpha\mid \alpha<\kappa\, \}$ 
is a discrete subset of 
$(\invs{\kappa}, \orddis_{\invs{\kappa}})$. 
\item\label{item:kpkp2}
The point $\kappa\in \invs{\kappa}$ is the unique accumulation point of the space 
$(\invs{\kappa}, \orddis_{\invs{\kappa}})$. 
\item\label{item:kpkp3}
The  $\kappa$-sequence 
$\{\alpha\}_{\alpha<\kappa}$
in $\invs{\kappa}$ converges to the point 
$\kappa\in \invs{\kappa}$. 
\end{enumerate}
\end{cor}
\begin{proof}
The statements (\ref{item:kpkp1}) and 
(\ref{item:kpkp2}) are deduced from 
 (\ref{item:ssss1}) and 
(\ref{item:ssss2}) in 
Lemma \ref{lem:kappaspsp}. 
The statement (\ref{item:kpkp3})
follows from the statement 
(\ref{item:kpkp2}) and the definition of 
$\orddis_{\invs{\kappa}}$. 
\end{proof}


\section{Retractions}\label{sec:retraction}

\subsection{Proof of Theorem \ref{thm:retract}}\label{subsec:proofof1}
To show Theorem \ref{thm:retract}, 
we first  prove 
the existence of Lipschitz retractions
 under 
certain conditions. 
\begin{df}\label{df:setdis}
Let $X$ be a set. 
Let $G$ be a linearly ordered Abelian group. 
Let $d$ be a $G$-metric on $X$. 
Then,
 we define 
 $\setdis_{d, A}\colon X\to G$
 by 
\[
\setdis_{d, A}(x)=
\inf\{\, d(x, a)\mid a\in A\, \}
\]
if the infimum in  right hand side exists. 
If $d$ is an 
$S$-ultrametric on $X$ for a linearly ordered set 
$S$, 
similarly to the case of  $G$-metrics,  
the function $\setdis_{d, A}$ is 
defined 
as 
$\setdis_{d, A}(x)=\inf\{\, d(x, a)\mid a\in A\, \}$. 
Note that 
 the infimum is taken in $S$ even if 
$S$ is a subset of some other linearly 
ordered set. 
Thus, we should actually 
denote by $\setdis_{S, d, A}(x)$ rather than 
$\setdis_{d, A}(x)$; however,  in this paper, 
 there are no confusions
on sets where we take the infimum.  

\end{df}

To show Theorem \ref{thm:uniformretract}, 
we use a celebrated  construction of 
Lipschitz retractions
in  the proof of 
\cite[Theorem 2.9]{brodskiy2007dimension}. 

\begin{thm}\label{thm:uniformretract}
Let $\kappa$ be a regular cardinal. 
Let $X$ be a topological space with 
$\kappa\in \metrank(X)$. 
Let $G\in \GGG{\kappa}$, 
and $S$ be a Dedekind complete g-characteristic  subset of  
$G$ 
(in this case, we have 
$\ult(X; S)\neq \emptyset$). 
Let 
$d\in \ult(X; S)$. 
Let $A$ be a closed subset of $X$. 
Let 
$\concon\in G$ with $1<\concon$. 
Then there exists a retraction 
$r\colon X\to A$ satisfying that 
for all 
$x, y\in X$, we have 
$d(r(x), r(y))\le \concon^{2}d(x, y)$. 
\end{thm}
\begin{proof}
Take a well-ordering 
$\preceq$ on $X$. 
We write $x\prec y$ if $x\preceq y$ and $x\neq y$. 
Remark that the well-ordering  $\preceq$ is not related to
the topology of $X$ nor 
 the order on $G$ 
 in general. 
Since $S$ is Dedekind complete, 
the value $\setdis_{d, A}(x)$ always exists. 
For every $x\in X$, we 
define 
\[
\stst{A}{x}=\{\, a\in A\mid \concon^{-1}\cdot d(x, a)\le 
\setdis_{d, A}(x)\, \}. 
\]
By $1<\concon$, 
we have $\setdis_{d, A}(x)<\concon\cdot \setdis_{d, A}(x)$, and hence there exits  $a\in A$ such that 
$d(x, a)\le \concon\cdot \setdis_{d, A}$. 
Thus, each $\stst{A}{x}$ is non-empty. 
Note that 
if $a\in A$, 
then $\stst{A}{a}=\{a\}$. 

We define
a 
map 
 $r\colon X\to A$ by 
$r(x)=\min_{\preceq}\stst{A}{x}$,
 namely, 
$r(x)$ is the minimal 
 element of 
$\stst{A}{x}$ with respect to the well-ordering 
$\preceq$. 
Since $\stst{A}{a}=\{a\}$ for all $a\in A$, 
we have $r(a)=a$ for all $a\in A$. 

Thus, to finish the proof, 
we only need to show 
 that 
for all $x, y\in X$, we have 
$d(r(x), r(y))\le \concon^{2}d(x, y)$. 
For the sake of contradiction, 
we suppose  that 
there exist $x, y\in X$ such that 
\begin{align}\label{al:000}
\concon^{2}d(x, y)<d(r(x), r(y))
\end{align}
We may assume that $r(x)\prec r(y)$. 
Depending on whether 
$d(y, r(x))\le d(y, r(y))$ holds true or not, 
we divide the proof into two cases. 

Case 1 ($d(y, r(x))\le d(y, r(y))$): 
In this case, 
we obtain $r(x)\in \stst{A}{y}$.
Since $r(x)\prec r(y)$,  this contradicts 
the minimality of $r(y)$ in $\stst{A}{y}$.

Case 2 ($d(y, r(y))<d(y, r(x))$):
We first prove that 
the three values 
$d(r(x), x)$, $d(r(x), y)$,  and $d(r(x), d(y))$ 
are equal to each other. 
This means that the four points 
$\{r(x), r(y), x, y\}$ form a 
trigonal pyramid whose lengths 
 of 
edges starting with  $r(x)$ 
are equal to each other. 
We put $D=d(r(x), r(y))$. 
Applying 
Lemma  \ref{lem:isosceles}
to $\{y, r(x), r(y)\}$, 
and 
by $d(y, r(y))<d(y, r(x))$, we obtain 
\begin{align}\label{al:111}
D=d(r(x), r(y))=d(r(x), y). 
\end{align}
By $1<\concon$ and (\ref{al:000}), 
we have $d(x, y)<D$. 
Since  $D=d(y, r(x))$, 
by applying Lemma \ref{lem:isosceles} to $\{x, y, r(x)\}$ 
and by $d(x, y)<D$, we have 
\begin{align}\label{al:444}
d(r(x), x)=d(r(x), y)=D. 
\end{align}
This completes  our first purpose. 

We next determine which two values in  
$\{d(x, y), d(x, r(y)), d(y, r(y))\}$ are 
equal to each other. 
By (\ref{al:000}), by the definitions of 
$r(x)$ and $\stst{A}{x}$, and by $r(y)\in A$,  we have 
\begin{align*}
\concon^{-1}\cdot D
=
\concon^{-1}\cdot d(x, r(x))
\le 
\setdis_{d, A}(x)
\le d(x, r(y)). 
\end{align*}
This implies 
\begin{align}\label{al:555}
\concon^{-1}\cdot D\le d(x, r(y)). 
\end{align}
By (\ref{al:000}), (\ref{al:555}) and 
$1<\concon$, 
we have
$d(x, y)<d(x, r(y))$. 
By this inequality  and by  applying Lemma \ref{lem:isosceles} to $\{x, y, r(x)\}$, we obtain 
\begin{align}\label{al:666}
d(r(y), x)=d(r(y), y). 
\end{align}

From these observations, 
we shall deduce a contradiction. 
By $r(x)\prec r(y)$, 
we see that  $r(x)\not \in \stst{A}{y}$. 
Thus, 
$\setdis_{d, A}(x)<\concon^{-1}\cdot d(r(x), y)$. 
Since $D=d(r(x), y)$, there exists $b\in A$ 
such that $d(x, b)<\concon^{-1}\cdot D$. 
By (\ref{al:000}), (\ref{al:555}), 
 and 
(\ref{al:666}), and by $b\in A$, we observe that 
\begin{align*}
d(x, y)&<
\concon^{-2}\cdot D
\le 
\concon^{-1}\cdot d(x, r(y))
=
\concon^{-1}\cdot d(y, r(y))\\
&\le \setdis_{d, A}(y)\le d(y, b). 
\end{align*}
Namely, $d(x, y)<d(y, b)$. 
Then, applying Lemma \ref{lem:isosceles} to $\{x, y, b\}$, 
we have 
$d(x, b)=d(y, b)$. 
By  $d(y, b)<\concon^{-1}\cdot D$ and  $D=d(x, r(x))$, and 
by 
$r(x)\in \stst{A}{x}$, we have 
\[
d(x, b)=d(y, b)<
\concon^{-1}\cdot d(x, r(x))
\le 
\setdis_{d, A}(x). 
\]
Namely,  $d(x, b)<\setdis_{d, A}(x)$, 
which contradicts 
$\setdis_{d, A}(x)\le d(x, b)$ (recall that $b\in A$). 

In any case, 
we obtain a contradiction. 
Therefore, 
we conclude that  
$d(r(x), r(y))\le \concon^{2}d(x, y)$
for all $x, y\in X$. 
This completes the proof. 
\end{proof}

\begin{proof}[Proof of Theorem \ref{thm:retract}]
We assume that $X$ 
is a topological space possessing 
an infinite metrizable gauge. 
Let $\kappa$ be a regular cardinal
 with 
$\kappa\in \metrank(X)$. 
Let $G\in \GGG{\kappa}$ and $d\in \met(X; G)$. 
 Let $A$ be a
non-empty closed subset of $X$. 
The latter part of the theorem 
follows from the former  one. 
We only need to prove the former part of 
Theorem \ref{thm:retract}.

Put $S=\lolo{\Arc{G}}$, 
and $e=\cova_{G}\circ d$. 
By Lemma \ref{lem:uniformdis}, 
we have  $e\in \ult(X; S)$. 
Since $S$ can be regarded as a characteristic subset 
of $\comp{S}$ (see Lemma \ref{lem:dedekindcomp}), and 
since $\comp{S}$ can be regarded as a 
g-characteristic subset of $\hahntai{\comp{S}}$
(see Proposition \ref{prop:ee}), 
we can considered that 
$e\in \ult(X; \comp{S})$ and 
$e\in \met(X; \hahntai{\comp{S}})$
(see Lemma \ref{lem:inclusions}). 
Take $\concon\in \hahntai{\comp{S}}$ with 
$1<\concon$. 
Since 
$\comp{S}$ is Dedekind complete, 
we can apply  Theorem \ref{thm:uniformretract} to 
$X$, $A$, $e$, and $\concon$. 
Thus, we obtain a retraction
 $r\colon X\to A$ such that 
$e(r(x), r(y))\le \concon^{2}e(x, y)$ for all $x, y\in X$. 
In particular, 
the map $r$ 
is uniformly continuous with respect to $e$. 
Since $d$ and $e(=\cova_{G}\circ d)$ are uniformly equivalent to each other (see Lemma \ref{lem:uniformdis}), 
the retraction $r\colon X\to A$ is uniformly continuous 
with respect to $d$. 
This finishes the proof of Theorem \ref{thm:retract}. 
\end{proof}



\subsection{Applications}\label{subsec:cors}
In this subsection, 
using the existence of retractions, 
we provide a characterization of
 the closedness of 
subsets of generalized metric spaces. 

Let $X$ be a topological space 
and 
$d\in \met(X; G)$ or $d\in \ult(X; S)$.
A map $f\colon X\to X$ is said to be 
\emph{
$1$-Lipschitz} if 
we have $d(f(x), f(y))\le d(x, y)$ for all 
$x, y\in X$. 

\begin{prop}\label{prop:1-lipretract}
We assume that $X$ is a topological space 
possessing an infinite metrizable gauge. 
Let $\kappa$ be a regular cardinal 
 with 
$\kappa\in \metrank(X)$. 
Let $A$ be a closed subset of $X$. 
Let $G\in \GGG{\kappa}$. 
Let $h\in \met(X; G)$. 
Let $r\colon X\to A$ be a retraction. 
We define a $G$-metric $k$
on $X$ by 
\[
k(x, y)=h(x, y)\lor h(r(x), r(y)).
\] 
Then 
we have 
$k\in \met(X; G)$ and 
the map
$r$ is $1$-Lipschitz with respect to 
the $G$-metric $k$. 
\end{prop}
\begin{proof}
Since $r$ is continuous, 
we obtain $k\in \met(X; G)$. 
Since $r$ is a retraction, 
for all $x, y\in X$, we have 
\begin{align*}
k(r(x), r(y))&=h(r(x), r(y))\lor h(r(r(x)), r(r(y)))\\
&=h(r(x), r(y))\lor h(r(x), r(y))=h(r(x), r(y))\\
&\le k(x, y). 
\end{align*}
Namely,  the map 
 $r$ is $1$-Lipschitz with respect to 
$k$. 
\end{proof}

Let $X$ be a topological space with 
$\metrank(X)\neq \emptyset$ and 
let $S$ be a linearly ordered set 
with $\ult(X; S)\neq \emptyset$. 
Let $d\in \ult(X; S)$. 
A subset $A$ of $X$ is 
said to be 
\emph{proximal}
  if 
for all $x\in X$ there exists $a\in A$ such that 
\[
d(x, a)=\inf\{\, d(x, z)\mid z\in A\, \}. 
\]
The proximality for  $G$-metric spaces is
defined in the same way. 

The following theorem  is  a characterization of the proximality in 
general ultrametric spaces using $1$-Lipschitz maps. 
The proof can be seen in \cite[Theorem 4.6 and Proposition 2.6]{artico1981some}. 
\begin{thm}\label{thm:Artico}
We assume that $X$ is a
topological space possessing 
an infinite metrizable gauge. 
Let $\kappa$ be a regular cardinal
with $\kappa\in \metrank(X)$. 
Let $A$ be a closed subset of $X$. 
Let $S\in \FFF{\kappa}$, and
$d\in \ult(X; S)$. 
Then 
the following two statements are
equivalent to each other. 
\begin{enumerate}
\item The set 
$A$ is a proximal  subset of $(X, d)$. 
\item 
There exists a $1$-Lipschitz retraction 
$r\colon X\to A$. 
\end{enumerate}
\end{thm}

Using Proposition \ref{prop:1-lipretract} and Theorems \ref{thm:retract} and \ref{thm:Artico}, 
we obtain a characterization of the closedness in a 
space possessing an infinite metrizable gauge. 
\begin{cor}\label{cor:characlosed}
We assume that 
$X$ is a topological space possessing 
an infinite metrizable gauge. 
Let $\kappa$ be a regular cardinal
 with
$\kappa\in \metrank(X)$. 
 Let 
 $A$ be a subset of $X$. 
Then 
the following statements  are equivalent to each other. 
\begin{enumerate}
\item\label{item:closd}
The set $A$ is closed.
\item\label{item:ret}
The set $A$ is a retract of $X$.
\item\label{item:sprox}
For all $S\in \FFF{\kappa}$, 
there exists  $d\in \ult(X; S)$ such that 
$A$ is a proximal subset of  $(X, d)$.
\item\label{item:alls1lp}
For all $S\in \FFF{\kappa}$, 
there exists  $d\in \ult(X;  S)$ such that 
$A$ is a $1$-Lipschitz retraction of 
$(X, d)$.

\end{enumerate}
\end{cor}
\begin{proof}
By Theorem \ref{thm:retract}, 
we obtain the implication 
$(\ref{item:closd})\To (\ref{item:ret})$. 

Since $X$ is Hausdorff, 
we obtain the implication 
$(\ref{item:ret})\To (\ref{item:closd})$. 

By Proposition \ref{prop:1-lipretract}, 
we obtain the implication 
$(\ref{item:ret})\To (\ref{item:alls1lp})$. 

Theorem \ref{thm:Artico} means the equivalence 
$(\ref{item:sprox})\iff (\ref{item:alls1lp})$. 

The condition (\ref{item:alls1lp}) literary 
implies the condition (\ref{item:ret}). 
Thus we have 
the implication 
$(\ref{item:alls1lp})\To (\ref{item:ret})$. 
\end{proof}


\section{Extensors of ultrametrics and metrics of high power}\label{sec:extensors}

\subsection{Proof of Theorem \ref{thm:extensor}}\label{subsec:proofof2}

We now  introduce a metric vanishing on a given closed subset. 
The construction is an analogue of   \cite{Ha1930}. 
\begin{df}\label{df:subsubmap}
Let $X$ be a topological space 
with $\metrank(X)\neq \emptyset$. 
Let $A$ be a closed subset of $X$. 
Let $G$ be a linearly ordered Abelian group with 
$\met(X; G)\neq \emptyset$. 
Let  $S$ be a Dedekind complete g-characteristic subset of $G$. 
Let $h\in \ult(X, S)$. 
We define a symmetric function 
$\subsubmap[S, h, A]\colon X\times X\to S$ by 
\[
\subsubmap[S, h, A](x, y)=h(x, y)\land (\setdis_{h, A}(x)\lor \setdis_{h, A}(y)). 
\]
Note that since $S$ is Dedekind complete, 
the function $\setdis_{h, A}$ always exists. 
\end{df}

\begin{prop}\label{prop:vanishingmetric}
We assume that $X$ is a topological space 
possessing an infinite 
metrizable gauge. 
Let $\kappa$ be a regular cardinal
with $\kappa\in \metrank(X)$. 
Let $A$ be a closed subset of $X$. 
Let $G\in \GGG{\kappa}$. 
Let  $S$ be a Dedekind complete g-characteristic subset of $G$. 
Let $h\in \ult(X, S)$. 
Then
the following properties  are satisfied: 
\begin{enumerate}
\item\label{item:pro1}
the map $\subsubmap[S, h, A]$ is an $S$-pseudo-ultrametric on $X$, namely, 
it satisfies the condition \emph{(U\ref{item:u1})}--\emph{(U\ref{item:u4})} in 
Definition \ref{df:sultrametrics}; 
\item\label{item:pro2}
for all $x, y\in A$, we have 
$\subsubmap[S, h, A](x, y)=\mzero_{S}(=0_{G})$; 
\item\label{item:pro3}
the restriction 
$\subsubmap[S, h, A]|_{(X\setminus A)\times(X\setminus A)}$ is 
an $S$-ultrametric on the set $X\setminus A$,  and 
it generates the same topology of $X\setminus A$.
\end{enumerate}
\end{prop}
\begin{proof}
We first prove the statement (\ref{item:pro1}). 
The proof of the triangle inequality is similar to \cite{Ha1930}. 
By the definition of 
$\setdis_{h, A}$, 
for all $x, y, z\in X$, 
we obtain the following:
\begin{enumerate}
\renewcommand{\labelenumi}{(\alph{enumi})}
\item 
$h(x, y)\le h(x, z)\lor h(z, y)$
\item 
$\setdis_{h, A}(x)\lor 
\setdis_{h, A}(y)\le 
(\setdis_{h, A}(x)\lor \setdis_{h, A}(z))\lor h(z, y)$
\item 
$\setdis_{h, A}(x)\lor \setdis_{h, A}(y)\le 
h(x, z)\lor(\setdis_{h, A}(z)\lor \setdis_{h, A}(y))$
\item 
$\setdis_{h, A}(x)\lor \setdis_{h, A}(y)\le 
(\setdis_{h, A}(x)\lor \setdis_{h, A}(z))
\lor (\setdis_{h, A}(z)\lor \setdis_{h, A}(y))$
\end{enumerate}
These inequalities imply
that 
$\subsubmap[S, h, A]$ satisfies 
the triangle inequality. 
The remaining part follows from 
the definition of $\subsubmap[S, h, A]$. 

By the definitions of $\setdis_{h, A}$ 
and $\subsubmap[S, h, A]$, 
the statement (\ref{item:pro2}) is true. 

We now show that the statement (\ref{item:pro3}). 
Put $\subsubmap=\subsubmap[S, h, A]$. 
By the definition of $\subsubmap$, 
we verify that $\subsubmap$ is an 
$S$-ultrametric on $X\setminus A$. 
It suffices to show that 
$\subsubmap|_{(X\setminus A)\times (X\setminus A)}$
and $h|_{(X\setminus A)\times (X\setminus A)}$
generates the same topology. 

Take arbitrary $a\in X\setminus A$, 
and 
$\epsilon\in G_{>0}$. 
 By $G\in \GGG{\kappa}$, 
 the set $G_{>0}$ does not have the least element, 
 and hence 
 we  can take
 $\delta\in G_{>0}$
 such that 
$\delta<\min\{\epsilon, \setdis_{h, A}(a)\}$. 
Take $x \in U(a, \delta; h)$. 
Take arbitrary $p\in A$. 
Since $\delta<\setdis_{h, A}(a)$, 
we have $\delta<h(a, p)$. 
By $h(x, a)<\delta$, 
we have $h(x, a)<h(a, p)$. 
According to 
Lemma \ref{lem:isosceles}, 
we obtain $h(x, p)=h(a, p)$. 
Thus, we have $\setdis_{h, A}(a)=\setdis_{h, A}(x)$. 
Since $h(x, a)<\delta$
and $\delta<\setdis_{h, A}(a)$, 
we obtain  $\subsubmap(a, x)=h(x, a)$. 
By $\delta<\epsilon$, 
we conclude that 
$U(a,\delta; h)\mysub U(a, \epsilon; \subsubmap)$. 

We next prove the converse inclusion. 
 Take arbitrary $a\in X\setminus A$ and $\epsilon\in G_{>0}$. 
 By $G\in \GGG{\kappa}$, 
 the set $G_{>0}$ does not have the least element, 
 and hence 
 we  can take $\delta\in G_{>0}$
 with  $\delta<\min\{\epsilon, \setdis_{h, A}(a)\}$. 
Take
  $x\in B(a, \delta; \subsubmap)$. Then, 
 we have 
 \begin{align}\label{al:4444}
 \subsubmap(a, x)
 =
 h(a, x)\land (\setdis_{h, A}(a)\lor \setdis_{h, A}(x))
 <\delta. 
 \end{align}
 For the sake of  contradiction, 
suppose that 
$\subsubmap(a, x)=
 (\setdis_{h, A}(a)\lor \setdis_{h, A}(x))$. 
 Then,  we have 
 $(\setdis_{h, A}(a)\lor \setdis_{h, A}(x))<\delta$, 
 and hence 
 $ \setdis_{h, A}(a)<\delta$. 
 This contradicts $\delta<\setdis_{h, A}(a)$ 
 (the definition of $\delta$). 
 Thus, we have 
 $ \subsubmap(a, x)=h(a, x)$. 
By  (\ref{al:4444}) and $\delta<\epsilon$, we have 
$h(a, x)<\epsilon$, which  implies that 
$U(a, \delta; \subsubmap)
\mysub U(a, \epsilon; h)$. 
 This completes the proof. 
\end{proof}

\begin{df}\label{df:pulback}
Let $X$, $Y$, and $Z$ be sets. 
Let 
$f\colon X\to Y$  and 
$d\colon Y\times Y\to Z$ be maps. 
Then, we define a map
$f^{*}d\colon X\times X\to Z$ by 
$f^{*}d(x, y)=d(f(x), f(y))$. 
This map is called the 
\emph{pullback of $d$ by 
$f$}. 
\end{df}

\begin{df}\label{df:modifymet}
Let $X$ be a topological space 
with $\metrank(X)\neq \emptyset$. 
Let $A$ be a closed subset of $X$. 
Let $G$ be a linearly ordered Abelian group with 
$\met(X; G)\neq \emptyset$. 
Let  $S$ be a Dedekind complete g-characteristic subset of $G$. 
Let $h\in \ult(X, S)$. 
Let $d\in \met(A; G)$. 
Let $r\colon X\to A$ be a uniformly continuous 
retraction with respect to $h$. 
We define a map 
$\submap[S, h, r, A, d]\colon
X\times X\to G$ by 
\[
\submap[S, h, r, A, d](x, y)=
r^{*}d(x, y)\lor \subsubmap[S, h, A](x, y), 
\] 
where $r^{*}d$ is 
the pullback of $d$ by $r$. 
\end{df}

\begin{lem}\label{lem:phantom}
Let $X$ be a topological space 
with $\metrank(X)\neq \emptyset$. 
Let $A$ be a closed subset of $X$. 
Let $G$ be a linearly ordered Abelian group with 
$\met(X; G)\neq \emptyset$. 
Let  $S$ be a Dedekind complete g-characteristic subset of $G$. 
Let $h\in \ult(X, S)$. 
Let $d\in \met(A; G)$. 
Let $r\colon X\to A$ be a uniformly continuous 
retraction with respect to $h$. 
Then the map 
$\submap[S, h, r, A, d]$ is 
a $G$-metric on $X$. 
Moreover, 
if $d$ is an $S$-ultrametric, 
then $\submap[S, h, r, A, d]$ is an
$S$-ultrametric. 
\end{lem}
\begin{proof}
Put $\submap=\submap[S, h, r, A, d]$. 
By
the statements (2) and  (3) in  Proposition \ref{prop:vanishingmetric}, 
and by $d\in \met(A; G)$,  
we observe 
 that $\submap$ satisfies the condition 
(M\ref{item:m0}) in Definition \ref{df:gmetrics}. 
By the statement (1) in  Proposition \ref{prop:vanishingmetric}, 
the map $\submap$ satisfies 
the condition (M\ref{item:m1})--(M\ref{item:m3}). 
From the fact that 
$(x+y)\lor (u+v)\le (x\lor u)+(y\lor v)$ for all 
$x, y, u, v\in G$, 
we deduce 
the triangle inequality 
(the condition (M\ref{item:u4})) for $\submap$. 
The latter part can be proven in a similar way. 
\end{proof}

The following lemma is proven using the ideas of 
Proof of (1-ii) in  the proof of 
\cite[Theorem 1]{NN1981}. 
\begin{lem}\label{lem:metmet}
We assume that 
$X$ is a topological space 
possessing 
an 
infinite metrizable gauge. 
Let $\kappa$ be a regular cardinal
with $\kappa\in \metrank(X)$. 
Let $A$ be a closed subset of $X$. 
Let $G\in \GGG{\kappa}$. 
Let  $S$ be a Dedekind complete g-characteristic subset of $G$. 
Let $h\in \ult(X, S)$. 
Let $d\in \met(A; G)$. 
Then
the $G$-metric 
$\submap[S, h, r,  A, d]$ generates 
the same topology of $X$. 
\end{lem}
\begin{proof}
Put $\submap=\submap[S, h, r, A, d]$ and 
$\subsubmap=\subsubmap[S, h, A]$. 
Since we have $h\in \ult(X; S)$, 
it suffices to show that $\submap$ and $h$
generate the same topology. 

Take $a\in X$ and $\epsilon\in G_{>0}$. 
If $a\not\in A$, 
then
we can 
take $\delta\in G_{>0}$
with  $\delta<\min\{\epsilon, \setdis_{h, A}(a)\}$. 
Similarly to the proof of Proposition \ref{prop:vanishingmetric}, we have
$\subsubmap(a, x)=h(a, x)$. 
Hence $\submap(a, x)=r^{*}d(a, x)\lor h(a, x)$. 
Since $r$ is continuous, 
there exists a sufficiently small $\eta\in G_{>0}$ 
such that 
$U(a, \eta; h)\mysub U(a, \epsilon; \submap)$. 
If $a\in A$, 
then we have 
$(\setdis_{h, A}(a)\lor \setdis_{h, A}(x))
=\setdis_{h, A}(x)\le h(x, a)$.
Thus, we also have 
$\submap(a, x)= r^{*}d(a, x)\lor \setdis_{h, A}(x)
\le r^{*}d(a, x)\lor h(x, a)$. 
Since $r$ is continuous, there 
exists a sufficiently small 
$\eta\in G_{>0}$ such that 
$U(a, \eta; h)\mysub U(a, \epsilon; \submap)$. 

We next prove the converse inclusion. 
 Take arbitrary $a\in X$ and $\epsilon\in G_{>0}$.

Case 1 ($a\not\in A$):
This case is proven in the same way as 
the proof of Proposition \ref{prop:vanishingmetric}. 
 By $G\in \GGG{\kappa}$, 
 the set $G_{>0}$ does not have the least element, 
 and hence 
 we  can take 
 $\delta\in G_{>0}$
 with  $\delta<\min\{\epsilon, \setdis_{h, A}(a)\}$. 
Take 
  $x\in U(a, \delta; \submap)$. Then, 
 we have 
 \begin{align}\label{al:41}
 \subsubmap(a, x)
 =
 h(a, x)\land (\setdis_{h, A}(a)\lor \setdis_{h, A}(x))
 <\delta. 
 \end{align}
 For the sake of  contradiction, 
 we suppose that 
\[
\subsubmap(a, x)=
 (\setdis_{h, A}(a)\lor \setdis_{h, A}(x)).
 \] 
 Then,  we have 
 $(\setdis_{h, A}(a)\lor \setdis_{h, A}(x))<\delta$. 
Hence $ \setdis_{h, A}(a)<\delta$, 
 which contradicts 
 $\delta<\setdis_{h, A}(a)$. 
 Thus, we obtain  
 $ \subsubmap(a, x)=h(a, x)$. 
By  (\ref{al:41}) and $\delta<\epsilon$, we have 
$h(a, x)<\epsilon$. 
This   implies the inclusion  
$U(a, \delta; \submap)
\mysub 
U(a, \epsilon; h)$.

Case 2  ($a\in A$):
Take  $\delta\in G_{>0}$ satisfying  
that 
\begin{enumerate}
\renewcommand{\labelenumi}{(a-\arabic{enumi})}
\item $\delta<\epsilon$;\label{item:delta1}
\item 
$\delta<\setdis_{h, A}(a)$;\label{item:delta0} 
\item 
for all $y\in X$ with 
$d(a, y)<\delta$, we have 
$h(a, y)<\epsilon$;\label{item:delta2}
\item 
if $h(u, v)<\delta$, then 
$h(r(u), r(v))<\epsilon$.\label{item:delta3}
\end{enumerate}
Since $d$ and $h$ generate the same topology on $A$, 
the condition (a-\ref{item:delta2}) is guaranteed. 
Since $r$ is uniformly continuous with respect to $h$, 
the condition (a-\ref{item:delta3}) is guaranteed. 

Take 
$x\in U(a, \delta; \submap)$. 
We now verify  $h(a, x)<\epsilon$. 
By the definition of $\submap$, 
we have  $e(a, x)<\delta$ and 
$\subsubmap(a, x)<\delta$. 
By the definitions of $e$ and $\subsubmap$, we obtain 
\begin{align}
&d(a, r(x))<\delta, \label{al:arxdelta}\\
&h(a, x)\land (\setdis_{h, A}(a)\lor \setdis_{h, A}(x))<\delta.\label{al:brabra}
\end{align} 

We first  assume that $h(a, x)\le \delta$. 
Then,  by (a-\ref{item:delta1}), 
we have $h(a, x)<\epsilon$. 

We next assume that  $\delta<h(a, x)$. 
By (a-\ref{item:delta0}) and (\ref{al:brabra}), 
we have  $\setdis_{h, A}(x)<\delta<h(a, x)$. 
Then,  there exists $b\in A$  with 
$\setdis_{h, A}(x)\le h(b, x)<\delta$. 
Thus,  we have 
\begin{align}\label{al:hbxepsilon}
h(b, x)<\epsilon.
\end{align} 
From $h(b, x)<\delta$ and (a-\ref{item:delta3}), 
it follows that $h(r(b), r(x))<\epsilon$.
By 
$r(b)=b$, 
we have  
\begin{align}\label{al:hbrxepsilon}
h(b, r(x))<\epsilon. 
\end{align}
By (a-\ref{item:delta2}), and by (\ref{al:arxdelta}), 
 we have 
 \begin{align}\label{al:arx}
 h(a, r(x))<\epsilon
 \end{align}
The inequalities 
(\ref{al:hbxepsilon}), 
(\ref{al:hbrxepsilon}), and 
(\ref{al:arx})
imply 
 \begin{align*}
 h(x, a)&\le h(x, b)\lor h(b, r(x))\lor h(r(x), a)
\\ 
&< \epsilon\lor \epsilon\lor \epsilon
=\epsilon. 
 \end{align*}
 Thus, $h(a, x)<\epsilon$. 
 
Therefore,   we conclude that 
 $U(a, \delta; \submap)\mysub U(a, \epsilon; h)$.
 This completes the proof  that 
 $\submap$ and $h$ generate the same topology of $X$. 
\end{proof}

\begin{df}\label{df:paramap}
Let $X$ be a topological space 
with $\metrank(X)\neq \emptyset$.
Let $A$ be a closed subset of $X$.  
Let $G$ be a linearly ordered Abelian group with 
$\met(X; G)\neq \emptyset$. 
Let  $S$ be a  g-characteristic subset of $G$. 
Let $h\in \ult(X, S)$. 
Let $k\in \ult(A, S)$. 
Let 
$r\colon X\to A$ be a uniformly continuous 
retraction with respect to $h$. 
We define a map 
$\paramap[S, h, r, k]\colon X\times X\to S$
 by
\[
\paramap[S, h, r, k]=r^{*}k \lor h 
\]
Note that $\paramap[S, h, r, k]$ is an 
$S$-ultrametric, 
and generates the same topology of $X$
since $h\in \ult(X; S)$. 
\end{df}

The following lemma is proven using the ideas of 
Proof of (1-vi) in the proof of
\cite[Theorem 1]{NN1981}.

\begin{lem}\label{lem:completeext}
We assume that
$X$ is a topological space 
possessing 
an infinite metrizable gauge. 
Let $\kappa$ be a regular cardinal 
with $\kappa\in \metrank(X)$. 
Let $A$ be a closed subset of $X$. 
Let $G\in \GGG{\kappa}$. 
Let  $S$ be a  g-characteristic subset of $G$. 
Let $d\in \met(A; G)$. 
Let $h\in \ult(X, S)$. 
Let $r\colon X\to A$ be a uniformly continuous 
retraction with respect to $h$. 
Let $k\in \ult(A; S)$ be an $S$-ultrametric on $A$ uniformly equivalent to 
$d$. 
Then the map 
$r$ is 
 uniformly continuous 
with respect to $\paramap[S, h, r, k]$. 
If $d$ and $h$ are complete, 
then so is 
$\submap[S, \paramap[S, h, r, k], r, A, d]$. 
\end{lem}
\begin{proof}
Put $\submap=
\submap[S, \paramap[S, h, r, k], r, A, d]$  and 
$u=\paramap[S, h, r, k]$. 
Note that $u$ is an
$S$-ultrametric. 
Since $r\colon X\to A$ is a retraction, 
for all $x, y\in X$, 
we have 
\begin{align*}
u(r(x), r(y))&=k(r(r(x)), r(r(x)))\lor h(r(x), r(y))\\
&=k(r(x), r(y))\lor h(r(x), r(y))\le 
u(x, y)\lor h(r(x), r(y)). 
\end{align*}
Since  $r$ is uniformly continuous with respect to 
$h$, 
by the definition of $u$, 
 the map $r$ is 
uniformly continuous 
with respect to $u$. 

We next show that 
every Cauchy filter 
$\myfilter{F}$
on 
$(X, \submap)$ 
is also Cauchy in $(X, u)$. 
We take  g-characteristic 
isotone embeddings
$l, m\colon \invs{\kappa}\to G$ such that 
\begin{enumerate}
\renewcommand{\labelenumi}{(a-\arabic{enumi})}
\item\label{item:hikarii}
for all $\alpha< \kappa$, 
if $x, y\in X$ satisfies $u(x, y)\le l(\alpha)$, 
then we have $u(r(x), r(y))\le m(\alpha)$; 
\item\label{item:hikariii}
for all $\alpha<\kappa$,  we have 
$l(\alpha)<m(\alpha)$. 
\end{enumerate}
By the regularity of $\kappa$ and 
by the 
uniform continuity of $r$ with respect to $u$,
we can construct
 g-characteristic isotone embeddings 
$l^{\prime}, m\colon \invs{\kappa}\to G$ satisfying
the condition  (a-\ref{item:hikarii}). 
We define $l\colon \invs{\kappa}\to G$
 by $l(\alpha)=\min\{l^{\prime}(\alpha), m(\alpha+1)\}$ if $\alpha<\kappa$, and by $l(\kappa)=0_{G}$. 
 Then 
 $m$ is a g-characteristic isotone embedding, 
 and  
 the pair $l, m \colon \invs{\kappa}\to G$ 
 satisfies the conditions 
 (a-\ref{item:hikarii}) and 
(a-\ref{item:hikariii}).

For the sake of contradiction, 
we 
suppose 
that 
there exists a
Cauchy 
filter $\myfilter{F}$
on 
$(X, \submap)$, 
which is not Cauchy 
on
$(X, u)$. 
Then, 
we obtain the following fact:
\begin{enumerate}
\renewcommand{\labelenumi}{(Fact \arabic{enumi})}
\setlength{\leftskip}{8mm}
\item\label{item:daijifact}
there exits 
$\epsilon\in G_{>0}$ 
such that 
for all 
$F\in \myfilter{F}$, 
and for all $x, y\in F$, 
we have 
$\epsilon\le u(x, y)$. 
\end{enumerate}
Since 
$l\colon\invs{\kappa}\to G$ is a g-characteristic 
isotone embedding and 
$\myfilter{F}$ is Cauchy on 
$(X, \submap)$, 
we can take a map 
 $\myBmap\colon\kappa\to \myfilter{F}$
 stated in Lemma \ref{lem:senkei}
 associated with the map $l$. 
 
 According to  the condition (\ref{item:senkei2}) in Lemma \ref{lem:senkei}, 
 the following statement  holds true: 
 \begin{enumerate}
 \setcounter{enumi}{1}
 \renewcommand{\labelenumi}{(Fact \arabic{enumi})}
\setlength{\leftskip}{8mm}
 \item\label{item:fact2}
 we have 
$\submap(x, y)\le l(\alpha)$
for all 
$\alpha<\kappa$, 
and for all 
$x, y\in \myBmap(\alpha)$, 
\end{enumerate}
Thus, we obtain 
$\setdis_{u, A}(x)\lor \setdis_{u, A}(y)\le l(\alpha)$.
Then, by (a-\ref{item:hikariii}), 
for each $x\in \myBmap(\alpha)$,  
there exists $a_{x}\in A$
with 
$u(x, a_{x})<m(\alpha)$. 
By (a-\ref{item:hikarii}), we have 
$u(a_{x}, r(x))=u(r(x), r(a_{x}))< m(\alpha)$. 
Then, for all $\alpha<\kappa$, 
and  for all 
$x\in \myBmap(\alpha)$, we have 
$u(r(x), x)\le u(r(x), a_{x})\lor 
u(a_{x}, x)\le 
m(\alpha)\lor m(\alpha)=m(\alpha)$. 
By (a-\ref{item:hikariii}), for all $\alpha<\kappa$, 
and  for all 
$x\in \myBmap(\alpha)$, we obtain 
\begin{align}\label{al:daiji47}
u(r(x), x)\le  m(\alpha).
\end{align}
By $r^{*}d(x, y)\le \submap(x, y)$, 
and by (Fact \ref{item:fact2}), 
we have 
$d(r(x), r(y))\le l(\alpha)$ for all $x, y\in \myBmap(\alpha)$ and for all $\alpha<\kappa$. 
Thus, the pushout 
filter $r_{\sharp}\myfilter{F}$ of 
$\myfilter{F}$ by $r$
 is a Cauchy filter on 
$(A, d)$.

Since $(A, d)$ is complete, 
there 
exits a limit $p\in A$ of $r_{\sharp}\myfilter{F}$. 
Since $d$ and $u$ generate the same topology on $A$, 
the point $p$ is also a limit of $r_{\sharp}\myfilter{F}$ with respect to 
$u$. 
Then, 
there exists $\theta<\kappa$ satisfying  that 
\begin{align}
&m(\theta)<\epsilon;\label{al:lep}\\
&r(\myBmap(\theta))\mysub 
U(p, \epsilon;  u)\label{al:epsilonballu}.  
\end{align}
From (\ref{al:daiji47}), 
(\ref{al:lep}), and 
(\ref{al:epsilonballu}), 
it follows that 
for all $x\in \myBmap(\theta)$, 
we have 
\[
u(p, x)\le u(p, r(x))\lor u(r(x), x)
<
\epsilon\lor m(\alpha)\le \epsilon. 
\]
Namely,  we obtain $\myBmap(\theta)\mysub 
U(p, \epsilon;  u)$. 
This inclusion and the strong triangle inequality 
 imply that 
$u(x, y)<\epsilon$ for all $x, y\in \myBmap(\theta)$. 
This contradicts 
(Fact \ref{item:daijifact}).  
Thus, the filter $\myfilter{F}$ is Cauchy on $(X, u)$. 
This leads to the lemma. 
\end{proof}

\begin{rmk}
The author does not know whether 
the metrics $d$ and $\submap[S, \paramap[S, h, r, k], r, A, d]$ in Lemma \ref{lem:completeext} are
uniformly equivalent to each other or not. 
\end{rmk}



\begin{proof}[Proof of Theorem \ref{thm:extensor}]
We assume that $X$ is a topological space 
possessing an infinite metrizable gauge. 
Let $\kappa$ be a regular cardinal
with with $\kappa\in \metrank(X)$. 
Let $G\in \GGG{\kappa}$. 
Let $S$ be a g-characteristic subset of 
$G$. 
Let $A$ be a closed subset of $X$. 

Remark that
since $S$ is g-characteristic in $G$, 
 the set $\cova_{G}(S)$ is 
 characteristic in the set $\lolo{\Arc{G}}$. 
Fix 
a characteristic isotone embedding
$m\colon \invs{\kappa}\to \lolo{\Arc{G}}$
such that $m(\alpha)\in \cova_{G}(S)$ for all $\alpha\in \invs{\kappa}$. 
For each $\alpha\in \invs{\kappa}$, 
we take $l_{\alpha}\in S$ such that 
$\cova_{G}(l_{\alpha})=m(\alpha)$. 
We define a map 
$l\colon \invs{\kappa}\to S$
by $l(\alpha)=l_{\alpha}$. 
Then 
the map $l$
is a characteristic isotone embedding, and 
the relation 
$\cova_{G}\circ l=m$ holds 
true. 
Let 
$\zeta_{\lolo{\Arc{G}}, m}\colon \lolo{\Arc{G}}\to \invs{\kappa}$ 
be a map
in Definition \ref{df:zeta}
associated with $m$. 
Put 
$\xi=\zeta_{\lolo{\Arc{G}}, m}$. 
We put $T=l(\invs{\kappa})$. 
Note that $T$ is a 
Dedekind complete 
g-characteristic subset 
of $G$ and $T\mysub S$. 

For all $\eta\in G_{>0}$, 
there exists $\alpha\in \invs{\kappa}$
such that $m(\alpha+1)<\cova_{G}(\eta)\le m(\alpha)$. 
In this case, we have 
$\xi\circ \cova_{G}(\eta)=\alpha+1$, and hence 
$l\circ \xi\circ \cova_{G}(\eta)=l(\alpha+1)$. 
By the definition of $l$, 
by 
$\cova_{G}\circ l=m$, 
and 
by 
$m(\alpha+1)<\cova_{G}(\eta)$, 
we have $l\circ \xi\circ \cova_{G}(\eta)<\eta$. 
Namely, the following holds true:
\begin{enumerate}
 \setcounter{enumi}{0}
 \renewcommand{\labelenumi}{(Fact \arabic{enumi})}
\setlength{\leftskip}{8mm}
\item\label{item:daijifact1}
for all $\eta\in G_{>0}$, 
we have $l\circ \xi\circ \cova_{G}(\eta)< \eta$. 
\end{enumerate}

Fix 
 $h\in \ult(X; T)$
 (see Proposition 
\ref{prop:characterization}). 
Using Theorem 
\ref{thm:retract}, 
we take a uniformly continuous 
retraction  $r\colon X\to A$
with respect to $h$. 

For each $d\in \met(A; G)$, 
we 
put $D[d]=\cova_{G}\circ d\in 
\ult(A; \lolo{\Arc{G}})$
(see Lemma \ref{lem:uniformdis}). 
Let 
$k[d]=l \circ \xi\circ D[d](=l \circ \xi\circ\cova_{G}\circ d)$. 
According to  Lemmas  \ref{lem:amenable} and \ref{lem:lll}, 
we observe  that $k[d]\in \ult(X; T)$. 

Put 
$u[d]=\paramap[T, h, r, k[d]](=r^{*}k[e]\lor h)$. 
Then we have $u[d]\in \ult(X; T)$. 
According to Lemma \ref{lem:cpltuni}, 
the metrics $d$ and $D[d]$ are
uniformly equivalent to each other. 
Thus, $D[d]$ satisfies the assumption 
in 
Lemma \ref{lem:completeext}. 
By Lemma \ref{lem:completeext}, 
the map $r$ is uniformly continuous with respect to 
$u[d]$. 

We define
a map $\mainmap\colon \met(A; G)\to \met(X; G)$ by 
\[
\mainmap(d)=\submap[T, u[d], r, A, d].
\] 
We now verify that 
the map $\mainmap$ is a desired one. 
By
Lemma
\ref{lem:metmet}, 
we observe  that 
$\mainmap$ is certainly  a map from $\met(A; G)$
into $\met(X; G)$. 

We now show  that $\mainmap$ satisfies 
(A\ref{item:contia1}). 
Fix $d\in \met(A; G)$ and $\epsilon\in G_{>0}$. 
Take $\eta\in G_{>0}$ with $\eta<\epsilon$. 
Take arbitrary $e\in \myVnbd(d; \eta)$, 
 which means that 
for all $x, y\in A$ we have 
\begin{align}
&d(x, y)< e(x, y)\lor \eta;\label{al:49} \\
&e(x, y)<d(x, y)\lor \eta.\label{al:410}
\end{align}

Since $\cova_{G}$ is isotone, 
by (\ref{al:49}) and (\ref{al:410}), 
for all $x, y\in A$, we have 
\begin{align}
D[d](x, y)\le D[e](x, y)\lor \cova_{G}(\eta);\label{al:daijial1} \\
D[e](x, y)\le D[d](x, y)\lor \cova_{G}(\eta). \label{al:daijial2}
\end{align}

Since $l$ and $\xi$ are isotone, 
by (Fact \ref{item:daijifact1}), 
and 
by (\ref{al:daijial1}), 
 and (\ref{al:daijial2}), 
for all $x, y\in A$, we obtain
\begin{align}
&k[d](x, y)\le k[e](x, y)\lor \eta;\label{al:413} \\
&k[e](x, y)\le k[d](x, y)\lor \eta.\label{al:414} 
\end{align}
By (\ref{al:413}) and (\ref{al:414}),  
for all $x, y\in X$, we have
\begin{align}
&r^{*}k[d](x, y)\le  r^{*}k[e](x, y)\lor \eta; \\
&r^{*}k[e](x, y)\le  r^{*}k[d](x, y)\lor \eta. 
\end{align}

These inequalities imply that 
for all $x, y\in X$, we have 
\begin{align}
&u[d](x, y)\le u[e](x, y)\lor \eta;\label{al:415}\\
&u[e](x, y)\le u[d](x, y)\lor \eta.\label{al:416} 
\end{align}

By (\ref{al:415}) and (\ref{al:416}), for all $x\in X$, we obtain 
\begin{align}
&\setdis_{u[d], A}(x)\le \setdis_{u[e], A}(x)\lor \eta; \\
&\setdis_{u[e], A}(x)\le \setdis_{u[d], A}(x)\lor \eta.  
 \end{align}
Thus,  for all $x, y\in X$,  the following hold true:
 \begin{align}
 &\subsubmap[u[d], A](x, y)\le \subsubmap[u[e], A](x, y)\lor \eta; \\
 &\subsubmap[u[e], A](x, y)\le 
 \subsubmap[u[d], A](x, y)\lor \eta. 
\end{align}
Therefore, by $\eta<\epsilon$, 
we conclude that
 $\mainmap(e)\in \myVnbd(\mainmap(d);\epsilon)$, 
 and hence 
 the map $\mainmap$ is continuous. 
 This implies the condition (A\ref{item:contia1}). 

By the definition of $\mainmap$, 
we see that $\mainmap$ satisfies 
(A\ref{item:a2}). 

The latter part of 
Lemma \ref{lem:phantom} implies 
the conditions (A\ref{item:a4})
and (A\ref{item:a45})
(note that $T\mysub S$).

Note that 
$D[d\lor e]=D[d]\lor D[e]$
and 
$k[d\lor e]=k[d]\lor k[e]$. 
By 
$k[d\lor e]=k[d]\lor k[e]$, 
we have 
$\setdis_{u[d]}\lor \setdis_{u[e]}=\setdis_{u[d\lor e]}$. 
Then 
$\subsubmap_{k[d\lor e]}=
\subsubmap_{k[d]}\lor \subsubmap_{k[e]}$, 
and hence
$\mainmap(d\lor e)=\mainmap(d)\lor \mainmap(e)$. 
Thus, the condition (A\ref{item:a5}) is satisfied. 
The condition (A\ref{item:a5}) implies the condition 
(A\ref{item:a3}). 

To show the condition (A\ref{item:a6}), 
we assume that
there exists a complete $G$-metric $D$
in $\met(X; G)$. 
By Lemma \ref{lem:completecomplete}, 
there exists a complete
$\invs{\kappa}$-ultrametric 
$H\in \ult(X; \invs{\kappa})$. 
Put $h'=l\circ H$. 
By Lemmas 
\ref{lem:amenable}
and \ref{lem:cpltuni}, 
we have $h'\in \ult(X; T)$ and 
 $h'$ is complete. 
 Using Theorem 
 \ref{thm:retract}, 
 we 
 can take a 
 uniformly continuous retraction 
 $r'\colon X\to A$
 with respect to $h'$. 
According to 
Lemmas \ref{lem:uniformdis}
and \ref{lem:cpltuni}, 
 if $d\in \met(A; G)$ is complte, 
 then so is 
 $D[d]$. 
Note that 
$\paramap[S, h', r', k[d]]$ is complete. 

Using $h'$ and $r'$ instead of $h$ and $r$ in the argument 
discussed above, 
by Lemma \ref{lem:completeext} ($D[d]$ and $d$ are uniformly equivalent to each other), 
we obtain a map 
$\mainmap\colon \met(A; G)\to \met(X; G)$ 
satisfying  that 
if  $d\in \met(X; G)$ is complete,  
then 
so is the $G$-metrics  $\mainmap(d)$. 
Namely, the map $\mainmap$ satisfies 
the condition (A\ref{item:a6}). 

This finishes  the proof of 
Theorem \ref{thm:extensor}. 
\end{proof}

\begin{proof}[Proof of Theorem \ref{thm:extensorult}]
Let $S$ be a characteristic subset of $\rr_{\ge 0}$.  
Let $X$ be an $\rr_{\ge 0}$-ultrametrizable space. 

Put $G=\hahntai{S}$. 
Then the set $S$ can considered as
 a 
g-characteristic subset of $G$
(see Proposition \ref{prop:ee}). 
Let 
$\mainmap: \met(A; G)\to \met(X; G)$
be a map stated in 
Theorem \ref{thm:extensor}. 
Note that 
$\ult(X; S)\mysub \met(X; G)$
(see Lemma \ref{lem:inclusions}). 
Put 
$\mainmapult=\mainmap|_{\ult(X; S)}$. 
By the condition (A\ref{item:a45}) in 
Theorem \ref{thm:extensor}, 
we confirm that 
 $\mainmapult$  is  a 
 map from $\ult(A; S)$  into 
 $\ult(X; S)$. 
Similarly to the proof of  
Theorem \ref{thm:extensor}, 
we observe 
 that $\mainmap$ 
satisfies 
(B\ref{item:b1})--(B\ref{item:b5}). 
In particular,  
from the argument verifying  that 
$e\in\myVnbd(d; \eta)$ implies $\mainmap(e)\in \myVnbd(\mainmap(d); \epsilon)$, 
it follows that 
$\mainmapult$ is an isometric embedding
with respect to 
$\umetdis_{A}^{S}$ and $\umetdis_{X}^{S}$.

Note that, 
to prove the condition (B\ref{item:b5}), 
we need to confirm that 
there exists a complete $S$-ultrametric 
 $D\in\met(X; G)$. 
 From 
 \cite[Proposition 2.17]{Ishiki2021ultra}, 
there exists complete $D\in \ult(X; S)$, 
and hence, 
we have $D\in \met(X; G)$. 
 Thus, the space $X$ satisfies the 
 assumption in Theorem \ref{thm:extensor} for 
 the condition (A\ref{item:a6}). 
 Hence the condition 
 (B\ref{item:b5}) holds true. 
 
This completes the proof of 
Theorem \ref{thm:extensorult}. 
\end{proof}

\subsection{Proof of Theorem \ref{thm:kappa-compact}}
We next show Theorem \ref{thm:kappa-compact}. 

The following lemma is stated in 
\cite[Theorem 1.3]{stevenson1969results}
(see also \cite[Proposition 2.2]{MR682706}. 
The proof is similary  to the case of 
ordinary metric spaces. 
\begin{lem}\label{lem:kappaandseq}
We assume that $X$ 
is a topological space possessing 
an infinite metrizable gauge. 
Let $\kappa$ be a regular cardinal
 with 
$\kappa\in \metrank(X)$. 
Then the following are equivalent to each other:
\begin{enumerate}
\item Every $\kappa$-sequence in $X$ has 
a convergent $\kappa$-subsequence. 
\item Every open cover of $X$ whose cardinal is  $\kappa$ has a subcover
with cardinal $<\kappa$. 
\item The space $X$ is finally $\kappa$-compact. 

\end{enumerate}
\end{lem}
\begin{lem}\label{lem:closedsubandcpt}
We assume that $X$ 
is a topological space possessing 
an infinite metrizable gauge. 
Let $\kappa$ be a regular cardinal
 with 
$\kappa\in \metrank(X)$. 
If $X$ is not finally $\kappa$-compact, 
then 
there exists a closed subset  $A$ of $X$ such that 
$\card(A)=\kappa$ and $A$ is discrete. 
\end{lem}
\begin{proof}
Since $X$ is not finally $\kappa$-compact, 
by Lemma \ref{lem:kappaandseq}, 
there exists a $\kappa$-sequence 
$\{p_{\alpha}\}_{\alpha<\kappa}$ whose all 
$\kappa$-subsequence have no limits. 
Put $A=\{\, p_{\alpha}\mid \alpha<\kappa\, \}$. 
Then,  for all $p\in A$, 
there exists an open set $U_{p}$ 
such that 
$p\in U_{p}$ 
and 
$U_{p}\cap A=\{p\}$. 
Thus, 
the set $A$ is discrete. 
Since all $\kappa$-subsequences of 
$\{p_{\alpha}\}_{\alpha<\kappa}$ have no limits, 
we conclude that 
for each $p\in A$, we have 
$\card(\{\, \alpha<\kappa \mid p_{\alpha}=p\, \})<\kappa$. 
Thus, by the regularity of $\kappa$, 
we see that $\card(A)=\kappa$. 
\end{proof}

\begin{proof}[Proof of Theorem \ref{thm:kappa-compact}]
We assume that $X$ 
is a topological space possessing 
an infinite metrizable gauge. 
Let $\kappa$ be a regular cardinal
 with 
$\kappa\in \metrank(X)$. 

Since $\GGG{\kappa}\neq \emptyset$
and $\FFF{\kappa}\neq \emptyset$, 
we obtain 
the implications 
$(\ref{item:com3})\To (\ref{item:com2})$
and
$(\ref{item:com5})\To (\ref{item:com4})$. 
According to  Lemma \ref{lem:cptandcomplete}, 
we have the implications $(\ref{item:com1})\To (\ref{item:com3})$ 
and $(\ref{item:com1})\To (\ref{item:com5})$. 

To show $(\ref{item:com2})\To (\ref{item:com4})$, 
we assume the statement  (\ref{item:com2}). 
Then all $d\in \ult(X; G_{\ge 0})$ are complete. 
Put
 $S=G_{\ge 0}$. 
According to 
Proposition \ref{prop:chacha}, 
we have  $S\in \FFF{\kappa}$. 
Hence
 the statement  (\ref{item:com4}) holds true. 
Thus, we obtain
the implication $(\ref{item:com2})\To (\ref{item:com4})$. 

To prove Theorem \ref{thm:kappa-compact}, 
we only need to show the implication 
$(\ref{item:com4})\To (\ref{item:com1})$. 
We assume that 
$X$ is not finally $\kappa$-compact. 
Then, 
by Lemma \ref{lem:closedsubandcpt}, 
there exists a closed subset $A$ of $X$ 
such that 
$\card(A)=\kappa$ and $A$ is discrete.   
Take a bijection 
$f\colon A\to \kappa$, 
and put $g=f^{-1}$. 
Let $e=f^{*}\orddis_{\invs{\kappa}}$, 
where $\orddis_{\invs{\kappa}}$ is the ultrametric on 
$\invs{\kappa}$ in 
Definition  \ref{df:kappasp}. 
Then $e\in \ult(A; \invs{\kappa})$. 

Note that 
$\kappa\mysub \invs{\kappa}$. 
Put $G=\hahntai{\invs{\kappa}}$.
The set  $\invs{\kappa}$ can be considered as  a 
g-characteristic subset of $G$ 
(see Proposition 
\ref{prop:ee}), 
and hence 
$e\in \ult(X; G_{\ge 0})$. 
We take 
a map 
$\mainmap\colon \met(A; G)\to \met(X; G)$
satisfying the conditions in
 Theorem \ref{thm:extensor}. 
Put 
$d=\mainmap(e)$. 
Then, 
by condition 
(A4) in Theorem \ref{thm:extensor}, 
we obtain 
$d\in \ult(X; G_{\ge 0})$ 
with 
$d|_{A^{2}}=e$. 
Since 
$\{\alpha\}_{\alpha<\kappa}$ 
is 
Cauchy in 
$(\invs{\kappa}, \orddis_{\invs{\kappa}})$ 
(see 
(\ref{item:kpkp3})
in 
Corollary \ref{cor:kappaspsp}), 
 the sequence 
$\{g(\alpha)\}_{\alpha<\kappa}$ is 
Cauchy in $(X, d)$. 
Since 
$A$ is closed, 
the sequence 
$\{g(\alpha)\}_{\alpha<\kappa}$ 
has no limits. 
 Hence 
 $(X, d)$ 
 is not complete. 
 Therefore we obtain 
the implication  
$(\ref{item:com4})\To (\ref{item:com1})$. 
 This finishes the proof of  
 Theorem \ref{thm:kappa-compact}. 
\end{proof}

\section{Table of symbols}\label{sec:tableofsymbols}

\setlongtables
\begin{tabularx}{\textwidth}{|l|X|}

\endhead
\hline 

Symbol 
&
Description
\\
\hline
\hline 

$\met(X; G)$
& 
The set of all $G$-metrics that generate the same topology on a topological space $X$.
\\
\hline

$\mzero_{S}$ 
& 
The minimal element of a bottomed inearly ordered set $S$. 
\\
\hline

$\lolo{L}$
&
The one-point order extension of 
a linearly ordered set $L$. 
This is a ordered set by putting the minimal element 
$\mzero_{\lolo{G}}$ to $L$ (Definition \ref{df:oneptextorder}). 
\\
\hline

$\chara(S)$
&
The character of $S$ (Definition \ref{df:character}). 
\\
\hline

$x\ll y$
& 
For all $n\in \zz_{\ge 1}$, we have $n\cdot x<y$. 
\\
\hline 

$x\arel y$
&
$x$
and $y$ are Archimedean equivalent to each other. 
Namely, there exist $n, n\in \zz_{\ge 1}$ such that 
$y\le n\cdot x$ and $x\le m\cdot y$. 
\\
\hline
$[x]_{\arel}$
&
The Archimedean equivalence class of $x$. 
\\
\hline 
$\Arc{G}$
&
The set of all Archimedean classes of 
a linearly ordered Abelian group $G$. 
\\
\hline 

$\alpha\arcle\beta$
&
$x\ll y$ or $x\arel y$, where 
$\alpha=[x]_{\arel}$ and $\beta=[y]_{\arel}$. 
\\
\hline

$\alpha\arclele \beta$
& 
$x\ll y$, where 
$\alpha=[x]_{\arel}$ and $\beta=[y]_{\arel}$. 
We use the same symbol to the order $\ll$ on $G$. 
\\
\hline

$\metrank(X)$ 
& 
The set (class) of all metrizable gauges of 
a topological space $X$. 
\\
\hline 

$\myomega$
&
The least infinite  cardinal. 
$\myomega=\{0, 1, 2, \dots, \}$
\\
\hline

$\GGG{\kappa}$
& 
The class of all linearly ordered Abelian group 
$G$ with 
$\chara(\Arc{G})=\kappa$. 
\\
\hline 
$\FFF{\kappa}$
& 
The class of all bottomed linearly ordered set 
$S$ with $\chara(S)=\kappa$. 
\\
\hline

$\myVnbd(d; \epsilon)$
& 
The 
set of all $e\in \met(X; d)$ such that 
for all $x, y\in X$ we have $e(x, y)<d(x, y)\lor \epsilon$ and $d(x, y)<e(x, y)\lor \epsilon$. 
\\
\hline

$\umetdis_{X}^{S}$ 
&
 The infimum of 
 $\epsilon\in S\sqcup \{\infty\}$ such that 
 for all $x, y\in X$ we have 
$d(x, y)\le e(x, y)\lor \epsilon$, 
 and 
$e(x, y)\le d(x, y)\lor \epsilon$. 
\\
\hline 

$\oposi{S}$
&
The ordered set  with the dual order to $(S, \le_{S})$. 
\\
\hline 

$\invs{\kappa}$

&
$\oposi{\kappa+1}$ (Definition \ref{df:opkappa}). 
\\
\hline

$\comp{S}$
&
The Dedekind completion of 
$(S, \le)$. 
\\
\hline

$\abs(x)$
& 
The absolute value of $x$ of $G$
(Definition \ref{df:lambda}). 
\\
\hline

$\cova_{G}(x)$
&
The natural valuation (or covaluation) of 
$G$ (Definition \ref{df:lambda}). 
\\
\hline

$\hahnsp{L}$
&
The Hahn group induced from a linearly ordered set $L$. 
\\
\hline 

$\hahnkor{G}$
&
The Hahn field induced from a linearly ordered Abelian group $G$. 
\\
\hline

$\stars{S}$
&
$\stars{S}=S\setminus \{\mzero_{S}\}$
(Definition \ref{df:stars}). 
\\
\hline 

$\hahntai{S}$
&
$\hahntai{S}=\hahnkor{\hahnsp{\oposi{\stars{S}}}}$ 
(Definition \ref{df:hahnp}). 
\\
\hline 

$\zeta_{S, l}$
& 
The map defined in Definition \ref{df:zeta} 
induced from 
$S\in \FFF{\kappa}$ 
and a characteristic 
isotone embedding 
 $l\colon \invs{\kappa}\to S$. 
\\
\hline 

$U(x, \epsilon; d)$
&
The open ball centered at $x$ with radius $\epsilon$
by a $G$-metric or $S$-ultrametric  $d$. 
\\
\hline 

$B(x, \epsilon; d)$
&
The closed ball centered at $x$ with radius $\epsilon$
by a $G$-metric or $S$-ultrametric  $d$. 
\\
\hline

$\orddis_{S}$
& 
The  $S$-ultrametric on $S$ defined in 
Definition \ref{df:kappasp}. 
\\
\hline

$\setdis_{d, A}(x)$
&
The distance between the set $A$ and the point $x$
(Definition \ref{df:setdis}). 
\\
\hline

$\subsubmap[S, h, A]$
&
The $S$-pseudo-ultrametric defined in Definition \ref{df:subsubmap}. 
\\
\hline

$\submap[S, h, r, A, d]$
&
The metric defined in Definition \ref{df:modifymet}. 
\\
\hline

$f^{*}d$
&
The pullback induced from $d$ and a map $f$ 
(Definition \ref{df:pulback}). 
\\
\hline

$\paramap[S, h, r, k]$
&
The metric defined in Definition \ref{df:paramap}.
\\
\hline

$f_{\sharp}\myfilter{F}$
&
The pushout filter  of $\myfilter{F}$ by a map $f$. 
\\
\hline

\end{tabularx}

\bibliographystyle{amsplain}
\bibliography{bibtex/ultra.bib}

\providecommand{\bysame}{\leavevmode\hbox to3em{\hrulefill}\thinspace}
\providecommand{\MR}{\relax\ifhmode\unskip\space\fi MR }
\providecommand{\MRhref}[2]{%
  \href{http://www.ams.org/mathscinet-getitem?mr=#1}{#2}
}
\providecommand{\href}[2]{#2}
\begin{thebibliography}{10}

\bibitem{abels2012topological}
H.~Abels and A.~Manoussos, \emph{Topological generators of abelian lie groups
  and hypercyclic finitely generated abelian semigroups of matrices}, Adv.
  Math. \textbf{229} (2012), no.~3, 1862--1872. \MR{2871159}

\bibitem{MR1419403}
G.~Artico, U.~Marconi, and J.~Pelant, \emph{On supercomplete
  {$\omega_\mu$}-metric spaces}, Bull. Polish Acad. Sci. Math. \textbf{44}
  (1996), no.~3, 299--310. \MR{1419403}

\bibitem{artico1981some}
G.~Artico and R.~Moresco, \emph{Some results on $\omega_{\mu}$-metric spaces},
  Annali di Matematica Pura ed Applicata \textbf{128} (1981), no.~1, 241--252.
  \MR{640785}

\bibitem{MR682706}
\bysame, \emph{{$\omega _{\mu }$}-additive topological spaces}, Rend. Sem. Mat.
  Univ. Padova \textbf{67} (1982), 131--141. \MR{682706}

\bibitem{brodskiy2007dimension}
N.~Brodskiy, J.~Dydak, J.~Higes, and A.~Mitra, \emph{Dimension zero at all
  scales}, Topology Appl. \textbf{154} (2007), no.~14, 2729--2740. \MR{2340955}

\bibitem{MR314012}
K.~A. Broughan, \emph{A metric characterizing \v{C}ech dimension zero}, Proc.
  Amer. Math. Soc. \textbf{39} (1973), 437--440. \MR{314012}

\bibitem{MR67882}
A.~H. Clifford, \emph{Note on {H}ahn's theorem on ordered abelian groups},
  Proc. Amer. Math. Soc. \textbf{5} (1954), 860--863. \MR{67882}

\bibitem{MR3946544}
A.~B. Comicheo, \emph{Generalized open mapping theorem for {$X$}-normed
  spaces}, $p$-Adic Numbers Ultrametric Anal. Appl. \textbf{11} (2019), no.~2,
  135--150. \MR{3946544}

\bibitem{dancis1993each}
J.~Dancis, \emph{Each closed subset of metric space {$X$} with {$\mathrm{Ind}\
  X= 0$} is a retract}, Houston J. Math \textbf{19} (1993), no.~4, 541--550.
  \MR{1251609}

\bibitem{MR2435142}
C.~Delhomm\'{e}, C.~Laflamme, M.~Pouzet, and N.~Sauer, \emph{Indivisible
  ultrametric spaces}, Topology Appl. \textbf{155} (2008), no.~14, 1462--1478.
  \MR{2435142}

\bibitem{MR3542043}
A.~Di~Concilio and C.~Guadagni, \emph{Uniform continuity in
  {$\omega_\mu$}-metric spaces and uc {$\omega_\mu$}-metric extendability},
  Acta Math. Hungar. \textbf{150} (2016), no.~1, 153--166. \MR{3542043}

\bibitem{MR3721339}
\bysame, \emph{Hypertopologies on {$\omega_\mu$}-metric spaces}, Filomat
  \textbf{31} (2017), no.~13, 4063--4068. \MR{3721339}

\bibitem{MR2854677}
D.~Dordovskyi, O.~Dovgoshey, and E.~Petrov, \emph{Diameter and diametrical
  pairs of points in ultrametric spaces}, $p$-Adic Numbers Ultrametric Anal.
  Appl. \textbf{3} (2011), no.~4, 253--262. \MR{2854677}

\bibitem{MR3135687}
A.~A. Dovgoshe\u{\i} and E.~A. Petrov, \emph{A subdominant pseudoultrametric on
  graphs}, Mat. Sb. \textbf{204} (2013), no.~8, 51--72. \MR{3135687}

\bibitem{MR4052988}
O.~Dovgoshey, \emph{On ultrametric-preserving functions}, Math. Slovaca
  \textbf{70} (2020), no.~1, 173--182. \MR{4052988}

\bibitem{MR3090172}
O.~Dovgoshey, O.~Martio, and M.~Vuorinen, \emph{Metrization of weighted
  graphs}, Ann. Comb. \textbf{17} (2013), no.~3, 455--476. \MR{3090172}

\bibitem{MR4335845}
O.~Dovgoshey and V.~Shcherbak, \emph{The range of ultrametrics, compactness,
  and separability}, Topology Appl. \textbf{305} (2022), Paper No. 107899, 19
  pages. \MR{4335845}

\bibitem{MR113217}
J.~Dugundji, \emph{Absolute neighborhood retracts and local connectedness in
  arbitrary metric spaces}, Compos. Math. \textbf{13} (1958), 229--246 (1958).
  \MR{113217}

\bibitem{MR239571}
R.~Engelking, \emph{On closed images of the space of irrationals}, Proc. Amer.
  Math. Soc. \textbf{21} (1969), 583--586. \MR{239571}

\bibitem{MR2183496}
A.~J. Engler and A.~Prestel, \emph{Valued fields}, Springer Monographs in
  Mathematics, Springer-Verlag, Berlin, 2005. \MR{2183496}

\bibitem{Hahn1907}
H.~Hahn, \emph{{\"{U}}ber die nichtarchimedischen {G}r\"{o}{\ss}ensysteme},
  Sitz. ber. K. Akad. der Math. Nat. Kl. IIa \textbf{116} (1907), 601--655,
  (Hans Hahn Gresammelte Abhandlungen Band 1, Hans Hahn Collected Works Vol. 1,
  pp. 445--499, Springer, 1995).

\bibitem{Ha1930}
F.~Hausdorff, \emph{Erweiterung einer hom{\"o}omorphie}, Fund. Math.
  \textbf{16} (1930), 353--360.

\bibitem{MR52045}
M.~Hausner and J.~G. Wendel, \emph{Ordered vector spaces}, Proc. Amer. Math.
  Soc. \textbf{3} (1952), 977--982. \MR{52045}

\bibitem{hayes1973uniformities}
A.~Hayes, \emph{Uniformities with totally ordered bases have paracompact
  topologies}, Proc. Cambridge Philos. Soc. \textbf{74} (1973), 67--68.
  \MR{315664}

\bibitem{MR314008}
H.~H. Hung, \emph{The amalgamation property for {$G$}-metric spaces}, Proc.
  Amer. Math. Soc. \textbf{37} (1973), 53--58. \MR{314008}

\bibitem{Ishiki2021ultra}
Y.~Ishiki, \emph{An embedding, an extension, and an interpolation of
  ultrametrics}, $p$-Adic Numbers Ultrametric Anal. Appl. \textbf{13} (2021),
  no.~2, 117--147.

\bibitem{ishiki2021dense}
\bysame, \emph{On dense subsets in spaces of metrics},  (2021), preprint
  arXiv:2104.12450, to apper in Colloq. Math.

\bibitem{MR1940513}
T.~Jech, \emph{Set theory}, Springer Monographs in Mathematics,
  Springer-Verlag, Berlin, 2003, The third millennium edition, revised and
  expanded. \MR{1940513}

\bibitem{MR195052}
I.~Juh\'{a}sz, \emph{Untersuchungen \"{u}ber {$\omega _{\mu }$}-metrisierbare
  {R}\"{a}ume}, Ann. Univ. Sci. Budapest. E\"{o}tv\"{o}s Sect. Math. \textbf{8}
  (1965), 129--145. \MR{195052}

\bibitem{MR0370454}
J.~L. Kelley, \emph{General topology}, Graduate Texts in Mathematics, No. 27,
  Springer-Verlag, New York-Berlin, 1975, Reprint of the 1955 edition [Van
  Nostrand, Toronto, Ont.]. \MR{0370454}

\bibitem{NN1981}
Nguyen~Van Khue and Nguyen~To Nhu, \emph{Two extensors of metrics}, Bull. Acad.
  Polon. Sci. \textbf{29} (1981), 285--291. \MR{640474}

\bibitem{MR3035503}
J.~K\k{a}kol, A.~Kubzdela, and W.~\'{S}liwa, \emph{A non-{A}rchimedean
  {D}ugundji extension theorem}, Czechoslov. Math. J. \textbf{63(138)} (2013),
  no.~1, 157--164. \MR{3035503}

\bibitem{MR89411}
Y.~Kodama, \emph{On {${\rm LC}^n$} metric spaces}, Proc. Japan Acad.
  \textbf{33} (1957), 79--83. \MR{89411}

\bibitem{MR326672}
A.~Kucia and W.~Kulpa, \emph{Spaces having uniformities with linearly ordered
  bases}, Uniw. \'{S}l\k{a}ski w Katowicach---Prace Mat. \textbf{3} (1973),
  45--50. \MR{326672}

\bibitem{MR1760173}
S.~Kuhlmann, \emph{Ordered exponential fields}, Fields Institute Monographs,
  vol.~12, American Mathematical Society, Providence, RI, 2000. \MR{1760173}

\bibitem{MR1529282}
E.~Michael, \emph{Selected {S}election {T}heorems}, Amer. Math. Monthly
  \textbf{63} (1956), no.~4, 233--238. \MR{1529282}

\bibitem{NT1928}
V.~Niemytzki and A.~Tychonoff, \emph{Beweis des satzes, dass ein metrisierbarer
  raum dann und nur dann kompakt ist, wenn er in jeder metrik vollst\"andig
  ist}, Fund. Math. \textbf{12} (1928), 118--120.

\bibitem{NO1961}
K.~Nomizu and H.~Ozeki, \emph{The existence of complete {R}iemannian metrics},
  Proc. Amer. Math. Soc. \textbf{12} (1961), 889--891. \MR{133785}

\bibitem{MR440515}
P.~Nyikos and H.~C. Reichel, \emph{On uniform spaces with linearly ordered
  bases. {II}. ({$\omega_{\mu }$}-metric spaces)}, Fund. Math. \textbf{93}
  (1976), no.~1, 1--10. \MR{440515}

\bibitem{MR3156547}
H.~Ochsenius and E.~Olivos, \emph{A comprehensive survey of non-{A}rchimedean
  analysis in {B}anach spaces over fields with an infinite rank valuation},
  Advances in ultrametric analysis, Contemp. Math., vol. 596, Amer. Math. Soc.,
  Providence, RI, 2013, pp.~215--236. \MR{3156547}

\bibitem{MR2598517}
C.~Perez-Garcia and W.~H. Schikhof, \emph{Locally convex spaces over
  non-{A}rchimedean valued fields}, Cambridge Studies in Advanced Mathematics,
  vol. 119, Cambridge University Press, Cambridge, 2010. \MR{2598517}

\bibitem{MR3206769}
P.~Pongsriiam and I.~Termwuttipong, \emph{Remarks on ultrametrics and
  metric-preserving functions}, Abstr. Appl. Anal. (2014), Art. ID 163258, 9.
  \MR{3206769}

\bibitem{MR2379002}
S.~Priess-Crampe, \emph{Generalized {K}eller spaces}, Bull. Belg. Math. Soc.
  Simon Stevin \textbf{14} (2007), no.~5, 979--991. \MR{2379002}

\bibitem{sikorski1950remarks}
R.~Sikorski, \emph{Remarks on some topological spaces of high power}, Fund.
  Math. \textbf{37} (1950), no.~1, 125--136. \MR{40643}

\bibitem{MR370524}
D.~J. Souppouris, \emph{Generalized metric spaces are paracompact}, Math. Proc.
  Cambridge Philos. Soc. \textbf{77} (1975), 325--326. \MR{370524}

\bibitem{MR1387953}
I.~S. Stares and J.~E. Vaughan, \emph{The {D}ugundji extension property can
  fail in {$\omega_\mu$}-metrizable spaces}, Fund. Math. \textbf{150} (1996),
  no.~1, 11--16. \MR{1387953}

\bibitem{stasyuk2009continuous}
I.~Stasyuk and E.~D. Tymchatyn, \emph{A continuous operator extending
  ultrametrics}, Comment. Math. Univ. Carolin. \textbf{50} (2009), no.~1,
  141--151. \MR{2562811}

\bibitem{stevenson1969results}
F.~W. Stevenson and W.~J. Thorn, \emph{Results on $\omega_{\mu}$-metric
  spaces}, Fund. Math. \textbf{65} (1969), no.~3, 317--324. \MR{247602}

\bibitem{tymchatyn2005note}
E.~D. Tymchatyn and M.~Zarichnyi, \emph{A note on operators extending partial
  ultrametrics}, Comment. Math. Univ. Carolin. \textbf{46} (2005), no.~3,
  515--524. \MR{2174529}

\bibitem{vD1975}
E.~K. van Douwen, \emph{Simultaneous extension of continuous functions}, Ph.D.
  thesis, Vrije Universiteit Amsterdam, 1975.

\bibitem{MR166749}
S.~{-T}. Wang, \emph{Remarks of {$\omega _{\mu }$}-additive spaces}, Fund.
  Math. \textbf{55} (1964), 101--112. \MR{166749}

\bibitem{W1970}
S.~Willard, \emph{General topology}, Dover Publications, 2004; originally
  published by the Addison-Wesley Publishing Company in 1970. \MR{2048350}

\end{thebibliography}

\end{document}